\documentclass[11pt]{article}
\usepackage[margin=1in]{geometry}
\usepackage[T1]{fontenc}
\usepackage{lmodern}
\usepackage{microtype}

\usepackage[sc]{mathpazo}

\usepackage{amsfonts}
\usepackage{amsmath}
\usepackage{amssymb}

\usepackage{amsthm}

\newtheorem{definition}{Definition}
\newtheorem{theorem}{Theorem}
\newtheorem{lemma}{Lemma}

\newtheorem{claim}{Claim}

\newtheorem{inftheorem}{Informal Theorem}

\usepackage{hyperref}
\hypersetup{
  colorlinks = true,
  urlcolor = {blue},
  citecolor = {blue}
}

\usepackage{color}

\usepackage{bm}

\usepackage{dsfont}

\usepackage{relsize}

\usepackage{algorithm}
\usepackage[noend]{algpseudocode}
\makeatletter
\def\BState{\State\hskip-\ALG@thistlm}
\makeatother

\usepackage{subcaption}
\usepackage{tikz}
 
\newcommand{\lp}{\left}
\newcommand{\rp}{\right}
\newcommand{\mrm}[1]{\mathrm{#1}}
\newcommand{\mbb}[1]{\mathbb{#1}}

\newcommand{\mcal}[1]{\mathcal{#1}}

\def\abs#1{\left|#1\right|}

\DeclareMathOperator*{\poly}{poly}

\def\R{\mathbb{R}}
\def\reals{\mathbb{R}}
\def\nats{\mathbb{N}}

\newcommand{\me}{\mathrm{e}}

\def\d{\mrm{d}}

\newcommand\norm[1]{\left\| #1 \right\|}
\newcommand\snorm[2]{\left\| #2 \right\|_{#1}}
\newcommand\mat[1]{\bm{#1}}
\newcommand\matr[1]{\bm{#1}}
\renewcommand\vec[1]{\bm{#1}}

\def\Domain{\mcal{D}}

\newcommand{\tr}{\mathrm{tr}}
\newcommand{\1}[1]{\bm{1}_{#1}}

\DeclareMathOperator{\sgn}{sgn}

\newcommand{\normal}{\mcal{N}}

\def\VC{\mrm{VC}}

\def\Prob{\mathbb{P}}

\DeclareMathOperator*{\E}{\mathlarger{\mbb{ E }}}
\DeclareMathOperator*{\Exp}{\mathlarger{\mbb{ E }}}

\def\NS{\bm{\mrm{NS}}}

\DeclareMathOperator*{\Var}{\mrm{Var}}

\newcommand{\dchi}[2]{D_{\chi^2} (#1\|#2) }
\newcommand{\dtv}[2]{d_{\mrm{TV}} (#1, #2) }
\newcommand{\dhel}[2]{d_{\mrm{H}} (#1, #2) }

\newcommand{\wh}[1]{\widehat{#1}}
\newcommand{\wt}[1]{\widetilde{#1}}

\def\eps{\varepsilon}

\def\tind{\psi}

\newcommand{\dist}{\mathrm{dist}}

\newcommand{\diag}{\mathrm{diag}}

\newcommand{\symm}{\mathcal{Q}}

\newcommand{\paragr}[1]{\noindent \textbf{#1}}

\newcommand{\Nt}{\normal^*}
\newcommand{\mt}{\vec{\mu}^*}
\newcommand{\St}{\mat{\Sigma}^*}

\newcommand{\tim}{\tilde{\vec{\mu}}}
\newcommand{\tiS}{\tilde{\mat{\Sigma}}}
\newcommand{\N}{\normal}
\newcommand{\m}{\vec{\mu}}
\renewcommand{\S}{\mat{\Sigma}}

\newcommand{\wb}[1]{\wt{ \boldsymbol {#1}}}

\newif\ifnotes\notestrue

\ifnotes
\usepackage{color}
\definecolor{mygrey}{gray}{0.50}
\newcommand{\notename}[2]{{\textcolor{red}{\footnotesize{\bf (#1:} {#2}{\bf
) }}}}

\else

\newcommand{\notename}[2]{{}}

\fi

\newif\ifhideproofs

\ifhideproofs
\usepackage{environ}
\NewEnviron{hide}{}

\fi

\newenvironment{prevproof}[2]{\noindent {\em {Proof of {#1}~\ref{#2}:}}}{$\hfill\qed$\vskip \belowdisplayskip}
\newenvironment{prevproofbig}[2]{\noindent \subsection*{Proof of {#1}~\ref{#2}\\}}{$\hfill\qed$\vskip \belowdisplayskip}

\begin{document}

\title{Efficient Truncated Statistics with Unknown Truncation}

\author{
  Vasilis Kontonis \\
  \small UW Madison \\
  \small \href{mailto:tzamos@wisc.edu}{kontonis@wisc.edu} \normalsize
  \and Christos Tzamos \\
  \small UW Madison \\
  \small \href{mailto:tzamos@wisc.edu}{tzamos@wisc.edu} \normalsize
  \and Manolis Zampetakis \\
  \small MIT \\
  \small \href{mailto:mzampet@mit.edu}{mzampet@mit.edu} \normalsize
}

\maketitle

\begin{abstract}

  We study the problem of estimating the parameters of a Gaussian
distribution when samples are only shown if they fall in some (unknown)
subset $S \subseteq \R^d$. This core problem in truncated statistics has
long history going back to Galton, Lee, Pearson and Fisher. Recent work
by Daskalakis et al. (FOCS'18), provides the first efficient algorithm
that works for arbitrary sets in high dimension when the set is known,
but leaves as an open problem the more challenging and relevant case of
unknown truncation set.

  Our main result is a computationally and sample efficient algorithm
for estimating the parameters of the Gaussian under arbitrary unknown
truncation sets whose performance decays with a natural measure of
complexity of the set, namely its Gaussian surface area. Notably, this
algorithm works for large families of sets including intersections of
halfspaces, polynomial threshold functions and general convex sets. We
show that our algorithm closely captures the tradeoff between the
complexity of the set and the number of samples needed to learn the
parameters by exhibiting a set with small Gaussian surface area for
which it is information theoretically impossible to learn the true
Gaussian with few samples.
 \end{abstract}

\setcounter{page}{0}
\thispagestyle{empty}
\newpage

\section{Introduction}
\label{sec:intro}

A classical challenge in Statistics is estimation from truncated samples.
Truncation occurs when samples falling outside of some subset $S$ of the
support of the distribution are not observed. Truncation of samples has
myriad manifestations in business, economics, engineering, social
sciences, and all areas of the physical sciences.

Statistical estimation under truncated samples has had a long history in
Statistics, going back to at least the work of Galton \cite{Galton1897}
who analyzed truncated samples corresponding to speeds of American
trotting horses. Following Galton's work, Pearson and Lee
\cite{Pearson1902, PearsonLee1908, Lee1914} used the method of moments in
order to estimate the mean and standard deviation of a truncated
univariate normal distribution and later Fisher \cite{fisher31} used the
maximum likelihood method for the same estimation problem. Since then,
there has been a large volume of research devoted to estimating the
truncated normal distribution; see e.g.
\cite{Schneider86,Cohen91,BalakrishnanCramer}. Nevertheless, the first
algorithm that is provably computationally and statistically efficient was
only recently developed by Daskalakis et al. \cite{DGTZ18}, under the
assumption that the truncation set $S$ is known.

In virtually all these works the question of estimation under unknown
truncation set is raised. Our work resolves this question by providing
tight sample complexity guarantees and an efficient algorithm for
recovering the underlying Gaussian distribution.  Although this estimation
problem has clear and important practical and theoretical motivation too
little was known prior to our work even in the asymptotic regime.  In the
early work of Shah and Jaiswal \cite{ShahJ1966} it was proven that the
method of moments can be used to estimate a single dimensional Gaussian
distribution when the truncation set is unknown but it is assumed to be an
interval. In the other extreme where the set is allowed to be arbitrarily
complex, Daskalakis et al.  \cite{DGTZ18} showed that it is information
theoretically impossible to recover the parameters.  We provide the first
complete analysis of the number of samples needed for recovery taking into
account the complexity of the underlying set.

\paragraph{Our Contributions.}

Our work studies the estimation task when the truncation set belongs in a
family $\mathcal{C}$ of ``low complexity''. We use two different notions for
quantifying the complexity of sets: the VC-dimension and the Gaussian Surface
Area.

Our first result is that for any set family with VC-dimension
$\VC(\mathcal{C})$, the mean and covariance of the true $d$-dimensional
Gaussian Distribution can be recovered up to accuracy $\eps$ using only
$\tilde{O} \left( \frac{\VC(\mathcal{C})}{\eps} + \frac{d^2}{\eps^2} \right)$
truncated samples.

\begin{inftheorem} \label{thm:mainInformal0}
  Let $\mathcal{C}$ be a class of sets with VC-dimension
  $\VC(\mathcal{C})$ and let $N = \tilde{O} \left(
  \frac{\VC(\mathcal{C})}{\eps} + \frac{d^2}{\eps^2} \right)$. Given $N$
  samples from a $d$-dimensional Gaussian $\normal(\vec \mu,\matr \Sigma)$
  with unknown mean $\mu$ and covariance $\matr \Sigma$, truncated on a
  set $S \in \mathcal{C}$ with mass at least $\alpha$, it is possible to
  find an estimate $(\hat{\vec{\mu}}, \hat{\matr{\Sigma}})$ such that
  $\dtv{\normal(\vec\mu,\matr\Sigma)}{\normal(\hat{ \vec{ \mu}},\hat{
  \matr {\Sigma}})} \le \eps$.
\end{inftheorem}

The estimation method computes the set of smallest mass that maximizes the
likelihood of the data observed and learns the truncated distribution
within error $O(\eps)$ in total variation distance. To translate this
error in total variation to parameter distance, we prove a general result
showing that it is impossible to create a set (no matter the complexity)
so that two Gaussians whose parameters are far have similar truncated
distributions (see Lemma~\ref{lem:tvdLowerBound}).

A simple but not successful approach would be to first try to learn an
approximation of the truncation set with symmetric difference roughly
$\eps^2/d^2$ with the true set and then run the algorithm of
\cite{DGTZ18} using the approximate oracle.  This approach would lead to a
$\VC(\mcal{S}) d^2/\eps^2$ sample complexity that is worse than what we
get.  More importantly, doing empirical risk minimization\footnote{That is
  finding a set of the family that contains all the observed samples.}
  using truncated samples does not guarantee that we will find a set of
  small \emph{symmetric} difference with the true and it is not clear how
  one could achieve that.

Our result bounds the sample complexity of identifying the underlying
Gaussian distribution in terms of the VC-dimension of the set but does not
yield a computationally efficient method for recovery.  Obtaining a
computationally efficient algorithm seems unlikely, unless one restricts
attention to simple specific set families, such as axis aligned
rectangles. One would hope that exploiting the fact that samples are drawn
from a ``tame'' distribution, such as a Gaussian, can lead to general
computationally efficient algorithms and even improved sample complexity.

Indeed, our main result is an algorithm that is both computationally and
statistically efficient for estimating the parameters of a spherical
Gaussian and uses only $d^{O( \Gamma^2(\mathcal{C}) )}$ samples, where
$\Gamma(\mathcal{C})$ is the \textit{Gaussian Surface Area} of the class
$\mathcal{C}$, an alternative complexity measure introduced by Klivans et
al.  \cite{KOS08}:

\begin{inftheorem} \label{thm:mainInformal1}
  Let $\mathcal{C}$ be a class of sets with Gaussian surface area at most
  $\Gamma(\mathcal{C})$ and let $k = \poly(1/\alpha,
  1/\eps)\Gamma(\mathcal{C})^2$. Given $N = d^k$ samples from a spherical
  $d$-dimensional Gaussian $\normal(\vec \mu,\sigma^2 \matr I)$, truncated on a set $S \in
  \mathcal{C}$ with mass at least $\alpha$,in time $\poly(m)$, we can find an estimate $\hat \mu,
  \hat \sigma^2$ such that \[
    \dtv{\normal(\vec \mu, \sigma^2 \matr I)}{\normal(\hat{\vec{\mu}},
  \hat \sigma^2 \matr{I})} \le \eps.\]
\end{inftheorem}

The notion of Gaussian surface area can lead to better sample complexity
bounds even when the VC dimension is infinite. An example of such a case
is when $\mathcal{C}$ is the class of all convex sets.
Table~\ref{tab:surface-area} summarizes the known bounds for the Gaussian
surface area of different concept classes and the implied sample
complexity in our setting when combined with our main theorem.

\begin{table*}[ht]
  \def\arraystretch{1.2}
  \centering
  \begin{tabular}{l c c}
    \hline
    \textbf{Concept Class} & \textbf{Gaussian Surface Area} & \textbf{Sample Complexity} \\
    \hline\hline
    Polynomial threshold functions of degree $k$ & $O(k)$ \cite{Kan11} & $d^{O(k^2)}$ \\
    Intersections of $k$ halfspaces & $O(\sqrt{\log{k}})$ \cite{KOS08} & $d^{O(\log k)}$ \\
    General convex sets & $O(d^{1/4})$ \cite{Bal93} & $d^{O(\sqrt{d})}$ \\
    \hline
    \hline
  \end{tabular}
  \caption{\small Summary of known results for Gaussian Surface Area. The last
  column gives the sample complexity we obtain for our setting.}
  \label{tab:surface-area}
\end{table*}

Beyond spherical Gaussians, our main result extends to Gaussians with
arbitrary diagonal covariance matrices. In addition, we provide an
information theoretic result showing that the case with general covariance
matrices can also be estimated using the same sample complexity bound by
finding a Gaussian and a set that matches the moments of the true
distribution.  We remark our main algorithmic result Informal
Theorem~\ref{thm:mainInformal2} uses Gaussian Surface Area whereas our
sample complexity result Informal Theorem~\ref{thm:mainInformal1} uses
VC-dimension.  We discuss the differences of the two approaches in
Section~\ref{sec:vcGsa}.

\begin{inftheorem} \label{thm:mainInformal2}
  Let $\mathcal{C}$ be a class of sets with Gaussian surface area at most
  $\Gamma(\mathcal{C})$ and let $k = \poly(1/\alpha,
  1/\eps)\Gamma(\mathcal{C})^2$.  Any truncated Gaussian with $\normal(\hat{
\vec{\mu}}, \hat{\matr{\Sigma}}, \hat S)$ with $\hat S \in \mathcal{C}$ that
  approximately matches the moments up to degree $k$ of a truncated
  $d$-dimensional Gaussian $\normal(\vec \mu,\matr \Sigma,S)$ with $S \in
  \mathcal{C}$, satisfies $\dtv{\normal(\vec \mu,\matr \Sigma)}{\normal(\hat{
  \vec{\mu}},\hat{\matr{\Sigma}})} \le \eps$. The number of samples to estimate
  the moments within the required accuracy is at most $d^{O(k)}$.
\end{inftheorem}

This shows that the first few moments are sufficient to identify the
parameters. Analyzing the guarantees of moment matching methods is
notoriously challenging as it involves bounding the error of a system of
many polynomial equations. Even for a single-dimensional Gaussian with
truncation in an interval, where closed form solutions of the moments
exist, it is highly non-trivial to bound these errors~\cite{ShahJ1966}. In
contrast, our analysis using Hermite polynomials allows us to easily
obtain bounds for arbitrary truncation sets in high dimensions, even
though no closed form expression for the moments exists.

We conclude by showing that the dependence of our sample complexity bounds
both on the VC-dimension and the Gaussian Surface Area is tight \emph{up
to polynomial factors}.  In particular, we construct a family in $d$
dimensions with VC dimension $2^d$ and Gaussian surface area $O(d)$ for
which it is not possible to learn the mean of the underlying Gaussian
within 1 standard deviation using $o(2^{d/2})$ samples.

\begin{inftheorem}\label{thm:mainInformal3}
  There exists a family of sets $\mcal{S}$ with $\Gamma({\mcal{S}}) =
  O(d)$ and VC-dimension $2^d$ such that any algorithm that draws $N$
  samples from $\normal(\vec \mu, \matr I, S)$ and computes an estimate
  $\wt {\vec \mu}$ with $\norm{\wt {\vec \mu} - \vec \mu}_2 \leq 1$ must
  have $N= \Omega(2^{d/2})$.
\end{inftheorem}

\paragraph{Our techniques and relation to prior work.}
The work of Klivans et al.~\cite{KOS08} provides a computationally and
sample efficient algorithm for learning geometric concepts from labeled
examples drawn from a Gaussian distribution. On the other hand, the recent
work of Daskalakis et al.~\cite{DGTZ18} provides efficient estimators for
truncated statistics with \emph{known} sets. One could hope to combine
these two approaches for our setting, by first learning the set and then
using the algorithm of~\cite{DGTZ18} to learn the parameters of the
Gaussian. This approach, however, fails for two reasons. First, the
results of Klivans et al.~\cite{KOS08} apply in the supervised learning
setting where one has access to both positive and negative samples, while
our problem can be thought of as observing only positive examples (those
falling inside the set). In addition, any direct approach that extends
their result to work with positive only examples requires that the
underlying Gaussian distribution is known in advance.

One of our key technical contributions is to extend the techniques of
Klivans et al. \cite{KOS08} to work with \emph{positive only examples}
from an \emph{unknown} Gaussian distribution, which is the major case of
interest in truncated statistics.  To perform the set estimation Klivans
et al.  \cite{KOS08}, rely on a family of orthogonal polynomials with
respect to the Gaussian distribution, namely the Hermite polynomials and
show that the indicator function of the set is well approximated by its
low degree Hermite expansion.  While we cannot learn this function
directly in our setting, we are able to recover an alternative function,
that contains ``entangled'' information of both the true Gaussian
parameters and the underlying set. After learning the function, we
formulate an optimization problem whose solution enables us to decouple
these two quantities and retrieve both the Gaussian parameters and the
underlying set. We describe our estimation method in more detail in
Section~\ref{sec:algorithm}. As a corollary of our approach, we obtain
the first efficient algorithm for learning geometric concepts from
positive examples drawn from an unknown spherical Gaussian distribution.

\paragraph{Simulations.}
In addition to the theoretical guarantees of our algorithm, we
empirically evaluate its performance using simulated data. We present the
results that we get in Figure~\ref{fig:simulations}, where one can see
that even when the truncation set is complex, our algorithm finds an
accurate estimation of the mean of the untruncated distribution.  Observe
that our algorithm succeeds in estimating the true mean of the input
distribution despite the fact that the set is unknown and the samples
look similar in both cases.
\begin{figure}[t!]
  \centering
  \begin{subfigure}[t]{0.48\textwidth}
    \centering
    \includegraphics[width=\textwidth]{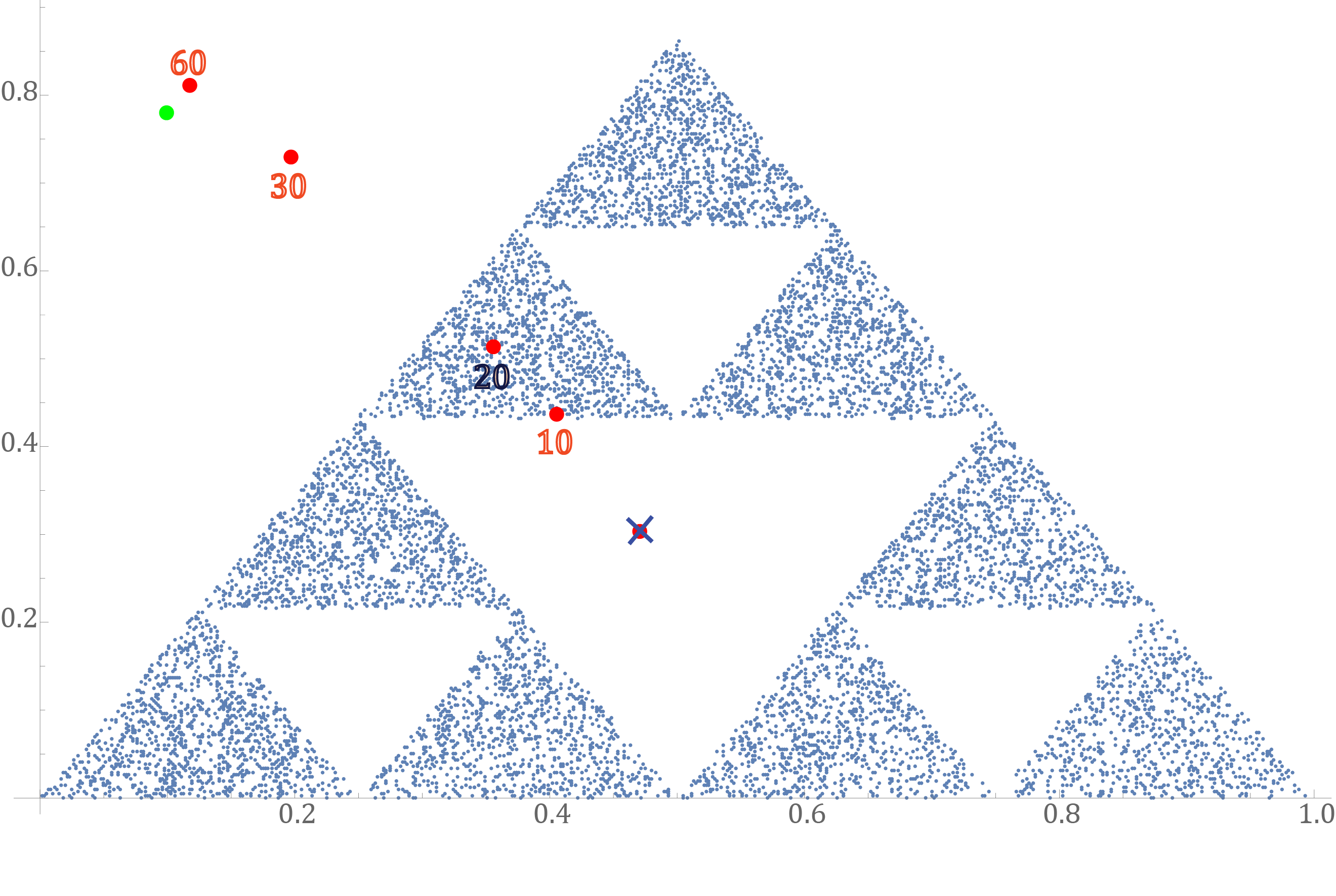}
    \caption{Execution of our algorithm for isotropic Gaussian distribution with $\vec{\mu}^* = (0.1, 0.78)$ and $\vec{\mu}_{S} = (0.48, 0.32)$.}
  \end{subfigure}
  ~~~
  \begin{subfigure}[t]{0.48 \textwidth}
    \centering
    \includegraphics[width=\textwidth]{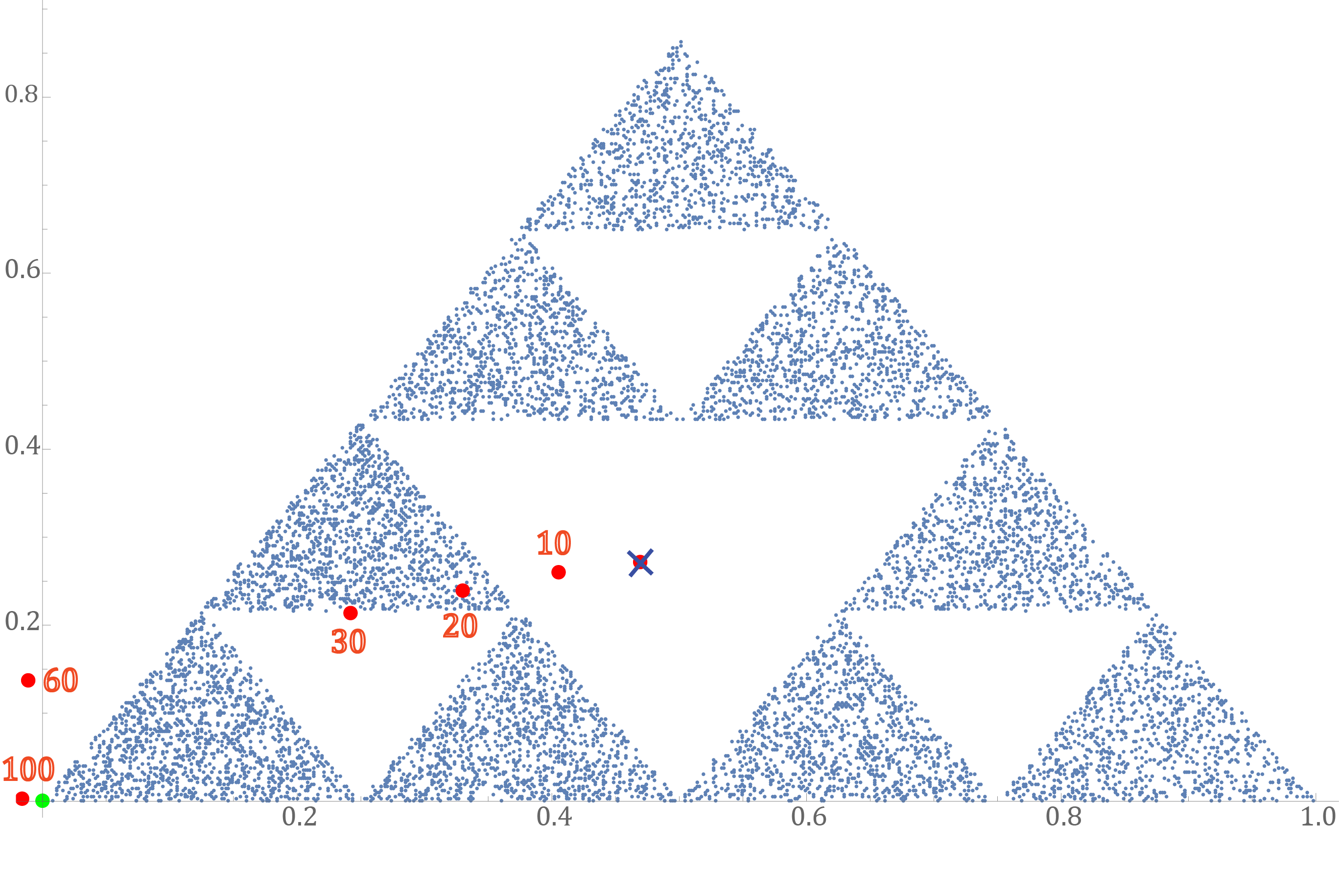}
    \caption{Execution of our algorithm for isotropic Gaussian distribution with $\vec{\mu}^* = (0, 0)$ and $\vec{\mu}_{S} = (0.47, 0.27)$.}
  \end{subfigure}
  \caption{Illustration of the results of our algorithm for an unknown
truncation set. The $\times$ sign corresponds to the conditional mean of
the truncated distribution, while the green point corresponds to the true
mean and the red points correspond to the estimated true mean depending
on the degree of the Hermite polynomials that are being used by the
algorithm.}\label{fig:simulations}
\end{figure}
 \subsection{Further Related Work} \label{sec:relatedWork}

\iffalse
  Early precursors of the truncated statistics literature can be found
in the simpler of estimating a truncated Normal distribution
\cite{Galton1897, Pearson1902, PearsonLee1908, fisher31, Lee1914, Hotelling48,Tukey49}.
All these works make the assumption the set truncation set is a known
axis aligned box. Our work is the first that is able to handle general Gaussian distributions in high
dimensions for arbitrary sets where the complexity of the estimation algorithms
degreases gracefully with the Gaussian surface area of the set.
\fi

  Our work is related to the field of robust statistics as it can
robustly learn a Gaussian even in the presence of an adversary erasing
samples outside a certain set. Recently, there has been a lot of
theoretical work doing robust estimation of the parameters of
multi-variate Gaussian distributions in the presence of arbitrary
corruptions to a small $\varepsilon$ fraction of the samples, allowing
for both deletions of samples and additions of samples that can also be
chosen adaptively
\cite{DKK+16b,CSV17,LRV16,DKK+17,DKK+18}. When the corruption of the data
is so powerful it is easy to see that the estimation error of the
parameter depends on $\eps$ and cannot shrink to $0$ as the number of
samples grows to infinity. In our model the corruption is more
restrictive but in return our results show how to estimate the
parameters of a multi-variate Gaussian distribution
\textit{to arbitrary} accuracy even when the fraction of corruption is
any constant less than $1$.

  Our work also has connections with the literature of learning from
positive examples. At the heart of virtually all of the results in this
literature is the use of the exact knowledge of the original
non-truncated distribution to be able to generate fake negative
examples, e.g. \cite{Den98, LDG00}. When the original distribution is
uniform, better algorithms are known. Diakonikolas et al. \cite{DDS14}
gave efficient learning algorithms for DNFs and linear threshold
functions, Frieze et al. \cite{FJK96} and Anderson et al. \cite{AGR13}
gave efficient learning algorithms for learning $d$-dimensional
simplices. Another line of work proves lower bounds on the sample
complexity of recovering an unknown set from positive examples.
Goyal et al. \cite{GR09} showed that learning a convex set in
$d$-dimensions to accuracy $\eps$ from positive samples, uniformly
distributed inside the set, requires at least $2^{\Omega(\sqrt{d/\eps})}$
samples, while the work of \cite{Eld11} showed that
$2^{\Omega(\sqrt{d})}$ samples are necessary even to estimate the mass
of the set. To the best of our knowledge, no matching upper bounds are
known for those results. Our estimation result implies that
$d^{\poly(\frac 1\eps) \sqrt{d}}$ are sufficient to learn the set and
its mass when given positive samples from a Gaussian truncated on the
convex set.
 \section{Preliminaries} \label{sec:model}

\paragr{Notation.} We use small bold letters $\vec{x}$ to refer to real
vectors in finite dimension $\reals^d$ and capital bold letters
$\matr{A}$ to refer to matrices in $\reals^{d \times \ell}$. Similarly, a
function with image in $\reals^d$ is represented by a small and bold
letter $\vec{f}$. Given a subset $S$ of $\R^d$ we define
$\1{S}(\vec x)$ to be its $0 - 1$ indicator. Let
$\matr{A} \in \reals^{d \times d}$, we define
$\matr{A}^{\flat} \in \reals^{d^2}$ to be the standard vectorization of
$\matr{A}$. Let also $\symm_d$ be the set of all the symmetric
$d \times d$ matrices.
The \emph{Frobenius norm} of a matrix
$\matr{A}$ is defined as $\norm{\matr{A}}_F=\norm{\matr{A}^{\flat}}_2$.
\medskip

\paragr{Gaussian Distribution.} Let $\normal(\vec{\mu}, \matr{\Sigma})$
be the normal distribution with mean $\vec{\mu}$ and covariance matrix
$\matr{\Sigma}$, with the following probability density function
\begin{align} \label{eq:normalDensityFunction}
  \normal(\vec{\mu}, \matr{\Sigma}; \vec{x}) =
  \frac{1}{\sqrt{\det(2 \pi \matr{\Sigma})}}
  \exp \left( - \frac{1}{2} (\vec{x} - \vec{\mu})^T \matr{\Sigma}^{-1}
    (\vec{x} - \vec{\mu}) \right).
\end{align}
Also, let $\normal(\vec{\mu},\matr{\Sigma}; S)$ denote the
\emph{probability mass of a measurable set $S$} under this Gaussian
measure. We shall also denote by $\normal_0$ the standard Gaussian,
whether it is single or multidimensional will be clear from the context.
\medskip

\paragr{Truncated Gaussian Distribution.} Let $S \subseteq \reals^d$ be
a subset of the $d$-dimensional Euclidean space, we define the
\textit{$S$-truncated normal distribution}
$\normal(\vec{\mu}, \matr{\Sigma}, S)$ the normal distribution
$\normal(\vec{\mu}, \matr{\Sigma})$ conditioned on taking values in the
subset $S$. The probability density function of
$\normal(\vec{\mu}, \matr{\Sigma}, S)$ is the following
\begin{equation} \label{eq:truncatedNormalDensityFunction}
  \normal(\vec{\mu}, \matr{\Sigma}, S; \vec{x}) =
\frac{\1{S}(\vec x)}{\normal(\vec \mu, \matr{\Sigma};S)} \normal(\vec{\mu}, \matr{\Sigma}; \vec{x}).
\end{equation}

\noindent We will assume that the covariance matrix $\Sigma$ is full rank.
The case where $\Sigma$ is not full rank we can easily detect and solve
the estimation problem in the linear subspace of samples.
\smallskip

\noindent The core complexity measure of Borel sets in $\reals^d$ that
we use is the notion of Gaussian Surface Area defined below.

\begin{definition}[\textsc{Gaussian Surface Area}]
  \label{def:GaussianSurfaceArea}
    For a Borel set $A \subseteq \R^d$, $\delta \geq 0$ let $A_\delta = \{x :
    \dist(x, A) \leq \delta\}$.  The Gaussian surface area of $A$ is
    \[ \Gamma(A) = \liminf_{\delta \to 0} \frac{\normal_0(A_\delta \setminus A)}{\delta}.\]
We define the Gaussian surface area of a family of sets $\mcal{C}$ to be
$\Gamma(\mcal{C}) = \sup_{C \in \mcal{C}} \Gamma(C)$.
\end{definition}

\subsection{Problem formulation} \label{sec:problemFormulation}

Given samples from a truncated Gaussian $\normal^*_S \triangleq
\normal(\vec{\mu}^*, \matr{\Sigma}^*, S)$, our goal is to learn the
parameters $(\vec{\mu}^*,\matr{\Sigma}^*)$ and recover the set $S$. We
denote by $\alpha^* = \normal(\vec{\mu}^*, \matr{\Sigma}^*; S)$, the total
mass contained in set $S$ by the untruncated Gaussian $\normal^*
\triangleq \normal(\vec{\mu}^*, \matr{\Sigma}^*)$. Throughout this paper,
we assume that we know an absolute constant $\alpha > 0$ such that
\begin{equation}\label{eq:globalLowerBound}
  \normal(\vec \mu^*, \mat \Sigma^* ; S) = \alpha^* \geq \alpha.
\end{equation}

\iffalse{
The assumption that the constant $\alpha$ is known is not crucial and
one can easily adapt our approach to handle the case where the lower
bound $\alpha$ is uknown, see \ref{rem:unknownAlpha}
}
\fi

\section{Identifiability with bounded VC dimension}
\label{sec:identifiabilityVC}
In this section we analyze the sample compexity of learning the true
Gaussian parameters when the truncation set has bounded VC-dimension.
In particular we show that the overhead over the $d^2/\eps^2$ samples
(which is the sample compexity of learning the parameters of the
Gaussian without truncation) is proportional to the VC dimension of the
class.
\begin{theorem}\label{thm:MLESamples}
  Let $\mcal{S}$ be a family of sets of finite VC dimension, and let
  $\normal(\vec \mu, \matr \Sigma, S)$ be a truncated Gaussian
  distribution such that $\normal(\vec \mu, \matr \Sigma; S) \geq \alpha$.
  Given $N$ samples with
  \[
    N = \poly(1/\alpha)\ \wt O\lp( \frac{d^2}{\eps^2} + \frac{\VC(\mcal{S})}{\eps} \rp)
  \]
  Then, with probability at least $99\%$, it is possible to identify
  $(\wb \mu, \wb \Sigma)$ that satisfy \( \dtv{\normal(\vec \mu, \matr
  \Sigma)}{\normal(\wb{\mu}, \wb{\Sigma })} \leq \eps \) and
  $ \norm{\matr \Sigma^{-1/2}(\vec \mu - \wb \mu)}_{2} \le \eps $ and $
  \norm{\mat{I} - \matr \Sigma^{-1/2} \wb {\Sigma} \matr
  \Sigma^{-1/2}}_{F} \le \eps$.
\end{theorem}

Our algorithm works by first learning the truncated distribution within
total variation distance $\eps$. To do this, we first assume that we
know the mean and covariance of the underlying Gaussian by guessing the
parameters and accurately learn the underlying set. After drawing $N =
\Theta(\frac{\VC(\mcal{S}) \log(1/\eps)}{\eps})$ samples from the
distribution, any set in the class that contains the samples will only
exclude at most an $\eps$ fraction of the total mass.  Picking the set
$\wt S$ that maximizes the likelihood of those samples, i.e. the set
with minimum mass according to the guessed Gaussian distribution,
guarantees that the total variation distance between the learned
truncated distribution and the true is at most $\eps$, if the guess of
the parameters was accurate (Lemma~\ref{lem:tvdVCBound}).  The proof of
Lemma~\ref{lem:tvdVCBound} can be found in Appendix~\ref{app:vc}.

\begin{lemma}\label{lem:tvdVCBound}
  Let $\mcal{S}$ be a family of subsets in $\R^d$ and Let $\normal(\vec
  \mu, \matr \Sigma, S^*) = \normal_S^*$ be a Normal distribution
  truncated on the set $S^* \in \mcal{S}$.  Fix $\eps \in (0,1), \delta
  \in (0,1/4)$ and let
  \[
    N = O\lp(\frac{\VC(\mcal{S}) \log(1/\eps)}{\eps} + \log\lp(\frac 1 \delta \rp) \rp)
  \]
  Moreover, let $\wb \mu, \wb \Sigma$ be such that $\dtv{\normal(\wb
  \mu, \wb \Sigma)}{\normal(\vec \mu, \matr \Sigma)} \leq \eps$.  Assume
  that we draw $N$ samples $\vec x_i$ from $\normal_{S^*}$, Let $\wt{S}$
  be the solution of the problem
  \begin{align*}
    \min_S \normal(\wb \mu, \wb \Sigma; S) \,\,
    \text{ subject to } x_i \in S \text{ for all } i\in [n]
  \end{align*}
  Then with probability at least $1-\delta$ we have
  \(
  \dtv{\normal(\wb \mu, \wb \Sigma, \wt S)}{\normal(\vec \mu, \matr \Sigma, S)}
  \leq 3\eps/(2 \alpha).
  \)
\end{lemma}

This is because the total variation distance between two densities $f$
and $g$ can be written as $\int (f(x) - g(x) ) \1{f(x) > g(x)} dx$.
Note that by choosing the set of the smallest mass consistent with the
samples, we guarantee that the guess will have higher density at every
point apart from those outside the support $\wt S$. However, as we
argued the outside mass is at most $\eps$ with respect to the true
distribution which gives the bound in the total variation distance.

To remove the assumption that the true parameters are known, we build a
cover of all possible mean and covariance matrices that the underlying
Gaussian might have and run the tournament from \cite{DK14} to identify
the best one (Lemma~\ref{lem:tournament}). While there are
$(d/\eps)^{O(d^2)}$ such parameters, the number of samples needed for
running the tournament is only logarithmic which shows that an
additional $\wt O(d^2/\eps^2)$ are sufficient to find a hypothesis in
total variation distance $\eps$ (Lemma~\ref{lem:learningTVD}). The
proof of Lemma~\ref{lem:learningTVD} can be found in
Appendix~\ref{app:vc}.

\begin{lemma}\label{lem:learningTVD}
  Let $S \in \mcal{S}$ be a subset of $\R^d$ and $\normal(\vec \mu, \matr
  \Sigma, S)$ be the corresponding truncated normal distribution.  Then \(
  \wt{O}\lp(\VC(\mcal{S})/\eps + d^2/\eps^2  \rp) \) samples are
  sufficient to find parameters $\wt{\vec \mu}, \wt {\matr \Sigma}, \wt S$
  such that $\dtv{\normal(\vec \mu, \matr \Sigma, S)}{\normal(\wt{\vec
  \mu}, \wt {\matr \Sigma}, \wt S)} \leq \eps$ with probability at least
  $99\%$.
\end{lemma}

We finally argue that the $\eps$ error in total variation of the truncated distributions
translates to an $O(\eps)$ bound in total variation distance of the untruncated distributions (Lemma~\ref{lem:tvdLowerBound}). We
show that this is true in general and does not depend on the complexity of
the set.  To prove this statement, we consider two Gaussians with
parameters that are far from each other and construct the worst possible
set to make their truncated distributions as close as possible. We show
that under the requirement that the set contains at least $\alpha$ mass,
the total variation distance of the truncated distributions will be large.

\begin{lemma}[Total Variation of Truncated Normals] \label{lem:tvdLowerBound}
  Let $D_1 = \normal(\vec \mu_1, \matr \Sigma_1, S_1)$ and $D_2 = \normal(\vec
  \mu_2, \matr \Sigma_2, S_2)$ be two truncated Normal distributions
  such that $\normal(\vec \mu_1, \matr \Sigma_1; S_1),
  \normal(\vec \mu_2, \matr \Sigma_2; S_2) \geq \alpha$.
  Then
  \begin{align*}
    \dtv{D_1}{D_2}
  &\geq C_{\alpha}\ \dtv{\normal(\vec \mu_1, \matr \Sigma_1)}{\normal(\vec \mu_2, \matr \Sigma_2)}
\end{align*}
where $C_\alpha < \alpha/8$ is a positive constant that only depends
on $\alpha$, $C_\alpha = \Omega(\alpha^3)$.
  \end{lemma}
  \begin{proof}
    Without loss of generality we assume that $D_1 = \normal(\vec 0, \matr
    I, S_1)$ and $D_2 = \normal(\vec \mu, \matr \Lambda, S_2)$, where
    $\matr \Lambda$ is a diagonal matrix.  We want to find the worst sets
    $S_1, S_2$ so that $\dtv{D_1}{D_2}$ is small.  If $D_1(S_1\setminus
    S_2) \geq \alpha/2$ then the statement holds.  Therefore, we consider
    the set $S = S_1 \cap S_2$ and relax the constraint that the truncated
    Gaussian $D_2$ integrates to $1$.  Taking into account the fact that
    the set $S = S_1 \cap S_2$ must have at least some mass $\alpha/2$
    with respect to $\normal(\vec 0, \matr I)$,
    the following optimization problem provides a lower bound
    on the total variation distance of $D_1$ and $D_2$.
    \begin{align*}
      \min_{S \in\mcal{S}, \beta>0} \qquad \quad & \frac 1 \alpha
      \int |\normal(\vec 0, \matr I; \vec x)
      - \frac \alpha \beta \normal(\vec \mu, \matr \Lambda; \vec x)
      |\ \1{S}(\vec x) \d x
      \\
      \mrm{subj.}\ \mrm{ to }\qquad &
      \int \normal(\vec 0, \matr I; \vec x) \ \1{S}(\vec x) \d \vec x
      \geq \alpha/2,
    \end{align*}
    For any fixed $\beta>0$ this is a fractional knapsack problem and
    therefore we should include in the set the points $x$ in order of
    increasing ratio of weight that is contribution to the $L_1$ error
    $
    |\normal(\vec 0, \matr I; \vec x) -
    \frac \alpha \beta \normal(\vec \mu, \matr \Lambda; \vec x) |
    $,
    over value, that is density $\normal(\vec 0, \matr I; \vec x)$ until we
    reach some threshold $T$.  Therefore, the set is defined to be
    \[
      S = \lp\{ \vec x \in \R^d:
        \frac{
          |\normal(\vec 0, \matr I; \vec x) -
        \frac \alpha \beta \normal(\bm \mu, \matr \Lambda; \vec x)|}
        {\normal(\vec 0, \matr I; \vec x)}
        \leq T
      \rp\}
      = \lp\{ \vec x \in \R^d:
        \lp| 1 - \exp(p(\vec x)) \rp|
        \leq T
      \rp\},
    \]
    where $p(\vec x) = -\frac12 (\vec \mu - \vec x)^T \matr
    \Lambda^{-1}(\vec \mu - \vec x) + \frac 12 \vec x^T \vec x +
    \log(\alpha/(\sqrt{|\matr \Lambda|}\beta)) $.  Using
    Theorem~\ref{thm:GaussianMeasurePolynomialThresholdFunctions} for the
    degree $2$ polynomial $p(\vec x)$ and setting $q=4$, $\gamma =
    \alpha^2 (\E_{x \sim \normal_0}p^2(x))^{1/2}/(256 C^2)$, where
    $C$ is the absolute constant of
    Theorem~\ref{thm:GaussianMeasurePolynomialThresholdFunctions}, we get that
    \[
      \normal_0(\{\vec z : |p(\vec z)| \leq \gamma\}) \leq \frac{\alpha}{4}.
    \]
    To simplify notation set $Q = \{\vec  z : |p(\vec z)|\leq \gamma \}$.
    Therefore, for any $\vec x$ in the remaining $\alpha/4$ mass of the
    set $S$ we know that $|p(\vec x)| \geq \gamma$.  Next, we lower bound
    $\gamma$ in terms of the distance of the parameters of the two
    Gaussians.  We have
    \begin{align*}
      \E_{\vec x \sim \normal_0}[p^2(\vec x)]
    &\geq \Var_{\vec x \sim \normal_0}[p(\vec x)]
    =\Var_{\vec x \sim \normal_0}
    \lp[
    -\frac12 (\vec \mu - \vec x)^T \matr \Lambda^{-1}(\vec \mu - \vec x)
    + \frac 12 \vec x^T \vec x
    \rp]
 \\ &=
 \Var_{\vec x \sim \normal_0}
 \lp[
 \sum_{i=1}^d \lp(\frac{\mu_i}{\lambda_i} x + x^2\frac{(1-1/\lambda_i)}{2}
 \rp)
 \rp]
 =
 \sum_{i=1}^d
 \Var_{x \sim \normal(0,1)}
 \lp[ \frac{\mu_i}{\lambda_i} x + x^2 \frac{(1-1/\lambda_i)}{2} \rp]
 \\
    &=
    \sum_{i=1}^d
    \frac 12
    \lp( \frac{1}{\lambda_i} - 1 \rp)^2 + \frac{\mu_i^2}{\lambda_i^2}
    = \frac 12 \norm{\matr \Lambda^{-1} - \matr I}_F^2 +
    \norm{\matr \Lambda^{-1/2} \vec \mu}_2^2
  \end{align*}
  Therefore, using the inequality $\sqrt{2}\sqrt{x+y} \geq \sqrt{x} +
  \sqrt{y}$ we obtain
  \[
    \gamma \geq \frac{\alpha^2}{256\sqrt 2 C^2}
    \lp(
    \frac{1}{\sqrt{2}} \norm{\matr \Lambda^{-1} - \matr I}_F +
    \norm{\matr \Lambda^{-1/2} \vec \mu}_2
    \rp)
    \geq \frac{\alpha^2}{256 C^2}
     \dtv{\normal(\vec \mu_1, \matr \Sigma_1)}{\normal(\vec \mu_2, \matr \Sigma_2)},
  \]
  where we used Lemma~\ref{lem:tvdParameter}.
  Assume first that $\gamma \leq 1$. We have that the $L_1$ distance
  between the functions $f(\vec x) = \normal(\vec 0, \matr I; \vec
  x)\1{S}(\vec x)$ and $g(\vec x) = \frac \alpha \beta \normal(\vec \mu,
  \matr \Lambda; \vec x) \1{S}(\vec x)$ is
  \begin{align*}
    \int | f(\vec x) - g(\vec x) | \d \vec x
    &=
    \Exp_{\vec x \sim \normal_0}[|1 - \exp(p(\vec x))| \1{S}(\vec x)]
    \geq
    \Exp_{\vec x \sim \normal_0}
    \lp[ \frac{|p(x)|}{2} \1{S\setminus Q}(\vec x) \rp]
    \\
    &\geq
    \gamma \Exp_{\vec x \sim \normal_0}
    \lp[\1{S\setminus Q}(\vec x) \rp]
    \geq
    \frac{\alpha \gamma}{4}
    \geq C_{\alpha} \dtv{\normal(\vec \mu_1, \matr \Sigma_1)}{\normal(\vec \mu_2, \matr \Sigma_2)},
  \end{align*}
  where for the first inequality we used the inequality $|1-\me^{x}|\geq |x|/2$
  for $|x| \leq 1$.  Note that $C_a = \Omega(\alpha^3)$.
  If $\gamma > 1$ we have
  \[
    \int | f(\vec x) - g(\vec x) | \d \vec x
    =
    \Exp_{\vec x \sim \normal_0}[|1 - \exp(p(\vec x))| \1{S}(\vec x)]
    \geq
    \Exp_{\vec x \sim \normal_0}
    \lp[ \frac{1}{2} \1{S\setminus Q}(\vec x) \rp]
    \geq \alpha/8,
    \\
  \]
  where we used the inequality $|1-\me^x | \geq 1/2$ for $|x| > 1$.
\end{proof}

 \section{Estimation Algorithm for bounded Gaussian Surface Area}\label{sec:algorithm}

In this section, we present the main steps of our estimation algorithm. In
later sections, we provide details of the individual components. The
algorithm can be thought of in 3 stages.

\paragraph{First Stage} In the first stage, our goal is to learn a
weighted characteristic function of the underlying set. Even though we
cannot access the underlying set directly, for any given function $f$ we
can evaluate the expectation $\E_{\vec x \sim \normal(\vec \mu^*, \mat
\Sigma^*, S)}[f(\vec x)]$ using truncated samples.

This expectation can be equivalently written as $\E_{\vec x \sim
  \normal(\vec 0, \mat I)}[f(\vec x) \tind(\vec x)]$ for the function
  $$\tind(\vec x) \triangleq \frac{\1{S}(\vec x)}{\alpha^*}
  \frac{\normal(\vec \mu^*, \mat \Sigma^*; \vec x)}{\normal(\vec 0,\mat I;
  \vec x)} = \frac{\1{S}(\vec x)}{\alpha^*} \frac{\normal^*(\vec
x)}{\normal_0(\vec x)}.  $$

By evaluating the above expectation for different functions $f$
corresponding to the Hermite polynomials $H_V(\vec x)$, we can recover
$\tind(\vec x)$, through its Hermite expansion: $$\tind(\vec x) = \sum_{V
\in \nats^d} \E_{\vec x \sim \normal_0}[H_V(\vec x) \tind(\vec x)]
H_V(\vec x) = \sum_{V \in \nats^d} \E_{\vec x \sim \normal^*_S}[H_V(\vec
x)] H_V(\vec x).$$

Of course, it is infeasible to calculate the Hermite expansion for any $V
\in \nats^d$. In Section~\ref{sec:weighted_learning}, we show that by
estimating only terms of degree at most $k$, we can achieve a good
approximation to $\psi$ where the error depends on the Gaussian surface
area of the underlying set $S$.  To do this, we show that most of the mass
of the coefficients $c_V = \E_{\vec x \sim \normal_0}[H_V(\vec x)
\tind(\vec x)]$ is concentrated on low degree terms, i.e. $\sum_{|V| > k}
c^2_V$ is significantly small. Moreover, we show that even though we can
only estimate the coefficients $c_V$ through sampling, the sampling error
is significantly small.

Overall, after the first stage, we obtain a non-negative function
$\tind_k$ that is close to $\tind$. The approximation error guarantees are
given in Theorem~\ref{thm:estimationApproximationError}.

\paragraph{Second Stage} Given the function $\tind_k$ that was recovered
in the first stage, our goal is to decouple the influence of the set
$\frac{\1{S}(\vec x)}{\alpha^*}$ and the influence of the underlying
Gaussian distribution which corresponds to the multiplicative term
$\frac{\normal(\vec \mu^*, \mat \Sigma^*; \vec x)}{\normal(\vec 0,\mat I;
\vec x)}$. This would be easy if we had the exact function $\tind$ in
hand. In contrast, for the polynomial function $\tind_k$ the problem is
significantly challenging as it is only close to $\tind$ on average but
not pointwise.

To perform the decoupling and identify the underlying Gaussian we
explicitly multiply the function $\tind_k$ with a corrective term of the
form $\frac{\normal(\vec 0,\mat I; \vec x)}{\normal(\vec \mu, \mat \Sigma;
\vec x)}$.  We set up an optimization problem seeking to minimize the
function $C(\vec \mu, \mat \Sigma) \E_{\vec x \sim \normal^*_S}[
	\frac{\normal(\vec 0,\mat I; \vec x)}{\normal(\vec \mu, \mat
	\Sigma; \vec x)} \tind_k(\vec x) ]$ with an appropriate choice of
$C(\vec \mu, \mat \Sigma)$ so that the unique solution corresponds to
$(\vec \mu, \mat \Sigma) = (\vec \mu^*, \mat \Sigma^*)$. Under a
reparameterization of $(\vec u, \mat B) = (\mat \Sigma^{-1} \vec \mu, \mat
\Sigma^{-1})$, we show that the corresponding problem is \emph{strongly}
convex. Still, optimizing it directly is non-trivial as it involves taking
the expectation with respect to the unknown truncated Gaussian.  Instead,
we perform stochastic gradient descent (SGD) and show that it quickly
converges in few steps to point close to the true minimizer
(Algorithm~\ref{alg:projectedSGD}).

This allows us to recover parameters $(\hat {\vec \mu}, \hat {\mat
\Sigma})$ so that the total variation distance between the recovered and
the true (untruncated) Gaussian is very small, i.e. $d_{TV}\lp(
\normal(\hat{\vec \mu}, \hat{\mat \Sigma}), \normal(\vec \mu^*, \mat
\Sigma^*) \rp) \le \eps$.  Theorem~\ref{thm:mainTheoremLocalConvergence}
describes the guarantees of the second stage. Further details are provided
in Section~\ref{sec:optimization}.

\paragraph{Third Stage} Given the weighted indicator function $\tind_k$
and the recovered Gaussian $\normal(\hat{\vec \mu}, \hat{\mat \Sigma})$,
we move on to recover the underlying set $S$. To do this, we compute the
function $\frac{\normal(\vec 0,\mat I; \vec x)}{\normal(\hat{\vec \mu},
\hat{\mat \Sigma}; \vec x)} \tind_k(\vec x)$ and set a threshold at $1/2$.
It is easy to check that if there were no errors, i.e. $\tind_k = \tind$
and $d_{TV}\lp( \normal(\hat{\vec \mu}, \hat{\mat \Sigma}), \normal(\vec
\mu^*, \mat \Sigma^*) \rp)=0$, that this thresholding step would correctly
identify the set. In Section~\ref{sec:set_recover} we bound the error
guarantees of this approach.  We show that it is possible to obtain an
estimate $\hat S$ of the underlying set so that the mass of the symmetric
difference with the true Gaussian is small, i.e. $\normal(\vec \mu^*, \mat
\Sigma^*; S \triangle \hat S) < \eps$.  Overall, our algorithm requires at
most $d^{\poly(1/\alpha,1/\eps) \Gamma^2(S)}$, where $\Gamma(S)$ is the
Gaussian surface area of the set $S$ and $\alpha$ is a lower-bound on the
mass that is assigned by the true Gaussian on the set $S$. The running
time of our algorithm is linear in the number of samples.

\paragraph{The guarantees of the algorithm}

We first show our algorithmic results under the assumption that the untruncated Gaussian
$\normal^*$ is known to be in near-isotropic position.

\begin{definition}[Near-Isotropic Position]\label{def:isotropic}
	Let $\vec{\mu} \in \R^d$, $\mat{\Sigma} \in \R^{d \times d}$ be a
	positive semidefinite symmetric matrix and $a, b > 0$.  We say
	that $(\vec{\mu}, \mat{\Sigma})$ is in $(a, b)$\textit{-isotropic
	position} if the following hold.  $$\norm{\vec{\mu}}_2^2 \le a,
	\quad \norm{\mat{\Sigma} - \mat{I}}^2_F \le a, \quad (1-b) \mat{I}
	\preceq \mat{\Sigma} \preceq \frac 1 {1-b} \mat{I}$$
\end{definition}

We later transform the more interesting case with an unknown mean and an
unknown diagonal covariance matrix to the isotropic case.

\begin{theorem} \label{thm:mainTheoremLocalConvergence}
  Let ${\cal N}(\vec{\mu}^*, \matr{\Sigma}^*)$ be a $d$-dimensional
	Gaussian distribution that is in $(O(\log(1/\alpha^*),
	1/16)$-isotropic position and consider a set $S$ such that ${\cal
	N}(\vec{\mu}^*, \matr{\Sigma}^*; S) \ge \alpha$. There exists an
	algorithm such that for all $\epsilon > 0$, the algorithm uses $n
	> d^{\poly(1/\alpha) \frac{\Gamma^2(S)}{\eps^8}}$ samples and
	produces, in $\poly(n)$ time, estimates that, with probability at
	least $99\%$, satisfy $d_{TV}(\N(\mt, \St), \N(\hat{\m}, \hat{\S})) \le \eps$.
\end{theorem}

\iffalse{
\begin{remark}\label{rem:unknownAlpha}
   We remark that our algorithm does not use the absolute constant $\alpha$.
We only use it to analyze its sample complexity.  If this constant is unknown
we can simply run \ref{alg:projectedSGD} assuming that the lower bound on the
mass of the set at run $i$ is $s = 1/2^i$.  We stop when the accuracy of
the output satisfies
\end{remark}
}
\fi

We can apply this theorem to estimate the parameters of any Gaussian
distribution with an unknown mean and an unknown diagonal covariance
matrix by bringing the Gaussian to an $(O(\log(1/\alpha^*),
1/16)$-isotropic position.  Lemma~\ref{lem:conditionalParameterDistance}
shows that with high probability, we can obtain initial estimates
$\wt{\vec{\mu}}_S$ and $\mat{\wt{\Sigma}}_S$ so that
$\|\mat{\Sigma}^{-1/2} (\wt{\vec{\mu}}_S - \vec{\mu}^*)\|^2_2 \leq O({\log
\frac{1}{\alpha}})$ and $${\wt{\mat{\Sigma}}}_S \succeq \Omega(\alpha^2)
\mat{\Sigma}^*, \quad \text{and}\quad \snorm{F} {\mat {\Sigma}^*{^{-1/2}}
\mat{\wt{\Sigma}}_S \mat {\Sigma}^*{^{-1/2}}  - \mat I}^2 \le O({\log
\frac{1}{\alpha}} ).$$

Given these estimates, we can transform the space so that
$\wt{\vec{\mu}}_S = 0$, and $\wt{\mat{\Sigma}}_S = \mat I$.  We note that
after this transformation, the mean will be at the right distance from
$0$, while the eigenvalues $\lambda_i$ of $\mat \Sigma^*$ will all be
within the desired range $\frac{15}{16} \le \lambda_i \le \frac{16}{15}$
apart from at most $O( \log({1/\alpha})  )$. This is because the condition
$\snorm{F}{\mat \Sigma^{*-1/2} \mat{\wt{\Sigma}}_S \mat \Sigma^{*-1/2}  -
\mat I}^2 \le O({\log \frac{1}{\alpha}} )$ implies that $\sum_i
(1-\frac{1}{\lambda_i})^2 \le O( \log(1/\alpha) )$.  With this
observation, since we know of the eigenvectors of $\Sigma^*$, we would be
able to search over all possible corrections to the eigenvalues to bring
the Gaussian in $(O(\log(1/\alpha)), \frac{1}{16})$-isotropic position as
required by Theorem~\ref{thm:mainTheoremLocalConvergence}. We only need to
correct $O(\log(1/\alpha))$ of them.

We can form a space of candidate hypotheses for the underlying
distribution, for each choice of $O(\log(1/\alpha))$ out of the $d$
vectors along with the all possible scalings. These hypotheses are at most
$d^{O(\log(1/\alpha))}$ times $(\log(1/\alpha))^{O(\log(1/\alpha))}$ for
all possible scalings. Thus, there are at most $d^{O(\log(1/\alpha))}$
hypotheses. Running the algorithm for each one of them, we would learn at
least one distribution and one set that is accurate according to the
guarantees of Theorems~\ref{thm:mainTheoremLocalConvergence}. Running the
generic hypothesis testing algorithm of Lemma~\ref{lem:tournament}, we can
identify one that is closest in total variation distance to the true
distribution $\Nt_S$. The sample complexity and runtime would thus only
increase by at most $d^{O(\log(1/\alpha))}$. As we showed in
Lemma~\ref{lem:tvdLowerBound}, knowing the truncated Gaussian in total
variation distance suffices to learn in accuracy $\eps$ the parameters of
the untruncated distribution. We thus obtain as corollary, that we can
estimate the parameters when the covariance is spherical or diagonal. The
same results hold when one wants to recover the underlying set in these
cases.
 \subsection{Learning a Weighted Characteristic Function}\label{sec:weighted_learning}

Our goal in this section is to recover using conditional samples from
$\normal^*_S$ a weighted characteristic function of the set $S$.  In
particular, we will show that it is possible to learn a good approximation to
the function
\begin{align} \label{eq:definitionOfShiftedIndicator}
  \tind(\vec x) =
  \frac{\1{S}(\vec x)}{\alpha^*}
  \frac{\normal(\vec \mu^*, \mat \Sigma^*; \vec x)}{\normal(\vec 0,\mat I; \vec x)}
  = \frac{\1{S}(\vec x)}{\alpha^*} \frac{\normal^*(\vec x)}{\normal_0(\vec x)}.
\end{align}

We will later use the knowledge of this function to extract the unknown parameters and learn the set $S$.

\subsubsection{Hermite Concentration}\label{subsec:concentration}

We start by showing that the function $\tind(\vec x)$
admits strong Hermite concentration.  This means that we can well-approximate
$\tind(\vec x)$ if we ignore the higher order terms in the Hermite expansion of $\tind(\vec x)$.

\begin{theorem}\label{thm:polynomialApproximation}\textsc{(Low Degree Approximation)}
  Let $S_k\tind$ denote the degree $k$ Hermite expansion of function $\tind$ defined in
  \eqref{eq:definitionOfShiftedIndicator}. We have that
  \[
    \E_{\vec x \sim \normal_0}
    \lp[ (S_k\tind(\vec x) - \tind(\vec x))^2 \rp]=
    \sum_{|V| \geq k} \hat{\tind}(V)^2
    \leq \poly(1/\alpha) \lp(\frac{\sqrt{\Gamma(S)}}{k^{1/4}} + \frac{1}{k} \rp).
  \]
  where $\Gamma(S)$ is the Gaussian surface area of $S$, and $a< \alpha^*$ is the absolute
  constant of \eqref{eq:globalLowerBound}.
\end{theorem}

We note that the Hermite expansion of $\tind$ is well-defined as $\psi(\vec x) \in L_2(\R^d, \normal_0).$ This can be seen from the following lemma
which will be useful in many calculations throughout the paper.

\begin{lemma}\label{lem:ratioBound}
  Let $\normal(\vec{\mu}_1, \mat{\Sigma}_1)$ and $\normal(\vec{\mu}_2, \mat{\Sigma}_2)$ be two $(B, \frac{1-\delta}{2k})$-isotropic Gaussians for some parameters $B,\delta>0$ and $k \in \nats$. It holds
  \[
    \exp\lp(- \frac{13 k^2}{\delta} B\rp) \le \E_{x \sim \normal_0}
    \lp[ \lp(\frac{\normal(\vec{\mu}_1, \mat{\Sigma}_1; \vec x)}
    {\normal(\vec{\mu}_2, \mat{\Sigma}_2; \vec x)}\rp)^k
    \rp] \le \exp\lp( \frac{13 k^2}{\delta} B\rp).
  \]
\end{lemma}

\noindent Lemma~\ref{lem:ratioBound} applied for $\normal_0$ and $\normal^*$ for $k=2$ implies that $\psi(\vec x) \in L_2(\R^d, \normal_0)$.

To get the desired bound for Theorem~\ref{thm:polynomialApproximation} we use the following lemma, which allows us to bound the Hermite concentration of a
function $f$ through its noise stability.
\begin{lemma}\label{lem:realSensitivityConcentration}
  For any function $f: \reals^d \mapsto \reals$ and parameter $\rho \in (0,1)$, it holds
  \[
    \sum_{|V| \geq 1/\rho} \hat{f}(V)^2
    \leq 2 \E_{\bm{x} \sim \mcal{N}(\bm{0}, \mat{I})}
    \lp[f(\bm{x})^2 - f(\bm{x}) T_{1-\rho} f(\bm{x}) \rp]
  \]
\end{lemma}
Lemma~\ref{lem:realSensitivityConcentration} was originally shown in \cite{KKMS05} for indicator functions of sets, but their proof extends to arbitrary real functions. We provide the proof in the appendix for completeness.

Using Lemma~\ref{lem:realSensitivityConcentration}, we can obtain Theorem~\ref{thm:polynomialApproximation} by bounding the noise sensitivity of the function $\tind$. The following lemma directly gives the desired result.

\begin{lemma}\label{lem:realSensitivity} For any $\rho \in (0,1)$,
  $
    \E_{\bm{x} \sim \normal_0}
    \lp[\tind(\bm{x})^2 - \tind(\bm{x}) T_{1-\rho} \tind(\bm{x}) \rp]
    \leq
    \poly(1/\alpha)\lp({\sqrt{\Gamma(S)}}{\rho^{1/4}} + \rho\rp)
  $.
\end{lemma}

To prove Lemma~\ref{lem:realSensitivity}, we will require the following lemma whose proof is provided in the appendix.

\begin{lemma}\label{lem:noiseDerivative}
  Let $r(\vec x) \in L_2(\R^d, \normal(\vec 0, \mat I))$
  be differentiable at every $\vec x \in \R^d$.
  Then
  \[
    \frac{1}{2}
    \E_{(x, z) \sim D_\rho}[(r(\vec x) - r(\vec z))^2]
    \leq \rho \E_{x \sim \normal(\vec 0, \mat I)}
    \lp[ \snorm{2}{\nabla r(\vec x)}^2 \rp]
  \]
\end{lemma}

We now move on to the proof of Lemma~\ref{lem:realSensitivity}.

\begin{prevproof}{Lemma}{lem:realSensitivity}
  For ease of notation we define the following distribution
  \[
    D_\rho =
    \normal\lp(\vec 0,
    \begin{pmatrix}
      \mat I & (1-\rho) \mat I \\
      (1-\rho) \mat I & \mat I
    \end{pmatrix}
    \rp).
  \]
  We also denote by $r(x) = \normal^*(\vec x)/ \normal_0(\vec x)$
  We can now write
  \begin{align*}{2}
    \E_{\vec x \sim \normal_0}
    \lp[\tind(\bm{x})^2 - \tind(\bm{x}) T_{1-\rho} \tind(\bm{x}) \rp]
    &=
    \E_{(\vec x, \vec z) \sim D_\rho}
    \lp[\tind(\bm{x})^2 - \tind(\bm{x}) \tind(\bm{z}) \rp]
  \\&=
  \frac{1}{\alpha^*{^2}}
  \E_{(\vec x, \vec z) \sim D_\rho}
  [
  \1{S}(\vec x) r^2(\vec x)
  - \1{S}(\vec x) \1{S}(\vec z) r^2(\vec x)]
  +
  \\
    &\E_{(\vec x, \vec z) \sim D_\rho}
    [
    \1{S}(\vec x) \1{S}(\vec z) r^2(\vec x)
    - \1{S}(\vec x) \1{S}(\vec z) r(\vec x) r(\vec z)
    ]
  \end{align*}
  We bound each of the two terms separately. For the first term,
  using Schwarz's inequality we get
  \begin{align*}
    \E_{(\vec x, \vec z) \sim D_\rho}
    [ \1{S}(\vec x) r^2(\vec x)
    - \1{S}(\vec x) \1{S}(\vec z) r^2(\vec x)]
    &\leq
    \Big(
      \E_{(\vec x, \vec z) \sim D_\rho}
      [ \1{S}(\vec x) \1{\bar {S}}(\vec z)]
    \Big)^{1/2}
    \Big(
      \E_{(\vec x, \vec z) \sim D_\rho}
      [r^4(\vec x)]
    \Big)^{1/2}
    \nonumber \\ &\le
    (\NS[S])^{1/2} \poly(1/\alpha) \le
        {\sqrt{\Gamma(S)}}{\rho^{1/4}} \poly(1/\alpha)
  \end{align*}
  where the bound on the expectation of $r^4(\vec x)$ follows from Lemma~\ref{lem:ratioBound} and the last inequality follows from
  Lemma~\ref{lem:KOS08NoiseSurface}.

  For the second term, we have that
  \begin{align*}
    \E_{(\vec x, \vec z) \sim D_\rho} [
    \1{S}(\vec x) \1{S}(\vec z) (r^2(\vec x) - r(\vec x) r(\vec z))
    ]
    &=
    \E_{(\vec x, \vec z) \sim D_\rho} \lp[
    \1{S}(\vec x) \1{S}(\vec z)
    \lp(\frac{r^2(\vec x)}{2} + \frac{r^2(\vec z)}{2} - r(\vec x) r(\vec z)\rp)
    \rp]
    \nonumber \\&=
    \E_{(\vec x, \vec z) \sim D_\rho} \lp[
    \1{S}(\vec x) \1{S}(\vec z)
    \frac12 (r(\vec x) - r(\vec z))^2
    \rp]
    \nonumber \\&\leq
    \frac{1}{2}
    \E_{(\vec x, \vec z) \sim D_\rho} \lp[
    (r(\vec x) - r(\vec z))^2
    \rp] \leq
    \rho \E_{\vec x \sim N_0}[\snorm{2}{\nabla r(\vec x)}^2],
  \end{align*}
  where the last inequality follows from
  Lemma~\ref{lem:noiseDerivative}.
  It thus suffices to bound the expectation of the gradient
  of $r$.  We have
  \begin{align*}
    \E_{\vec x \sim N_0}[\snorm{2}{\nabla r(\vec x)}^2]
   &=
   \E_{\vec x \sim N_0}\lp[
   \snorm{2}{ -\mat \Sigma^*{^{-1}}(\vec x - \vec \mu^*) +\vec x}^2 r^2(\vec x)\rp]
    \nonumber \\ &\leq
    2 \E_{\vec x \sim N_0}[
    \snorm{2}{ (\mat I -\mat \Sigma^*) {^{-1}}\vec x}^2 r^2(\vec x) ] +
    2 \snorm{2}{{\Sigma^*}^{-1} \vec \mu^*}^2 \E_{\vec x \sim N_0}[ r^2(\vec x)]
     \\ &\leq
    2 \sqrt{ \E_{\vec x \sim N_0}[
    \snorm{2}{ (\mat I -\mat \Sigma^*{^{-1}}) \vec x}^4 ] \E_{\vec x \sim N_0}[ r^4(\vec x) ] } +
    2 \snorm{2}{{\Sigma^*}^{-1} \vec \mu^*}^2 \E_{\vec x \sim N_0}[ r^2(\vec x)] \le \poly(1/\alpha)
  \end{align*}
  where the bound on the expectation of $r^4(\vec x)$ and $r^2(\vec x)$ follows
  from Lemma~\ref{lem:ratioBound} and the expectation
  $$\E_{\vec x \sim N_0}\lp[
    \snorm{2}{ (\mat I -\mat \Sigma^*{^{-1}}) \vec x}^4 \rp]
    = \E_{\vec x \sim N_0}\lp[ \lp( \sum_i (1-\lambda_i)^2 x_i^2 \rp)^2 \rp]
    \le 3 \lp( \sum_i (1-\lambda_i)^2 \rp)^2
    \le 3 \log^2(1/\alpha)
    \le \poly(1/\alpha)
    $$
\end{prevproof}

\subsubsection{Learning the Hermite Expansion}\label{subsec:learningHermite}
In this section we deal with the sample complexity of estimating the
coefficients of the Hermite expansion.
We have
\[
  c_V = \E_{ \vec x \sim \normal(\vec \mu, \mat \Sigma, S)}[H_V(\vec x)]
\]
Using samples $\vec x_i$ from $\normal(\vec \mu, \mat \Sigma, S)$, we can
estimate this expectation empirically with the unbiased estimate
\[
  \wt c_V = \frac{\sum_{i=1}^N H_V(\vec x_i)}{N}.
\]

We now show an upper bound for the variance of the above estimate.
The proof of this lemma can be found in Appendix~\ref{app:weighted_learning}.
\begin{lemma}\label{lem:hermiteCoefficientVariance}
  Let $\normal(\vec \mu^*, \mat \Sigma^*, S)$ be the unknown truncated Gaussian.
  The variance of the following unbiased estimator of the Hermite coefficients
  \(
  \wt c_V = \frac{\sum_{i=1}^N H_V(\vec x_i)}{N},
  \)
  is upper bounded
  \[
    \E_{\vec x \sim \normal(\vec \mu, \mat \Sigma, S)} [(\wt c_V - c_V)^2]
    \leq \poly(1/\alpha) \frac{5^{|V|}}{N}.
  \]
\end{lemma}

\begin{theorem}\label{thm:estimationApproximationError}
  Let $S$ be an arbitrary (Borel) subset of $\R^d$. Let $\alpha$ be
  the constant of \eqref{eq:globalLowerBound}.
  Let $\normal(\vec \mu^*, \mat \Sigma^*, S)$ be the corresponding
  truncated Gaussian in $(O\log(1/\alpha), 1/16)$-isotropic position
  (see Definition~\ref{def:isotropic}),
  Then, for the estimate
  \[
    \tind_{k}(\vec x) = \max\lp(0, \sum_{V: 0 \leq |V| \leq k} \wt c_V
    H_V(\vec x)\rp), \quad \wt c_V = \frac{\sum_{i=1}^N H_V(\vec x_i)}{N}
  \]
  it holds for $k \ll d$, $\Gamma(S) > 1$,
  \begin{align*}
    \E_{\vec x_1,\ldots, \vec x_N \sim \normal(\vec \mu^*, \mat \Sigma^*, S)}
    \lp[
    \E_{\vec x \sim \normal(\vec 0, \mat I)}
    \lp[ (\tind_{k}(\vec x)- \tind(\vec x))^2 \rp] \rp]
&\leq
      \poly(1/\alpha)
      \lp(
      \frac{\sqrt{\Gamma(S)}}{k^{1/4}}
      +
      \frac{(5d)^k}{N}
      \rp).
  \end{align*}
  Alternatively, for $k = \poly(1/\alpha) \Gamma(S)^2/\eps^4$  we
  obtain that with $N=d^{\poly(1/\alpha) \Gamma(S)^2/\eps^4}$ samples,
  with probability at least $9/10$, it holds
  $\E_{\vec x \sim \normal_0}[(\tind_{N,k}(\vec x) - \tind(\vec x))^2] \leq \eps$.
\end{theorem}
\begin{proof}
  Instead of considering the positive part of the Hermite expansion, we will
  prove the claim for the empirical Hermite expansion of degree $k$ and $N$
  samples
  \[
    p_{N, k} = \sum_{V:0 \leq |V| \leq k} \wt{c}_V H_V(\vec x).
  \]
  As usual we denote by $S_k \tind(\vec x)$ the true (exact) Hermite expansion
  of degree $k$ of $\tind(\vec x)$.
  Using the inequality $(a-b)^2 \leq 2 (a-c)^2 + 2 (c-b)^2$ we
  obtain
  \[
    \E_{\vec x \sim \normal_0}
    \lp[
    (p_{N,k}(\vec x)- f(\vec x))^2
    \rp]
    \leq
    2 \E_{\vec x \sim \normal_0}
    \lp[
    (p_{N,k}(\vec x)- S_k\tind(\vec x))^2
    \rp]
    +
    2 \E_{\vec x \sim \normal_0}
    \lp[
    (S_k \tind (\vec x)- \tind(\vec x))^2
    \rp]
  \]
  Since Hermite polynomials form an orthonormal system with respect to
  $\normal_0$, we obtain
  \[
    \E_{\vec x \sim \normal_0}
    \lp[
    (p_{N,k}(\vec x)- S_k \tind(\vec x))^2
    \rp]
    =
    \E_{\vec x \sim \normal_0}
    \lp[ \lp(
    \sum_{V: 0\leq |V| \leq k} (\wt c_V - c_V ) H_V(\vec x)
    \rp)^2 \rp]
    =
    \sum_{V: 0\leq |V| \leq k} (\wt c_V - c_V)^2.
  \]
  Using Lemma~\ref{lem:hermiteCoefficientVariance} we obtain
  \begin{align*}
    \E_{\vec x_1,\ldots, \vec x_N \sim \normal^*}
    \lp[
    \sum_{V: 0\leq |V| \leq k} (\wt c_V - c_V)^2
    \rp]
    \leq
    \frac{\poly(1/\alpha)}{N} \sum_{V: 0\leq |V| \leq k} 5^{|V|}
    \leq \frac{\poly(1/\alpha)}{N} \binom{d+k}{k} 5^k,
  \end{align*}
  where we used the fact that the number of all multi-indices $V$ of $d$ elements
  such that $0 \leq |V| \leq k$ is $\binom{d + k}{k}$.
  Moreover, from Theorem~\ref{thm:polynomialApproximation} we obtain that
  \[
    \E_{\vec x \sim \normal_0}
    \lp[ (S_{k}\tind (\vec x)- \tind(\vec x))^2 \rp]
    \leq \poly(1/\alpha)
    \lp(
    \frac{\sqrt{\Gamma(S)}}{k^{1/4}} + \frac{1}{k}
    \rp).
  \]
  The theorem follows.
\end{proof}
 \subsection{Optimization of Gaussian Parameters} \label{sec:optimization}

In this section we show that we can formulate a convex objective function that
can be optimized to yield the unknown parameters $\vec{\mu}^*, \matr{\Sigma}^*$
of the truncated Gaussian. Let $S$ be the unknown (Borel) subset of $\R^d$ such
that $\normal(\vec{\mu}^*, \mat{\Sigma}^*; S) = \alpha^*$ and let $\Nt_S =
\normal(\vec{\mu}^*, \mat{\Sigma}^*, S)$ be the corresponding truncated
Gaussian.

To find the parameters $\mt, \St$, we define the function
\begin{align} \label{eq:optimizationFunctionDefinition}
M_f (\vec u, \mat B)
  \triangleq
  \E_{\vec x \sim \Nt_S}
  \lp[
  \me^{h(\vec u, \mat B ; \vec x)}
  \normal(\vec 0, \mat I ; \vec x)
  f(\vec x)
  \rp]
\end{align}
\noindent where $h(\vec u, \mat B ; \vec x) = \frac{\vec x^T \mat B \vec x}{2} -
  \frac{\tr( (\mat B - \mat I) (\wt{\S}_S + \wt{\m}_S \wt{\m}_S^T) )}{2} -
  \vec u^T (\vec x-\wt{\m}_S) + \frac d 2 \log{2\pi}$.

We will show that the minimizer of $M_f(\vec u, \mat B)$ for the polynomial
function $f = \tind_k$, will satisfy
$(\mat B^{-1} \vec u, \mat B^{-1}) \approx (\mt, \St)$.
Note that $M_f(\vec u, \mat B)$ can be estimated through samples. Our goal will
be to optimize it through stochastic gradient descent.

In order to make sure that SGD algorithm for $M_{\tind_k}$ converges fast in the
parameter space we need to project after every iteration to some subset of the
space as we will see in more details later in this Section. Assuming that the
pair $(\mt, \St)$ is in $(\sqrt{\log(1/\alpha^*)}, 1/16)$-isotropic position we
define the following set
\begin{align} \label{eq:projectionSetDefinition}
  \Domain = \left\{ (\vec{u}, \mat{B}) \mid (\mat{B}^{-1} \vec{u}, \mat{B}^{-1})
  ~~\text{ is in $\left( c \cdot \log(1/\alpha^*), 1/16 \right)$-isotropic
  position} \right\}
\end{align}
\noindent Where $c$ is the universal constant guaranteed to exist from Section
\ref{sec:problemFormulation} so that
\[ \max\left\{ \norm{\mt - \tim}_{\St}, \norm{\St - \tiS}_F \right\} \le c \cdot
\log(1/\alpha^*). \]
It is not hard to see that $\Domain$ is a convex set and that
for any $(\vec{u}, \mat{B})$ the projection to $\Domain$ can be done
efficiently. For more details we refer to Lemma 8 of \cite{DGTZ18}. Since after
every iteration of our algorithm we project to $\Domain$ we will assume for the
rest of this Section that $(\vec{u}, \mat{B}) \in \Domain$.

An equivalent formulation of $M_f (\vec u, \mat B)$ that will be useful for
the analysis of the SGD algorithm is
\begin{align}
  M_f (\vec u, \mat B) &=
e^{
    -\frac12 \lp(
    \tr( (\mat B - \mat I) (\wt{\S}_S + \wt{\m}_S \wt{\m}_S^T) ) ) + \vec u^T \mat B^{-1} \vec u - \vec u^T \wt{\m}_S
    \rp)} \sqrt{|\matr{B}|}
\E_{\vec x \sim \Nt_S}
  \lp[
  \frac{
    \normal(\vec 0, \mat I ; \vec x)
  }{
    \normal(\mat B^{-1} \vec u, \mat B^{-1};\vec x)
  }
  f(\vec x)
  \rp] \nonumber \\
  & \triangleq  C_{\vec u, \mat B}
    \E_{\vec x \sim \Nt_S}
    \lp[
    \frac{
      \N_0(\vec x)
    }{
      \N_{\vec u, \mat B}(\vec x)
    }
    f(\vec x)
    \rp] \label{eq:optimizationFunctionEquivalentExpression}
\end{align}

\begin{lemma} \label{lem:cubBound}
  For $(\vec{u}, \mat{B}) \in \Domain$, we have that
  $\poly(\alpha) \le C_{\vec u, \mat B} \le \poly(1/\alpha)$.
\end{lemma}

\begin{proof}
  We have that
  \begin{align*}
    |2 \log{C_{\vec u, \mat B}}| &= \lp|
    \tr( (\mat B - \mat I) (\wt{\S}_S + \wt{\m}_S \wt{\m}_S^T) ) ) + \vec u^T
    \mat B^{-1} \vec u - \vec u^T \wt{\m}_S - \log |\matr{B}|\rp|
    \\
    &= \lp|
    \tr( \mat B - \mat I )  +
    \tr( (\mat B - \mat I) (\wt{\S}_S - \mat I) ) + \vec u^T \mat B^{-1} \vec u
    - \log |\matr{B}|\rp|\\
    &\le
    \lp| \tr( \mat B - \mat I ) - \log |\matr{B}| \rp| +
    \lp| \tr( (\mat B - \mat I) (\wt{\S}_S - \mat I) ) \rp| +
    \lp| \vec u^T \mat B^{-1} \vec u \rp|
  \end{align*}

  We now bound each of the terms separately. Let $\lambda_1,...,\lambda_d$ be the eigenvalues of $\mat B$.

  \begin{enumerate}
  \item For the first term, we have that
  $$
     | \tr( \mat B - \mat I ) - \log |\matr{B}| | = | \sum_{i=1}^d ( \lambda_i-1-\log{\lambda_i} ) | \le \sum_{i=1}^d \frac{( \lambda_i-1)^2}{\lambda_i} \le \frac{\|\mat B - \mat I\|^2_{F}}{\lambda_{min}} $$
  where we used the fact that $0 \le x-1-\log{x} \le \frac {(x-1)^2}{x}$ for all $x > 0$.
  \item For the second term, we have that $\lp| \tr( (\mat B - \mat I) (\wt{\S}_S - \mat I) ) \rp| \le \|\mat B - \mat I\|_{F} \|\wt{\S}_S - \mat I\|_{F} $
  \item For the third term, we have that $\lp|\vec u^T \mat B^{-1} \vec u \rp| = \vec u^T \mat B^{-1} \mat B \mat B^{-1} \vec u \le \lambda_{max} \|\mat B^{-1} \vec u\|_2^2$
  \end{enumerate}

  Now from the assumption $(\vec{u}, \mat{B}) \in \Domain$ we have that
  $\norm{\mat{B} - \mat{I}}_F \le O( \sqrt{\log(1/\alpha^*)} )$,
  $\norm{\mat{B}^{-1} \vec{u}}_2 \le O(\sqrt{\log(1/\alpha^*)})$, $\lambda_{min} \ge 15/16$
  and $\lambda_{max} \le 17/16$. Also from Lemma~\ref{lem:conditionalParameterDistance} we get that
  $\norm{\tiS_S - \mat{I}}_F \le O( \sqrt{  \log(1/\alpha^*) } )$ and hence $|2 \log{C_{\vec u, \mat B}}| \le O(\log(1/\alpha^*))$.
  This means that $C_{\vec u, \mat B} = \poly(1/\alpha)$ and the lemma follows.
\end{proof}

\subsubsection{The Objective Function and its Approximation}
\label{sec:objectivesDefinition}

To show that the minimizer of the function $M_{\tind_k}$ is a good estimator for
the unknown parameters $\vec{\mu}^*, \matr{\Sigma}^*$, we consider the function
$M'_{f}$, defined as $
M_f (\vec u, \mat B)
  =
  \E_{\vec x \sim \Nt_S}
  \lp[
  \me^{h'(\vec u, \mat B ; \vec x)}
  \normal(\vec 0, \mat I ; \vec x)
  f(\vec x)
  \rp]
$ for $
  h'(\vec u, \mat B ; \vec x) = \frac{\vec x^T \mat B \vec x}{2} -
  \frac{\tr( (\mat B - \mat I) (\S_S + {\m}_S {\m}_S^T) )}{2} -
  \vec u^T (\vec x-{\m}_S) + \frac d 2 \log{2\pi}
$.
This function corresponds to an ideal situation where we know the parameters
${\m}_S, \S_S$ exactly. Similarly to
\eqref{eq:optimizationFunctionEquivalentExpression}, we can write $M'_{f}$ as
$C'_{\vec u, \mat B}
    \E_{\vec x \sim \Nt_S}
    \lp[
    \frac{
      \N_0(\vec x)
    }{
      \N_{\vec u, \mat B}(\vec x)
    }
    f(\vec x)
    \rp]$.
We argue that both $M_f$ and $M'_f$ are convex.

\begin{claim} \label{clm:convex}
  For any function $f: \R^d \mapsto \R_{\ge 0}$, $M_f(\vec u, \mat B)$ and
  $M'_f(\vec u, \mat B)$ are convex functions of the parameters
  $(\vec u, \mat B)$.
\end{claim}
\begin{proof}
We show the statement for $M_f$. The proof for $M'_f$ is identical. The proof
follows by computing the Hessian of $M_f$ and arguing that it is positive
semidefinite.

The gradient with respect to $(\vec u, \mat B)$ is
\begin{align}
  \nabla M_f(\vec u, \mat B)
  & =
  \E_{\vec x \sim \normal(\vec \mu^*, \mat \Sigma^*, S)}
  \lp[
  \nabla h(\vec u, \mat B; \vec x)
  \me^{h(\vec u, \mat B; \vec x)}
  \normal(\vec 0, \mat I; \vec x)
  f(\vec x)
  \rp] \nonumber \\
  & =
  \E_{\vec x \sim \normal(\vec \mu^*, \mat \Sigma^*, S)}
  \lp[
  \begin{pmatrix}
    \frac12 \lp(\vec x \vec x^T - \tiS_S - {\tim}_S {\tim}_S^T \rp)^\flat
    \\
    {\tim}_S - \vec x
  \end{pmatrix}
  \me^{h(\vec u, \mat B; \vec x)}
  \normal(\vec 0, \mat I; \vec x)
  f(\vec x)
  \rp] \label{eq:objectiveGradientComputation}
\end{align}
Moreover, the Hessian is
\begin{align*}
  \mathcal{H}_{M_f} (\vec u, \mat B)
  =
  \E_{\vec x \sim \normal(\vec \mu^*, \mat \Sigma^*, S)}
  \lp[
  \begin{pmatrix}
    \frac12 \lp(\vec x \vec x^T - \tiS_S - {\tim}_S {\tim}_S^T \rp)^\flat
    \\
    {\tim}_S - \vec x
  \end{pmatrix}
  \begin{pmatrix}
    \frac12 \lp(\vec x \vec x^T - \tiS_S - {\tim}_S {\tim}_S^T \rp)^\flat
    \\
    {\tim}_S - \vec x
  \end{pmatrix}^T
  \me^{h(\vec u, \mat B; \vec x)}
  \normal(\vec 0, \mat I; \vec x)
  f(\vec x)
  \rp]
\end{align*}
which is clearly positive semidefinite since for any $\vec z \in \R^{d \times d
+ d}$ we have
\[
  \vec z^T
  \mathcal{H}_{M_f}(\vec u, \mat B)
  \vec z
  =
  \E_{\vec x \sim \normal(\vec \mu^*, \mat \Sigma^*, S)}
  \lp[
  \lp(\vec z^T
  \begin{pmatrix}
    \frac12 \lp(\vec x \vec x^T - \tiS_S - {\tim}_S {\tim}_S^T \rp)^\flat
    \\
    {\tim}_S - \vec x
  \end{pmatrix}
  \rp)^2
  \me^{h(\vec u, \mat B; \vec x)}
  \normal(\vec 0, \mat I; \vec x)
  f(\vec x)
  \rp]
  \geq 0.
\]
\end{proof}

\noindent We now argue that the minimizer of the convex function $M'_{\psi}$ for
the weighted characteristic function
$ \tind(\vec x) =
  \frac{\1{S}(\vec x)}{\alpha^*}
  \frac{\normal(\vec{\mu}^*, \mat{\Sigma}^*; \vec x)}{\normal(\vec{0},\matr{I}; \vec x)}
$
is $(\vec u, \mat B) = ({\St}^{-1}, {\St}^{-1} \mt)$.

\begin{claim} \label{clm:consistency}
    The minimizer of $M'_{\tind}(\vec u, \mat B)$ is $(\vec u, \mat B) =
  ({\St}^{-1}, {\St}^{-1} \mt)$.
\end{claim}
\begin{proof}
  The gradient of $M'_{\tind}$ with respect to $(\vec u, \mat B)$ is
  \begin{align*}
    \nabla M'_{\tind}(\vec u, \mat B)
     &=
    \E_{\vec x \sim \Nt_S}
    \lp[
    \begin{pmatrix}
      \frac12 \lp(\vec x \vec x^T - \S_S - {\m}_S {\m}_S^T \rp)^\flat
      \\
      {\m}_S - \vec x
    \end{pmatrix}
    \me^{h(\vec u, \mat B; \vec x)}
    \normal(\vec 0, \mat I; \vec x)
    \frac{\1{S}(\vec x)}{\alpha^*}
    \frac{\normal(\vec{\mu}^*, \mat{\Sigma}^*; \vec x)}{\normal(\vec{0},\matr{I}; \vec x)}
    \rp]\\
     &=
    \E_{\vec x \sim \Nt_S}
    \lp[
    \begin{pmatrix}
      \frac12 \lp(\vec x \vec x^T - \S_S - {\m}_S {\m}_S^T \rp)^\flat
      \\
      {\m}_S - \vec x
    \end{pmatrix}
    \me^{h(\vec u, \mat B; \vec x)}
    \frac{\normal(\vec{\mu}^*, \mat{\Sigma}^*; \vec x)}{\alpha^*}
    \rp]
  \end{align*}
  For $(\vec u, \mat B) = ({\St}^{-1} \mt, {\St}^{-1})$, this is equal to
  \begin{align*}
  \nabla M'_{\tind}({\St}^{-1} \mt, {\St}^{-1}) & =
  C_{\vec u, \mat B} \cdot \E_{\vec x \sim \Nt_S}
  \lp[
  \begin{pmatrix}
    \frac12 \lp(\vec x \vec x^T - \S_S - {\m}_S {\m}_S^T \rp)^\flat
    \\
    {\m}_S - \vec x
  \end{pmatrix}
  \frac{1}{\normal(\vec{\mu}^*, \mat{\Sigma}^*; \vec x)}
  \frac{\normal(\vec{\mu}^*, \mat{\Sigma}^*; \vec x)}{\alpha^*}
  \rp] \\
  & = \frac{C_{\vec u, \mat B}}{\alpha^*}
   \cdot \E_{\vec x \sim \Nt_S}
  \lp[
  \begin{pmatrix}
    \frac12 \lp(\vec x \vec x^T - \S_S - {\m}_S {\m}_S^T \rp)^\flat
    \\
    {\m}_S - \vec x
  \end{pmatrix}
  \rp]
  \end{align*}
  where $C_{\vec u, \mat B}$ that does not depend on $x$. This is equal to 0 by
  definition of ${\m}_S$ and $\S_S$.
\end{proof}

We want to show that the minimizer of $M_{\tind_k}$ is close to that of
$M'_{\tind}$. To do this, we bound the difference of the two functions
pointwise.  The proof of the following lemma is technical and can be
found in Appendix~\ref{app:optimization}.
\begin{lemma}[\textsc{Pointwise Approximation of the Objective Function}]
\label{lem:pointwiseApproximationOfTheObjectiveFunction}
  Assume that we use Lemma \ref{lem:conditionalEstimation} to estimate
  $\tim_S, \tiS_S$ with $\eps = \frac{1}{\poly(1/\alpha^*)} \eps'$
  and Theorem \ref{thm:estimationApproximationError}
  with $\eps = \frac{1}{p(1/\alpha^*)} \eps'^2$ then
  \[ \lp|
    M_{\tind_k}(\vec u, \mat B) -
    M'_{\tind}(\vec u, \mat B)
    \rp|
    \leq \eps'.
  \]
\end{lemma}

  Now that we have established that $M_{\tind_k}$ is a good approximation of
$M'_{\tind}$ we will prove that we can optimize $M_{\tind_k}$ and get a solution
that is very close to the optimal solution of $M'_{\tind}$.

\subsubsection{Optimization of the Approximate Objective Function}
\label{sec:approxOptimization}

  Our goal in this section is to prove that using sample access to
$\N(\mt, \St, S)$ we can find the minimum of the function $M_{\tind_k}$ defined
in the previous section. First of all recall that $M_{\tind_k}$ can be written
as an expectation over $\N(\mt, \St, S)$ in the following way
\begin{align*}
M_{\tind_k} (\vec u, \mat B)
  \triangleq
  \E_{\vec x \sim \Nt_S}
  \lp[
  \me^{h(\vec u, \mat B ; \vec x)}
  \normal(\vec 0, \mat I ; \vec x)
  \tind_k(\vec x)
  \rp].
\end{align*}
\noindent In Section \ref{sec:weighted_learning} we prove that we can learn the
function $\tind_k$ and hence $M_{\tind_k}$ can be written as
\[ M_{\tind_k}(\vec{u}, \mat{B}) = \Exp_{\vec{x} \sim \Nt_S}\left[
     m_{\tind_k}(\vec{u}, \mat{B}; x) \right] \]
\noindent where
$m_{\tind_k}(\vec{u}, \mat{B}; x) = \me^{h(\vec u, \mat B ; \vec x)}
\normal(\vec 0, \mat I ; \vec x)
\tind_k(\vec x)$, and for any $\vec{u}, \matr{B}$ and $\vec{x}$ we can compute
$m_{\tind_k}(\vec{u}, \mat{B}; x)$. Since $M_{\tind_k}$ is convex we are going
to use stochastic gradient descent to find its minimum. To prove the convergence
of SGD and bound the number of steps that SGD needs to converge we will use the
the formulation developed in Chapter 14 of \cite{ShalevB14}. To be able to use
their results we have to define for any $(\vec{u}, \mat{B})$ a random vector
$\vec{v}(\vec{u}, \mat{B})$ and prove the following
\begin{description}
  \item[\textsc{Unbiased Gradient Estimation}] \[ \Exp\left[ \vec{v} (\vec{u},
  \mat{B}) \right] = \nabla M_{\tind_k}, \]
  \item[\textsc{Bounded Step Variance}] \[ \Exp\left[ \norm{\vec{v}(\vec{u},
  \mat{B})}_2^2 \right] \le \rho, \]
  \item[\textsc{Strong Convexity}] for any $\vec{z} \in \Domain$ it holds \[
  \vec{z}^T \mathcal{H}_{M_f}(\vec{u}, \matr{B}) \vec{z} \ge \lambda. \]
\end{description}

\noindent We start with the definition of the random vector $\vec{v}$. Given a
sample $\vec{x}$ from $\normal(\mt, \St, S)$, for any $(\vec{u}, \mat{B})$ we
define
\begin{align}
  \vec{v}(\vec{u}, \mat{B}) & = \nabla_{\vec{u}, \mat{B}} ~ m_{\tind_k}(\vec{u},
  \vec{B}; \vec{x}) \label{eq:unbiasedGradientEstimationDefinition1}
  \\ & =
  \begin{pmatrix}
    \frac12 \lp(\vec x \vec x^T - \tiS_S - \tim_S \tim_S^T \rp)^\flat
    \\
    \tim_S - \vec x
  \end{pmatrix}
  \me^{h(\vec u, \mat B; \vec x)}
  \normal(\vec 0, \mat I; \vec x)
  \tind_k(\vec x) \label{eq:unbiasedGradientEstimationDefinition2}
\end{align}
\noindent observe that the randomness of $\vec{v}$ only comes from the random
sample $\vec{x} \sim \normal(\mt, \St, S)$. The fact that
$\vec{v}(\vec{u}, \mat{B})$ is an unbiased estimator of
$\nabla M_f(\vec{u}, \mat{B})$ follows directly from the fact calculation of
$\nabla M_f(\vec{u}, \mat{B})$ in Section \ref{sec:objectivesDefinition}. For
the other two properties that we need we have the following lemmas.
The following lemma bounds the variance of the step of the SGD algorithm.
It's rather technical proof can be found in Appendix~\ref{app:optimization}.
\begin{lemma}[\textsc{Bounded Step Variance}] \label{lem:boundedVarianceStep}
Let $\alpha$ be the constant of \eqref{eq:globalLowerBound}.
  For every $(\vec{u}, \mat{B}) \in \Domain$ it holds
  \[ \Exp_{\vec{x} \sim \Nt_S} \left[
       \norm{ \vec{v}(\vec{u}, \mat{B}) }_2^2
     \right] \le
  \poly(1/\alpha) \cdot d^{2 k}, \]
\end{lemma}

  We are now going to prove the strong convexity of the objective function
  $M_{\tind_k}$. For this we are going to use a known anti-concentration result
  (Theorem~\ref{thm:GaussianMeasurePolynomialThresholdFunctions}) for polynomial
  functions over the Gaussian measure.  See Appendix~\ref{app:prelims}.

The following lemma shows that our objective is strongly convex
as long as the guess $\vec u, \vec B$ remains in the set $\mcal{D}$.
Its proof can be found in Appendix~\ref{app:optimization}.
\begin{lemma}[\textsc{Strong Convexity}] \label{lem:strongConvexity}
  Let $\alpha$ be the absolute constant of \eqref{eq:globalLowerBound}.
  For every $(\vec{u}, \mat{B}) \in \Domain$, any $\vec{z} \in \R^d$ such that
  $\norm{\vec{z}}_2 = 1$ and the first $d^2$ coordinated of $\vec{z}$ correspond
  to a symmetric matrix, then
  \[
    \vec{z}^T \mathcal{H}_{M_f}(\vec{u}, \matr{B}) \vec{z} \ge \poly(\alpha),
  \]
\end{lemma}

\subsubsection{Recovering the Unconditional Mean and Covariance}

  The framework that we use for proving the fast convergence of our SGD
algorithm is summarized in the following theorem and the following lemma.

\begin{theorem}[Theorem 14.11 of \cite{ShalevB14}.] \label{thm:mainSGDAnalysis}
  Let $f : \R^{d} \to \R$. Assume that $f$ is $\lambda$-strongly convex,
that $\Exp\left[ \vec{v}^{(i)} \mid \vec{w}^{(i - 1)}\right] \in
\partial f(\vec{w}^{(i - 1)})$ and that
$\Exp\left[ \norm{\vec{v}^{(i)}}_2^2 \right] \le \rho^2$. Let
$\vec{w}^* \in \arg \min_{\vec{w} \in \Domain} f(\vec{w})$ be an optimal
solution. Then,
\[ \Exp\left[ f(\vec{\bar{w}}) \right] - f(\vec{w}^*) \le
   \frac{\rho^2}{2 \lambda T}\left( 1 + \log T \right), \]
\noindent where $\vec{\bar{\vec{w}}}$ is the output projected stochastic
gradient descent with steps $\vec{v}^{(i)}$ and projection set $\Domain$ after
$T$ iterations.
\end{theorem}

\begin{lemma}[Lemma 13.5 of \cite{ShalevB14}.] \label{lem:strongConvexityToDistance}
    If $f$ is $\lambda$-strongly convex and $\vec{w}^*$ is a minimizer of $f$,
  then, for any $\vec{w}$ it holds that
  \[ f(\vec{w}) - f(\vec{w}^*) \ge \frac{\lambda}{2} \norm{\vec{w} - \vec{w}^*}_2^2. \]
\end{lemma}

Now we have all the ingredients to present the proof of Theorem~\ref{thm:mainTheoremLocalConvergence}.

\begin{prevproof}{Theorem}{thm:mainTheoremLocalConvergence}
    The estimation procedure starts by computing the polynomial function
  $\tind_k$ using $d^{\poly(1/\alpha^*) \frac{\Gamma^2(S)}{\eps'^8}}$ samples
  from $\N(\mt, \St, S)$ as explained in Theorem
  \ref{thm:estimationApproximationError} to get error $\poly(\alpha^*) \eps'^2$.
  Then we compute $\tim_S$ and $\tiS_S$ as explained in Section
  \ref{sec:problemFormulation} with
  $\eps = \frac{q(\alpha^*)}{8 p(1/\alpha^*)} (\eps')^2$ where $p$ comes from
  Lemma \ref{lem:pointwiseApproximationOfTheObjectiveFunction} and $q$ comes
  from Lemma \ref{lem:strongConvexity}. Our estimators for
  $\hat{\m}, \hat{\S}$ are the outputs of Algorithm \ref{alg:projectedSGD}.

  We analyze the accuracy of our estimation by proving that the
  minimum of $M_{\tind_k}$ is close in the parameter space to the minimum of
  $M'_{\tind}$. Let $\vec{u}', \mat{B}'$ be the minimum of the convex function
  $M'_{\tind}$ and $\vec{u}_k, \mat{B}_k$ be the minimum of the convex function
  $M_{\tind_k}$. Using Lemma
  \ref{lem:pointwiseApproximationOfTheObjectiveFunction} we have the following
  relations
  \[ \abs{M'_{\tind}(\vec{u}', \mat{B}') - M_{\tind_k}(\vec{u}', \mat{B}')} \le \eps',
     ~~~~~~~ \abs{M'_{\tind}(\vec{u}_k, \mat{B}_k) - M_{\tind_k}(\vec{u}_k, \mat{B}_k)} \le \eps' \]
  \noindent and also
  \[ M'_{\tind}(\vec{u}', \mat{B}') \le M'_{\tind}(\vec{u}_k, \mat{B}_k),
     ~~~~~~~ M_{\tind_k}(\vec{u}_k, \mat{B}_k) \le M_{\tind_k}(\vec{u}', \mat{B}'). \]
  \noindent These relations imply that
  \begin{align*}
    \abs{M_{\tind_k}(\vec{u}', \mat{B}') - M_{\tind_k}(\vec{u}_k, \mat{B}_k)}
    & = M_{\tind_k}(\vec{u}', \mat{B}') - M_{\tind_k}(\vec{u}_k, \mat{B}_k) \\
    & \le M_{\tind_k}(\vec{u}', \mat{B}') - M'_{\tind}(\vec{u}', \mat{B}')
      + M'_{\tind}(\vec{u}_k, \mat{B}_k) - M_{\tind_k}(\vec{u}_k, \mat{B}_k) \\
    & \le \abs{M'_{\tind}(\vec{u}', \mat{B}') - M_{\tind_k}(\vec{u}', \mat{B}')} +
        \abs{M'_{\tind}(\vec{u}_k, \mat{B}_k) - M_{\tind_k}(\vec{u}_k, \mat{B}_k)} \le 2 \eps'.
  \end{align*}
  \noindent But from Lemma \ref{lem:strongConvexity} and Lemma
  \ref{lem:strongConvexityToDistance} we get that
  $
    \norm{\begin{pmatrix}
            \mat{B}'^{\flat} \\
            \vec{u}'
          \end{pmatrix} -
          \begin{pmatrix}
            \mat{B}_k^{\flat} \\
            \vec{u}_k
          \end{pmatrix}}_2 \le \frac{\eps'}{2}
  $.
  Now we can apply the Claim \ref{clm:consistency} which implies that
  \begin{align} \label{eq:distanceToTheOptimum}
    \norm{\begin{pmatrix}
            (\S^{*-1})^{\flat} \\
            \S^{*-1} \mt
          \end{pmatrix} -
          \begin{pmatrix}
            \mat{B}_k^{\flat} \\
            \vec{u}_k
          \end{pmatrix}}_2 \le \frac{\eps'}{2}.
  \end{align}
  \noindent Therefore it suffices to find $(\vec{u}_k, \mat{B}_k)$ with accuracy
  $\eps'/2$ to get our theorem.

    Let $\vec{w}^* = \begin{pmatrix}
      \mat{B}_k^{\flat} \\
      \vec{u}_k
    \end{pmatrix}$
  To prove that Algorithm \ref{alg:projectedSGD} converges to  $\vec{w}^*$ we
  use Theorem \ref{thm:mainSGDAnalysis} which together with Markov's inequality,
  Lemma \ref{lem:boundedVarianceStep} and Lemma \ref{lem:strongConvexity} gives
  us
  \begin{equation} \label{eq:MarkovFinalBound}
    \Prob \left( M_{\tind_k}(\hat{\vec{u}}, \hat{\mat{B}}) - M_{\tind_k}(\vec{u}_k, \mat{B}_k)
                 \ge \poly(1/\alpha^*) \cdot \frac{d^{2k}}{T}
    \left( 1 + \log(T) \right) \right) \le \frac{1}{3}.
  \end{equation}

  To get our estimation we first repeat the SGD procedure $K = \log(1/\delta)$
  times independently, with parameters $T, \lambda$ each time. We then get the
  set of estimates $\mathcal{E} = \{\bar{\vec{w}}_1, \bar{\vec{w}}_2, \dots,
  \bar{\vec{w}}_K\}$. Because of \eqref{eq:MarkovFinalBound} we know that, with
  high probability $1 - \delta$, for at least the 2/3 of the points
  $\bar{\vec{w}}$ in $\mathcal{E}$ it is true that
  $M_{\tind_k}(\vec{w}) - M_{\tind_k}(\vec{w}^*) \le \eta$ where
  $\eta = \poly(1/\alpha^*) \cdot \frac{d^{2k}}{T}
          \left( 1 + \log(T) \right)$. Moreover we will prove later that
  $M_{\tind_k}(\vec{w}) - M_{\tind_k}(\vec{w}^*) \le \eta$ and this implies
  $\norm{\vec{w} - \vec{w}^*} \le c \cdot \eta$, where $c$ is a universal
  constant. Therefore with high probability $1 - \delta$ for at least the 2/3 of
  the points $\bar{\vec{w}}, \bar{\vec{w}'}$ in $\mathcal{E}$ it is true that
  $\norm{\vec{w} - \vec{w}'} \le 2 c \cdot \eta$. Hence if we set
  $\hat{\vec{w}}$ to be a point that is at least $2 c \cdot \eta$ close to more
  that the half of the points in $\mathcal{E}$ then with high probability
  $1 - \delta$ we have that $f(\vec{\bar{w}}) - f(\vec{w}^*) \le \eta$. Hence we
  can we lose probability at most $\delta$ if we condition on the event
  \[ M_{\tind_k}(\hat{\vec{u}}, \hat{\mat{B}}) - M_{\tind_k}(\vec{u}_k, \mat{B}_k)
               \le \poly(1/\alpha^*) \cdot \frac{d^{2k}}{T}
               \left( 1 + \log(T) \right). \]
  \noindent Using once again Lemma \ref{lem:strongConvexityToDistance} we get
  that
  \begin{align*}
    \norm{\begin{pmatrix}
            \hat{\mat{B}}^{\flat} \\
            \hat{\vec{u}}
          \end{pmatrix} -
          \begin{pmatrix}
            \mat{B}_k^{\flat} \\
            \vec{u}_k
          \end{pmatrix}}_2 \le \frac{\eps'}{2}.
  \end{align*}
  \noindent which together with \eqref{eq:distanceToTheOptimum} implies
  \begin{align*}
    \norm{\begin{pmatrix}
            \hat{\mat{B}}^{\flat} \\
            \hat{\vec{u}}
          \end{pmatrix} -
          \begin{pmatrix}
            (\S^{*-1})^{\flat} \\
            \S^{*-1} \mt
          \end{pmatrix}}_2 \le \frac{\eps'}{2}.
  \end{align*}
  \noindent and the theorem follows as closeness in parameter distance implies closeness in total variation distance for the corresponding
  untruncated Gaussian distributions.
\end{prevproof}

  \begin{algorithm}[H]
  \caption{Projected Stochastic Gradient Descent. Given access to samples from $\normal(\vec{\mu}^*, \matr{\Sigma}^*, S)$.}
  \label{alg:projectedSGD}
  \begin{algorithmic}[1]
    \Procedure{Sgd}{$T, \lambda$}\Comment{$T$: number of steps, $\lambda$: parameter.}
    \State $\vec{w}^{(0)} = \begin{pmatrix}
                              (\mat{B}^{(0)})^{\flat} \\
                              \vec{u}^{(0)}
                            \end{pmatrix}
                            \gets
                            \begin{pmatrix}
                              (\tiS_S^{-1})^{\flat} \\
                              \tiS_S^{-1} \tim_S
                            \end{pmatrix}$
    \For{$i = 1, \dots, T$}
      \State Sample $\vec{x}^{(i)}$ from $\N(\mt, \St, S)$
      \State $\eta_i \gets \frac{1}{\lambda \cdot i}$
      \State $\begin{pmatrix}
                (\matr{B}^{(i - 1)})^{\flat} \\
                \vec{u}^{(i - 1)}
              \end{pmatrix}
              \gets \vec{w}^{(i - 1)}$
      \State $\vec{v}^{(i)} \gets
             \begin{pmatrix}
               \frac12 \lp(\vec x^{(i)} \vec x^{(i) T} - \tiS_S - {\tim}_S {\tim}_S^T \rp)^\flat \\
               {\tim}_S - \vec x^{(i)}
             \end{pmatrix}
             \me^{h(\vec u^{(i - 1)}, \mat B^{(i - 1)}; \vec x^{(i)})}
             \normal(\vec 0, \mat I; \vec x^{(i)}) \tind_k\left(\vec{x}^{(i)}\right)$ \Comment{From \eqref{eq:objectiveGradientComputation}.}
      \State $\vec{r}^{(i)} \gets \vec{w}^{(i - 1)} - \eta_i \vec{v}^{(i)}$
      \State $\vec{w}^{(i)} \gets \arg\min_{\vec{w} \in \Domain} \norm{\vec{w} - \vec{r}^{(i)}}_2^2$ \Comment{From Lemma 8 of \cite{DGTZ18}.}
    \EndFor\label{euclidendwhile}
    \State $\begin{pmatrix}
              \hat{\matr{B}}^{\flat} \\
              \hat{\vec{u}}
            \end{pmatrix}
            \gets \frac{1}{T} \sum_{i = 1}^T \vec{w}^{(i)}$
    \State $\hat{\S} \gets \hat{\matr{B}}^{-1}$
    \State $\hat{\m} \gets \hat{\matr{B}}^{-1} \hat{\vec{u}}$
    \State \textbf{return} $(\hat{\m}, \hat{\S})$
  \EndProcedure
  \end{algorithmic}
  \end{algorithm}

 \subsection{Recovering the Set}\label{sec:set_recover}
In this section we prove that, given only positive examples from
an unknown truncated Gaussian distribution, that is samples
from the conditional distribution on the truncation set, 
one can in fact learn the truncation set.
We only give here the main result, for details see 
Appendix~\ref{app:setRecover}.
\begin{theorem}[\textsc{Recovering the Set}]\label{thm:pacLearning}
  Let $\mcal{S}$ be a class of measurable sets with Gaussian surface area at most
  $\Gamma(\mcal{S})$.
  Let $\normal^*$ be a Gaussian in $(O(\log(1/\alpha), 1/16))$-isotropic position.
  Then, given $d^{\poly(1/\alpha)\Gamma(\mcal{S})^2 /\eps^{32}}$ samples
  from the conditional distribution $\normal^*_S$ we can recover an  indicator
  of the set $\wt{S}$ such that with probability at least $99\%$ it holds
  \(
  \Prob_{\vec x \sim \normal^*}[\wt{S}(\vec x) \neq \1{S}(\vec x)] \leq \eps.
  \)
\end{theorem}
 \section{Lower Bound for Learning the Mean of a Truncated Normal}

\begin{theorem}\label{thm:mean_lower_bound}
  There exists a family of sets $\mcal{S}$ with $\Gamma({\mcal{S}}) = O(d)$ such
  that any algorithm that draws $m$ samples from $\normal(\vec \mu, \matr I, S)$
  and computes an estimate $\wt {\vec \mu}$ with
  $\norm{\wt {\vec \mu} - \vec \mu}_2 \leq 1$ must have $m= \Omega(2^{d/2})$.
\end{theorem}
\begin{proof}
Let $H = [-1, 1]^{d+1}$ be the $d+1$-dimensional cube.
We will also use the left and right subcubes $H_+ = [-1,0] \times [-1, 1]^d$,
$H_- = [0,1] \times [-1, 1]^d$ respectively.
Let $\normal_+ = \normal(\vec e_1, \mat I)$ and
$\normal_- = \normal(\vec - e_1, \mat I)$.
We denote by $r$ the (scaled) pointwise minimum of the two densities truncated
at the cube $H$, that is
\[
  r(x) = \frac{\min(\normal_+(H; \vec x), \normal_-(H; \vec x))}{c}
  = \frac{\1{H}(\vec x)}{c} \min(\normal_+(\vec x), \normal_-(\vec x)),
\]
where $c = 1-\dtv{\normal_+}{\normal_-}$.

To simplify notation we assume that we work in $\R^{d+1}$ instead of $\R^d$.
Let $V = (v_1,\ldots, v_d) \in \{+1, -1\}^d$.  For every $V$ we define the set
$G_V = H \cap \{\vec y \in \R^d: y_i v_i \geq 0\}$.  We also
define the subcubes $H_V = [0,1] \times G_V$. We consider the following
subset of $H$ parameterized by the $2^d$ parameters $t_V \in [0,1]$ and
$\delta \in [-1,1]$.
\[
  S_+ =
  [-1 + \delta, 0] \times [-1,1]^d \cup \bigcup_{V \in \{-1, +1\}^d} [0,t_V] \times G_V
\]
We will argue that there exists a distribution $D^+$ on the values $t_V$ such
that on expectation $\dtv{\normal_+^{S_+}}{\normal_-^{S_-}} $ is $O(2^{-d})$.
We show how to construct the distribution $D_+$ since the construction for $D_-$
is the same.
In fact we will show that both distributions are very close to $r(x)$.
Notice that for some $(t, \vec y) \in \R^{d+1}$ we have
We draw each $t_V$ independently from the distribution with cdf
\[
  F(t) = \1{[0,1)}(t) (1 - \me^{-2 t}) + \1{[1,+\infty)}(t)
\]
Notice that for $t\in (0,1)$ and any $\vec y \in \R^d$ we have that
$1-F(t) = \normal_-(t, \vec y)/\normal_+(t, \vec y)$.

After we draw all $t_V$ from $F$ we choose $\delta$ so that $\normal_+(S_+; x) = c$.
We will show that on expectation over the $t_V$ we have $\delta = 0$, which
means that no correction is needed.  In fact we show something stronger,
namely that for all $x \in H_+$ we have that
\( \E_{S_+ \sim D_+}[N_+(S_+; \vec x)] = r(\vec x) \).
Assume that $x \in H_V$.  Indeed,
\begin{align*}
  \E_{S_+ \sim D_+}[ \normal_+(S_+; \vec x)]
  &= \frac{\normal_+(\vec x)}{c} \E_{S_+ \sim D_+}[ \1{S_+}(\vec x)]
  = \frac{\normal_+(\vec x)}{c} \E_{S_+ \sim D_+}[ \1{\{ x_1 \leq t_V \}}]
  \\
  &= \frac{\normal_+(\vec x)}{c} (1- F(t_V)) =  \frac{\normal_-(\vec x)}{c}
  = r(\vec x)
\end{align*}
Moreover, observe that for all $\vec x \in H_- \cap S_+$ we have that
$N_+(S_+;\vec x) = r(\vec x)$ always (with probability $1$).  We now argue that
in order to have constant probability to distinguish $N_+(S_+)$ from $r(x)$ one
needs to draw $\Omega(2^d)$ samples.  Since the expected density of $N_+(S_+)$
matches $r(x)$ for all $x \in H_+$, to be able to distinguish the two
distributions one needs to observe at least two samples in the same cube $H_V$.
Since we have $2^d$ disjoint cubes $H_V$ the probability of a sample landing in
$H_V$ is at most $1/2^d$.  Therefore, using the birthday problem, to have
constant probability to observe a collision one needs to draw
$\Omega(\sqrt{2^d}) = \Omega(2^{d/2})$ samples.  Since for all $x \in H_-
\cap S_+$, $N_+(S_+)$ exactly matches $r(x)$, to distinguish between the two
distributions one needs to observe a sample $\vec x$ with $-1+\delta < x_1 <
-1$. Due to symmetry, $N_+$ assigns to all cubes $H_V$ equal probability, call
that $p$.  Moreover, we have that $c = 2^{d+1} p$.  Now let $p_V$ be the random
variable corresponding to the probability that $N_+$ assigns to
$[0,t_V] \times G_V$.  We have that $\E_{t_V \sim F}[p_V] = p$ for all $V$.
Since the independent random variables $p_V$ are bounded in $[0,1/2^d]$,
Hoeffding's inequality implies that $|\sum_{V \in \{-1, 1\}^d} (p_V - p)| <
1/2^{d/2}$ with probability at least $1 - 2/\me^2$.  This means that with
probability at least $3/4$ one will need to draw $\Omega(2^{d/2})$ samples
in order to observe one with $x_1 < -1 + \delta$.

Since any set $S$ in our family $\mcal{S}$ has almost everywhere (that is
except the set of its vertices which a finite set and thus of measure zero)
smooth boundary we may use the following equivalent (see e.g. \cite{Naz03})
definition of its surface area
\[
  \Gamma(S) = \int_{\partial S} \normal_0(\vec x) \d \sigma(\vec x),
\]
where $\d \sigma(x)$ is the standard surface measure on $\R^d$.
Without loss of generality we assume that $S$ corresponds to the
set $S_+$ defined above (the proof is the same if we consider
a set $S_-$).
We have
\[
  \partial S \subseteq
  \bigcup_{V \in \{+1,-1\}^d} (\{t_V\} \times G_V )
  \cup
  \partial([-1, +1]^{d+1})
  \cup \bigcup_{i=1}^{d+1} \{\vec x: x_i = 0 \}.
\]
By the definition of Gaussian surface area it is clear that $\Gamma(A \cup B)
\leq \Gamma(A) + \Gamma(B)$.  From Table~\ref{tab:surface-area} we know that
$\Gamma([-1, +1]^{d+1}) = O(\sqrt{\log d})$.  Moreover, we know that a single
halfspace has surface area at most $\sqrt{2/\pi}$ (see e.g. \cite{KOS08}).
Therefore $\Gamma\lp(\bigcup_{i=1}^{d+1}\{x: x_i = 0\}\rp) \leq
\sum_{i=1}^{d+1} \sqrt{2/\pi} = O(d)$.  Finally, we notice that for any point
$x$ on the hyperplane $\{\vec x: x_1 = 0\}$ and any $\vec y$ on $\{\vec
x: x_1 = c \}$ (for any $c \geq 0$), we have $\normal_0(\vec x) \geq
\normal_0(\vec y)$.  Therefore, the surface area of each set
$t_V \times G_V$ is maximized for $t_V = 0$.  In this case $\bigcup_{V \in \{+1,-1\}^d}
(\{t_V\} \times G_V ) \subseteq \{\vec x : x_1 = 0\}$, which implies that the
set $\bigcup_{V \in \{+1,-1\}^d} (\{t_V\} \times G_V )$ contributes at most
$\sqrt{2/\pi}$ to the total surface area.  Putting everything together, we have
that $\Gamma(S) = O(d)$.

\begin{figure}
  \caption{The set $S_+$ when $d=1$.}
  \centering
  \begin{tikzpicture}[scale=2]
    \draw[-](-1,1) -- (1,1);
    \draw[-](-1,-1) -- (1,-1);
    \draw[-](-1,-1) -- (-1,1);
    \draw[-](1,-1) -- (1,1);
    \draw[-](0, -1) -- (0,1);
    \draw[-](-1, 0) -- (1,0);
    \draw[-](1/3, 1) -- (1/3,0);
    \draw[-](7/9, 0) -- (7/9,-1);
    \draw[draw=none,fill=gray,opacity=0.4] (-1,1) -- (1/3, 1) -- (1/3, 0) -- (7/9,0) -- (7/9,-1) -- (-1,-1);
    \draw (7/9,-0.9) node [label=below:\small$t_{-1}$]{};
    \draw (1/3,0.9) node [label=above:\small $t_{+1}$]{};
    \draw (1/2,0.29) node [label=$H_{+1}$]{};
    \draw (1/2,-0.71) node [label=$H_{-1}$]{};
  \end{tikzpicture}
\end{figure}
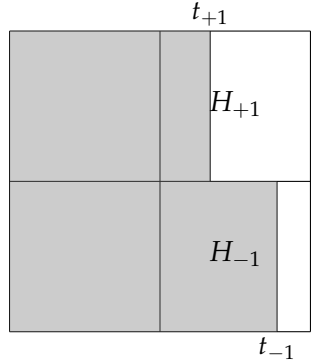
\end{proof}
 \section{Identifiability with bounded Gaussian Surface Area}
\label{sec:IdentifiabilityGaussian}

In this section we investigate the sample complexity of the problem of
estimating the parameters of a truncated Gaussian using a different approach
that does not depend on the VC dimension of the family $\mcal{S}$ of the
truncation sets to be finite.  For example, we settle the sample complexity of
learning the parameters of a Gaussian distribution truncated at an unknown
convex set (recall that the class of convex sets has infinite VC dimension).
Our method relies on finding a tuple $(\wb \mu, \wb \Sigma, \wt S)$ of
parameters so that the moments of the corresponding truncated Gaussian
$\normal(\wb \mu, \wb \Sigma, \wt S)$ are all close to the moments of the
unknown truncated Gaussian distribution, for which we have unbiased estimates
using samples.  The main question that we need to answer to determine the
sample complexity of this problem is how many moments are needed to be matched
in order to be sure that our guessed parameters are close to the parameters of
the unknown truncated Gaussian.  We state now the main result.
Its proof is based on Lemma~\ref{lem:TVDPolynomial} and can be found in
Appendix~\ref{app:IdentifiabilityGaussian}.

\begin{theorem}[Moment Matching] \label{thm:momentMatching}
  Let $\mcal{S}$ be a family of subsets of $\R^d$ of bounded Gaussian surface
  area $\Gamma(\mcal{S})$.  Moreover, assume that if $T$ is an affine map and
  $T(\mcal{S}) = \{ T(S): S\in\mcal{S} \}$ is the family of the images of the
  sets of $\mcal{S}$, then it holds $\Gamma(T(\mcal{S})) =
  O(\Gamma(\mcal{S}))$.  For some $S \in \mcal{S}$, let $\normal(\vec \mu,
  \matr \Sigma, S)$ be an unknown truncated Gaussian.
  $d^{O(\Gamma(\mcal{S})/\eps^4)}$ samples are sufficient to find parameters
  $\wt{\vec \mu}, \wt {\matr \Sigma}, \wt S$ such that $\dtv{\normal(\vec \mu,
  \matr \Sigma, S)}{\normal(\wt{\vec \mu}, \wt {\matr \Sigma}, \wt S)} \leq
  \eps$.
\end{theorem}

The key lemma of this section is Lemma~\ref{lem:TVDPolynomial}. It shows that
if two truncated normals are in total variation distance $\eps$ then there
exists a moment where they differ.   The main idea is to prove that there
exists a polynomial that approximates well the indicator of the set $\{f_1 >
f_2\}$.  Notice that the total variation distance between two densities can be
written as $\int \1{\{f_1 > f_2\}}(\vec x) f_1(\vec x) - f_2(\vec x)\d x$.  In
our proof we use the chi squared divergence, which for two distributions with
densities $f_1, f_2$ is defined as
\[
  \dchi{f_1}{f_2} = \int \frac{(f_1(\vec x) - f_2(\vec x))^2}{f_2(\vec x)} \d x
\]
To prove it we need the following nice
fact about chi squared divergence between Gaussian distributions.
In general chi squared divergence may be infinite for some pairs of Gaussians.
In the following lemma we prove that for any pair of Gaussians, there
exists another Gaussian $N$ such that $\dchi{N_1}{N}$ $\dchi{N_2}{N}$ are
finite even if $\dchi{N_1}{N_2} = \infty$.
\begin{lemma}\label{lem:chiSquaredExistence}
  Let $N_1 = \normal(\vec \mu_1, \matr \Sigma_1)$, and $N_2 = \normal(\vec
  \mu_1, \matr \Sigma_2)$ be two Normal distributions that satisfy the
  conditions of Lemma~\ref{lem:conditionalParameterDistance}.
  Then there exists
  a Normal distribution $N$ such that
  \[
    \dchi{N_1}{N}, \dchi{N_2}{N} \leq
    \exp\lp(2 \norm{\matr \Sigma_1^{-1/2}(\vec \mu_1 - \vec \mu_2)}_2
    + \frac 12 \max(1, \norm{\matr \Sigma_1}_2) \norm{\matr \Sigma_1^{-1/2} \matr \Sigma_2 \matr\Sigma_1^{-1/2}
    - \matr I}_F^2
    \rp)
  \]
\end{lemma}
Now we state the main lemma of this section.  We give here a sketch of its
proof.  It's full version can be found in
Appendix~\ref{app:IdentifiabilityGaussian}.
\begin{lemma} \label{lem:TVDPolynomial}
  Let $\mcal{S}$ be a family of subsets of $\R^d$ of bounded Gaussian surface
  area $\Gamma(\mcal{S})$.  Moreover, assume that if $T$ is an affine map and
  $T(\mcal{S}) = \{ T(S): S\in\mcal{S} \}$ is the family of the images of the
  sets of $\mcal{S}$, then it holds $\Gamma(T(\mcal{S})) =
  O(\Gamma(\mcal{S}))$.  Let $\normal(\vec \mu_1, \matr \Sigma_1, S_1)$ and
  $\normal(\vec \mu_2, \matr \Sigma_2, S_2)$ be two truncated Gaussians with
  densities $f_1, f_2$ respectively.  Let $ k = O(\Gamma(\mcal{S})/\eps^4)$.
  If $\dtv{f_1}{f_2} \geq \eps$, then there exists a $V\in \nats^d$ with $|V|
  \leq k$ such that
  \[
    \lp| \E_{\vec x \sim \normal(\vec \mu_1, \matr \Sigma_1, S_1) }
    [ \vec x^V]
    -
    \E_{\vec x \sim \normal(\vec \mu_2, \matr \Sigma_2, S_2)}
    [ \vec x^V ]
    \rp|
    \geq \eps/d^{O(k)}.
  \]
\end{lemma}
\begin{proof}[Proof sketch.]
  Let $W = S_1 \cap S_2 \cap \{f_1 > f_2\} \cup S_1\setminus S_2$,
  that is the set of points where the first density is larger than the
  second.  We now write the $L_1$ distance between $f_1, f_2$ as
  \[
    \int |f_1(\vec x) - f_2(\vec x)| \d \vec x =
    \int \1{W}(\vec x)
    (f_1(\vec x) - f_2(\vec x)) \d \vec x
  \]
  Denote $p(\vec x)$ the polynomial that will do the approximation
  of the $L_1$ distance.
  From Lemma~\ref{lem:chiSquaredExistence} we know that there exists
  a Normal distribution within small chi-squared divergence of both
  $\normal(\vec \mu_1, \matr \Sigma_1)$ and $\normal(\vec \mu_2, \matr \Sigma_2)$.
  Call the density function of this distribution $g(\vec x)$.
  We have
  \begin{align}\label{eq:tvd_apx}
    \Big| \int |f_1(\vec x) &- f_2(\vec x)|\d \vec x -
    \int p(\vec x) (f_1(\vec x) - f_2(\vec x))
    \Big|
                            \\
                            &\leq  \int |\1{W}(\vec x) - p(\vec x)|\ |f_1(\vec x) - f_2(\vec x)| \d \vec x \nonumber
                            \\
                            &\leq  \int |\1{W}(\vec x) - p(\vec x)| \sqrt{g(\vec x)}\ \frac{|f_1(\vec x) - f_2(\vec x)|}{\sqrt{g(\vec x)}} \d x \nonumber
                            \\
                            &\leq  \sqrt{\int (\1{W}(\vec x) - p(\vec x))^2 g(\vec x) \d \vec x}
                            \sqrt{\int \frac{(f_1(\vec x) - f_2(\vec x))^2}{g(\vec x)} \d \vec x},
  \end{align}
  where we use Schwarzs' inequality.  From Lemma~\ref{lem:chiSquaredExistence}
  we know that
  \[
    \int \frac{f_1(\vec x)^2}{g(\vec x)} \d \vec x
    \leq \int \frac{\normal(\vec \mu_1, \matr \Sigma_1; \vec x)^2}{g(\vec x)}
    \d \vec x
    = \exp(\poly(1/\alpha)).
  \]
  Similarly, $\int \frac{f_2(\vec x)^2}{g(\vec x)} \d x = \exp(\poly(1/\alpha))$.
  Therefore we have,
  \[
    \Big|
    \int |f_1(\vec x) - f_2(\vec x)|\d \vec x - \int p(\vec x) (f_1(\vec x) - f_2(\vec x))
    \Big|
    \leq \exp(\poly(1/\alpha)) \sqrt{\int (\1{W}(\vec x) - p(\vec x))^2 g(\vec x) \d \vec x}
  \]
  Recall that $g(\vec x)$ is the density function of a Gaussian distribution, and
  let $\vec \mu, \matr \Sigma$ be the parameters of this Gaussian.  Notice that
  it remains to show that there exists a good approximating polynomial $p(\vec
  x)$ to the indicator function $\1{W}$.  We can now transform the space so that
  $g(\vec x)$ becomes the standard normal.  Notice that this is an affine
  transformation that also transforms the set $W$;  Since the Gaussian
  surface area is "invariant" under linear transformations

  Since $\1{W} \in L^2(\R^d, \N_0)$ we can approximate it using Hermite
  polynomials.  For some $k \in \N$ we set $p(\vec x) = S_k \1{W}(x)$, that is
  \[
    p_k(\vec x) = \sum_{V: |V| \leq k} \wh{\1{W}} H_V(\vec x).
  \]
  Combining Lemma~\ref{lem:realSensitivityConcentration} and
  Lemma~\ref{lem:KOS08NoiseSurface} we obtain
  \[
    \E_{\vec x \sim \normal_0}[(\1{W}(\vec x) - p_k(x))^2] =
    O\lp( \frac{\Gamma(\mcal{S})} {k^{1/2}} \rp).
  \]
  Therefore,
  \(
    \Big|
    \int |f_1(\vec x) - f_2(\vec x)|\d \vec x - \int p_k(\vec x) (f_1(\vec x) - f_2(\vec x))
    \Big|
    = \exp(\poly(1/\alpha)) \frac{\Gamma(\mcal{S})^{1/2}}{k^{1/4}}.
    \)
  Ignoring the dependence on the absolute constant $\alpha$, to achieve
  error $O(\eps)$ we need degree $k = O(\Gamma(\mcal{S})^2/\eps^4)$.

  To complete the proof, it remains to obtain a bound for the coefficients
  of the polynomial $q(\vec x) = p_k(\matr \Sigma^{-1/2} (\vec x-\vec
  \mu))$.  Using known facts about the coefficients of Hermite polynomials
  we obtain that
  \(
    \norm{q(\vec x)}_\infty \leq
    \binom{d+k}{k}^2 (4d)^{k/2} (O(1/\alpha^2))^{k}.
  \)
  To conclude the proof we notice that we can pick the degree $k$ so that
  \[
    \lp| \int q(\vec x) (f_1(\vec x)- f_2(\vec x)) \rp| =
    \lp| \sum_{V: |V| \leq k} \vec x^V (f_1(\vec x) - f_2(\vec x)) \rp|
    \geq \eps /2.
  \]
  Since the maximum coefficient of $q(\vec x)$ is bounded by $d^{O(k)}$
  we obtain the result.
\end{proof}

 \section{VC-dimension vs Gaussian Surface Area}
\label{sec:vcGsa}

We use two different complexity measures of the truncation set to get sample
complexity bounds, the VC-dimension and the Gaussian Surface Area (GSA) of the
class of the sets.  As we already mentioned in the introduction, there are
classes, for example convex sets, that have bounded Gaussian surface area but
infinite VC-dimension.  However, this is not the main difference between the
two complexity measures in our setting.  Having a class with bounded
VC-dimension means that the empirical risk minimization needs finite samples.
To get an efficient algorithm we still need to \emph{implement the ERM for this
specific class}.  Therefore, it is not clear whether it is possible to get an
algorithm that works for all sets of bounded VC-dimension.  On the other hand,
bounded GSA means that we can approximate the weighted indicator of the set
using its low order Hermite coeffients.  This approximation works for all sets
of bounded GSA and does not depend on the specific class of sets.  Therefore,
using GSA we manage to get a unified approach that learns the parameters of the
underlying Gaussian distribution using only the assumption that the truncation
set has bounded GSA.  In other words, our approach uses the information of the
class that the truncation set belongs only to decide how large the degree of
the approximating polynomial should be.  Having said that, it is an interesting
open problem to design algorithms that learn the parameters of the Gaussian and
use the information that the truncation set belongs to some class (e.g.
intersection of $k$-halfspaces) to beat the runtime of our generic approach
that only depends on the GSA of the class.
 
\clearpage

\bibliographystyle{alpha}
\bibliography{refs}

\newcommand{\etalchar}[1]{$^{#1}$}
\begin{thebibliography}{DKK{\etalchar{+}}18}

\bibitem[AGR13]{AGR13}
Joseph Anderson, Navin Goyal, and Luis Rademacher.
\newblock Efficient learning of simplices.
\newblock In {\em Conference on Learning Theory}, pages 1020--1045, 2013.

\bibitem[Bal93]{Bal93}
Keith Ball.
\newblock The reverse isoperimetric problem for gaussian measure.
\newblock {\em Discrete \& Computational Geometry}, 10(1):411--420, 1993.

\bibitem[BC14]{BalakrishnanCramer}
N~Balakrishnan and Erhard Cramer.
\newblock {\em The art of progressive censoring}.
\newblock Springer, 2014.

\bibitem[Coh16]{Cohen91}
A~Clifford Cohen.
\newblock {\em Truncated and censored samples: theory and applications}.
\newblock CRC press, 2016.

\bibitem[CSV17]{CSV17}
Moses Charikar, Jacob Steinhardt, and Gregory Valiant.
\newblock Learning from untrusted data.
\newblock In {\em Proceedings of the 49th Annual {ACM} {SIGACT} Symposium on
  Theory of Computing, {STOC} 2017, Montreal, QC, Canada, June 19-23, 2017},
  pages 47--60, 2017.

\bibitem[CW01]{CarberyW01}
Anthony Carbery and James Wright.
\newblock Distributional and l\^{q} norm inequalities for polynomials over
  convex bodies in r\^{n}.
\newblock {\em Mathematical research letters}, 8(3):233--248, 2001.

\bibitem[DDS14]{DDS14}
Anindya De, Ilias Diakonikolas, and Rocco~A Servedio.
\newblock Learning from satisfying assignments.
\newblock In {\em Proceedings of the twenty-sixth annual ACM-SIAM symposium on
  Discrete algorithms}, pages 478--497. SIAM, 2014.

\bibitem[Den98]{Den98}
Fran{\c{c}}ois Denis.
\newblock Pac learning from positive statistical queries.
\newblock In {\em International Conference on Algorithmic Learning Theory},
  pages 112--126. Springer, 1998.

\bibitem[DGTZ18]{DGTZ18}
Constantinos Daskalakis, Themis Gouleakis, Christos Tzamos, and Manolis
  Zampetakis.
\newblock Efficient statistics, in high dimensions, from truncated samples.
\newblock In {\em the 59th Annual IEEE Symposium on Foundations of Computer
  Science (FOCS)}, 2018.

\bibitem[DK14]{DK14}
Constantinos Daskalakis and Gautam Kamath.
\newblock Faster and sample near-optimal algorithms for proper learning
  mixtures of gaussians.
\newblock In {\em Proceedings of The 27th Conference on Learning Theory, {COLT}
  2014, Barcelona, Spain, June 13-15, 2014}, pages 1183--1213, 2014.

\bibitem[DKK{\etalchar{+}}16]{DKK+16b}
Ilias Diakonikolas, Gautam Kamath, Daniel~M. Kane, Jerry Li, Ankur Moitra, and
  Alistair Stewart.
\newblock Robust estimators in high dimensions without the computational
  intractability.
\newblock In {\em {IEEE} 57th Annual Symposium on Foundations of Computer
  Science, {FOCS} 2016, 9-11 October 2016, Hyatt Regency, New Brunswick, New
  Jersey, {USA}}, pages 655--664, 2016.

\bibitem[DKK{\etalchar{+}}17]{DKK+17}
Ilias Diakonikolas, Gautam Kamath, Daniel~M. Kane, Jerry Li, Ankur Moitra, and
  Alistair Stewart.
\newblock Being robust (in high dimensions) can be practical.
\newblock In {\em Proceedings of the 34th International Conference on Machine
  Learning, {ICML} 2017, Sydney, NSW, Australia, 6-11 August 2017}, pages
  999--1008, 2017.

\bibitem[DKK{\etalchar{+}}18]{DKK+18}
Ilias Diakonikolas, Gautam Kamath, Daniel~M. Kane, Jerry Li, Ankur Moitra, and
  Alistair Stewart.
\newblock Robustly learning a gaussian: Getting optimal error, efficiently.
\newblock In {\em Proceedings of the Twenty-Ninth Annual {ACM-SIAM} Symposium
  on Discrete Algorithms, {SODA} 2018, New Orleans, LA, USA, January 7-10,
  2018}, pages 2683--2702, 2018.

\bibitem[DL12]{DL12}
Luc Devroye and G{\'a}bor Lugosi.
\newblock {\em Combinatorial methods in density estimation}.
\newblock Springer Science \& Business Media, 2012.

\bibitem[Eld11]{Eld11}
Ronen Eldan.
\newblock A polynomial number of random points does not determine the volume of
  a convex body.
\newblock {\em Discrete \& Computational Geometry}, 46(1):29--47, 2011.

\bibitem[Fis31]{fisher31}
RA~Fisher.
\newblock Properties and applications of {Hh} functions.
\newblock {\em Mathematical tables}, 1:815--852, 1931.

\bibitem[FJK96]{FJK96}
Alan Frieze, Mark Jerrum, and Ravi Kannan.
\newblock Learning linear transformations.
\newblock In {\em Foundations of Computer Science, 1996. Proceedings., 37th
  Annual Symposium on}, pages 359--368. IEEE, 1996.

\bibitem[Gal97]{Galton1897}
Francis Galton.
\newblock An examination into the registered speeds of american trotting
  horses, with remarks on their value as hereditary data.
\newblock {\em Proceedings of the Royal Society of London},
  62(379-387):310--315, 1897.

\bibitem[GR09]{GR09}
Navin Goyal and Luis Rademacher.
\newblock Learning convex bodies is hard.
\newblock {\em arXiv preprint arXiv:0904.1227}, 2009.

\bibitem[Kan11]{Kan11}
Daniel~M Kane.
\newblock The gaussian surface area and noise sensitivity of degree-d
  polynomial threshold functions.
\newblock {\em computational complexity}, 20(2):389--412, 2011.

\bibitem[KKMS05]{KKMS05}
Adam~Tauman Kalai, Adam~R. Klivans, Yishay Mansour, and Rocco~A. Servedio.
\newblock Agnostically learning halfspaces.
\newblock In {\em 46th Annual {IEEE} Symposium on Foundations of Computer
  Science {(FOCS} 2005), 23-25 October 2005, Pittsburgh, PA, USA, Proceedings},
  pages 11--20, 2005.

\bibitem[KOS08]{KOS08}
Adam~R. Klivans, Ryan O'Donnell, and Rocco~A. Servedio.
\newblock Learning geometric concepts via gaussian surface area.
\newblock In {\em 49th Annual {IEEE} Symposium on Foundations of Computer
  Science, {FOCS} 2008, October 25-28, 2008, Philadelphia, PA, {USA}}, pages
  541--550, 2008.

\bibitem[LDG00]{LDG00}
Fabien Letouzey, Fran{\c{c}}ois Denis, and R{\'e}mi Gilleron.
\newblock Learning from positive and unlabeled examples.
\newblock In {\em International Conference on Algorithmic Learning Theory},
  pages 71--85. Springer, 2000.

\bibitem[Led94]{Led94}
Michel Ledoux.
\newblock Semigroup proofs of the isoperimetric inequality in euclidean and
  gauss space.
\newblock {\em Bulletin des sciences math{\'e}matiques}, 118(6):485--510, 1994.

\bibitem[Lee14]{Lee1914}
Alice Lee.
\newblock Table of the gaussian" tail" functions; when the" tail" is larger
  than the body.
\newblock {\em Biometrika}, 10(2/3):208--214, 1914.

\bibitem[LRV16]{LRV16}
Kevin~A. Lai, Anup~B. Rao, and Santosh Vempala.
\newblock Agnostic estimation of mean and covariance.
\newblock In {\em {IEEE} 57th Annual Symposium on Foundations of Computer
  Science, {FOCS} 2016, 9-11 October 2016, Hyatt Regency, New Brunswick, New
  Jersey, {USA}}, pages 665--674, 2016.

\bibitem[Naz03]{Naz03}
Fedor Nazarov.
\newblock {\em On the Maximal Perimeter of a Convex Set in $\R^n$ with Respect
  to a Gaussian Measure}, pages 169--187.
\newblock Springer Berlin Heidelberg, Berlin, Heidelberg, 2003.

\bibitem[O'D14]{Don14}
Ryan O'Donnell.
\newblock {\em Analysis of Boolean Functions}.
\newblock Cambridge University Press, 2014.

\bibitem[Pea02]{Pearson1902}
Karl Pearson.
\newblock On the systematic fitting of frequency curves.
\newblock {\em Biometrika}, 2:2--7, 1902.

\bibitem[Pis86]{Pis86}
Gilles Pisier.
\newblock Probabilistic methods in the geometry of banach spaces.
\newblock In {\em Probability and analysis}, pages 167--241. Springer, 1986.

\bibitem[PL08]{PearsonLee1908}
Karl Pearson and Alice Lee.
\newblock On the generalised probable error in multiple normal correlation.
\newblock {\em Biometrika}, 6(1):59--68, 1908.

\bibitem[Sch86]{Schneider86}
Helmut Schneider.
\newblock {\em Truncated and censored samples from normal populations}.
\newblock Marcel Dekker, Inc., 1986.

\bibitem[SJ66]{ShahJ1966}
SM~Shah and MC~Jaiswal.
\newblock Estimation of parameters of doubly truncated normal distribution from
  first four sample moments.
\newblock {\em Annals of the Institute of Statistical Mathematics},
  18(1):107--111, 1966.

\bibitem[SSBD14]{ShalevB14}
Shai Shalev-Shwartz and Shai Ben-David.
\newblock {\em Understanding machine learning: From theory to algorithms}.
\newblock Cambridge university press, 2014.

\bibitem[Sze67]{Sze67}
G.~Szeg{\"o}.
\newblock {\em Orthogonal Polynomials}.
\newblock Number $\tau$. 23 in American Mathematical Society colloquium
  publications. American Mathematical Society, 1967.

\end{thebibliography}

\appendix
\section{Additional Preliminaries and Notation}
\label{app:prelims}
We first state the following simple lemma that connects the total
variation distance of two Normal distributions with their parameter
distance.  For a proof see e.g. Corollaries 2.13 and 2.14 of \cite{DKK+16b}.
\begin{lemma}\label{lem:tvdParameter}
  Let $N_1 = \normal(\vec \mu_1, \matr \Sigma_1)$ ,
  $N_2 = \normal(\vec \mu_2, \matr \Sigma_2)$ be two Normal distributions.
Then
  \[
    \dtv{N_1}{N_2} \leq
    \frac1{2} \snorm{2}{\matr \Sigma_1^{-1/2}(\vec \mu_1 - \vec \mu_2)}
    +
    \sqrt{2}\snorm{F}
    {\matr I - \matr \Sigma_1^{-1/2} \matr \Sigma_2 \matr \Sigma_1^{-1/2}}
  \]
\end{lemma}

We readily use the following two lemmas from \cite{DGTZ18}. The first suggests
that we can accurately estimate the parameters $(\mu_S,\Sigma_S)$.

\begin{lemma}\label{lem:conditionalEstimation}
  Let $(\vec{\mu}_S, \mat{\Sigma}_S)$ be the mean and covariance of the truncated
  Gaussian $\normal(\vec \mu,\mat \Sigma,S)$ for a set $S$ such that
  $\normal(\vec \mu,\mat \Sigma;S) \ge \alpha$.  Using $\tilde O(\frac{d}{\eps^2}
  \log(1/\alpha) \log^2(1/\delta))$ samples, we can compute estimates $
  \wt{\vec{\mu}}_S$ and $\wt{\mat{\Sigma}}_S$ such that
  ,with probability at least $1-\delta$,
  $$\|\mat{\Sigma}^{-1/2}
  (\wt{\vec{\mu}}_S - \vec{\mu}_S)\|_2 \le \eps \quad \text{ and } \quad
  (1-\eps) \mat{\Sigma}_S \preceq \wt{\mat{\Sigma}}_S \preceq (1+\eps)
  \mat{\Sigma}_S$$
\end{lemma}

The second lemma suggests that the empirical estimates are close to the true
parameters of underlying truncated Gaussian.

\begin{lemma}\label{lem:conditionalParameterDistance}
  The empirical mean and covariance $\wt{\vec{\mu}}_S$ and
  $\wt{\mat{\Sigma}}_S$ computed using $\tilde O(d^2 \log^2(1/\alpha\delta))$
  samples from a truncated Normal $\normal(\vec \mu,\mat \Sigma, S)$ with
  $\normal(\vec \mu,\mat \Sigma;S) \ge \alpha$ satisfies with probability
  $1-\delta$ that:
  $$\|\mat{\Sigma}^{-1/2} (\wt{\vec{\mu}}_S - \vec{\mu})\|^2_2 \leq O({\log
  \frac{1}{\alpha}}), \quad {\wt{\mat{\Sigma}}}_S \succeq \Omega(\alpha^2)
  \mat{\Sigma}, \quad \snorm{F}{\mat \Sigma^{-1/2} \mat{\wt{\Sigma}}_S \mat
  \Sigma^{-1/2}  - \mat I}^2 \le O({\log \frac{1}{\alpha}} ).$$
  Moreover, $\Omega(\alpha^2) \leq
  \norm{\wb \Sigma_S^{-1/2} \matr \Sigma \wb \Sigma_S^{-1/2}}_2
  \leq O(1/\alpha^2)$.
\end{lemma}

In particular, the mean and covariance $\wt{\vec{\mu}}_S$ and
$\wt{\mat{\Sigma}}_S$ that satisfy the conditions of
Lemma~\ref{lem:conditionalParameterDistance}, are in
$(O(\log(1/\alpha)),1-O(\alpha^2))$-near isotropic position.

We will use the following very useful anti-concentration result about the
Gaussian mass of sets defined by polynomials.
\begin{theorem}[Theorem 8 of \cite{CarberyW01}] \label{thm:GaussianMeasurePolynomialThresholdFunctions}
  Let $q, \gamma \in \reals_+$, $\vec{\mu} \in \reals^d$,
  $\matr{\Sigma} \in \R^{d \times d}$ such that $\S$ is symmetric positive
  semidefinite and $p : \reals^d \to \reals$ be a multivariate
  polynomial of degree at most $\ell$, we define
  \[ \bar{Q} = \left\{ \vec{x} \in \reals^d \mid \abs{p(\vec{x})} \le \gamma
  \right\}, \]
  then there exists an absolute constant $C$ such that
  \[ \normal(\vec{\mu}, \matr{\Sigma}; \bar{Q}) \le \frac{C q \gamma^{1/\ell}}
    {\left( \Exp_{\vec{z} \sim \normal(\vec{\mu}, \matr{\Sigma})}
  \left[ \abs{p(\vec{z})}^{q/\ell} \right] \right)^{1/q}}. \]
\end{theorem}

\subsection{Hermite Polynomials, Ornstein-Uhlenbeck Operator, and Gaussian Surface Area.}
We denote by $L^2(\R^d, \normal_0)$ the vector space of all functions $f:\R^d
\to \R$ such that $\E_{\vec x \sim \normal_0}[f^2(x)] < \infty$.  The usual
inner product for this space is
$\E_{\vec x \sim \normal_0}[f(\vec x) g(\vec x)]$.
While, usually one considers the probabilists's or physicists' Hermite polynomials,
in this work we define the \emph{normalized} Hermite polynomial of degree $i$ to be
\(
H_0(x) = 1, H_1(x) = x, H_2(x) = \frac{x^2 - 1}{\sqrt{2}},\ldots,
H_i(x) = \frac{He_i(x)}{\sqrt{i!}}, \ldots
\)
where by $He_i(x)$ we denote the probabilists' Hermite polynomial of degree $i$.
These normalized Hermite polynomials form a complete orthonormal basis for the
single dimensional version of the inner product space defined above. To get an
orthonormal basis for $L^2(\R^d, \normal_0)$, we use a multi-index $V\in \nats^d$
to define the $d$-variate normalized Hermite polynomial as
$H_V(\vec x) = \prod_{i=1}^d H_{v_i}(x_i)$.  The total degree of $H_V$ is
$|V| = \sum{v_i \in V} v_i$.
Given a function $f \in L^2$ we compute its Hermite coefficients as
\(
\hat{f}(V) = \E_{\vec x\sim \normal_0} [f(\vec x) H_V(\vec x)]
\)
and express it uniquely as
\(
\sum_{V \in \nats^d} \hat{f}(V) H_V(\vec x).
\)
We denote by $S_k f(x)$ the degree $k$ partial sum of the Hermite expansion of $f$,
$S_kf (\vec x) = \sum_{|V| \leq k} \hat{f}(V) H_V(\vec x)$.
Then, since the basis of Hermite polynomials is complete, we have
\(
\lim_{k \to \infty} \E_{x \sim \normal_0}[\lp(f(\vec x) - S_k f(\vec x) \rp)^2] = 0.
\)
We would like to quantify the convergence rate of $S_k f$ to $f$.
Parseval's identity states that
\[
  \E_{x \sim \normal_0}[ \lp(f(\bm{x}) - S_kf(\bm{x}) \rp)^2 ]
  = \sum_{|V| = k}^{\infty} \hat{f}(V)^2.
\]
\begin{definition}[\textsc{Hermite Concentration}]\label{def:hermiteConcentration}
  Let $\gamma(\eps, d)$ be a function $\gamma : (0,1/2) \times \nats \mapsto \nats$.
  We say that a class of functions $\mcal{F}$ over $\R^d$ has a Hermite
  concentration bound of $\gamma(\eps, d)$, if for all $d \geq 1$, all $ \eps
  \in (0,1/2)$, and $f \in \mcal{F}$ it holds
  \(
  \sum_{|V| \geq \gamma(\eps, d)} \hat{f}(V)^2 \leq \eps.
  \)
\end{definition}

We now define the Gaussian Noise Operator as in \cite{Don14}.
Using a different parametrization, which is not convenient for
our purposes, these operators are also known as the
Ornstein-Uhlenbeck semigroup, or the Mehler transform.
\begin{definition}
  The Gaussian Noise operator $T_{\rho}$ is the linear operator defined on the
  space of functions $L^1(\R^d, \normal_0)$ by
  \[
    T_{\rho} f(\bm{x}) = \E_{\bm{y} \sim \normal_0}
    \lp[f( \rho \bm{x} + \sqrt{1- \rho^2} \bm{y}) \rp].
  \]
\end{definition}
A nice property of operator $T_{1-\rho}$ that we will use is that it
has a simple Hermite expansion
\begin{equation}\label{eq:noiseOperatorExpansion}
  S_k (T_\rho f)(\vec x) = \sum_{V: |V| \leq k} \rho^{|V|} \wh{f}(V) H_V(\vec x)
\end{equation}
We also define the noise sensitivity of a function $f$.
\begin{definition}[\textsc{Noise Sensitivity}]
  Let $f: \R^d \mapsto \R$ be a function in $L^2(\R^d, \normal_0)$.
  The noise sensitivity of $f$ at $\rho \in [0,1]$ is defined to be
  \[
    \NS_\rho[f] = 2
    \E_{\vec x \sim \normal_0}
    [f(\vec x)^2 - f(\vec x) T_{1-\rho}f(\vec x)]
  \]
\end{definition}
Since, the vectors $\vec x$ and
$\vec z = (1-\rho) \vec x + \sqrt{1-\rho^2} \vec y$
are jointly distributed according to
\begin{equation}\label{eq:correlated}
  D_\rho =
  \normal\lp(
  \begin{pmatrix}
    \vec 0
    \\
    \vec 0
  \end{pmatrix}
  ,
  \begin{pmatrix}
    \mat I & (1-\rho) \mat I \\
    (1-\rho) \mat I & \mat I
  \end{pmatrix}
  \rp).
\end{equation}
we can write
\begin{equation}\label{eq:noiseSensitivityCorrelation}
  \NS_\rho[f] =
  \E_{(\vec x,\vec z) \sim D_\rho}\lp[f(\vec x)^2\rp]+
  \E_{(\vec x, \vec z) \sim D_\rho} \lp[
  f(\vec z)^2 - 2 f(\vec x) f(\vec z)\rp]=
  \E_{(\vec x,\vec z) \sim D_\rho}[ (f(\vec x) - f(\vec z))^2].
\end{equation}

When $f$ is an indicator of a set, the noise sensitivity is
\begin{equation}\label{eq:binaryNoiseSensitivity}
  \NS_\rho[\1{S}] =
  2 \E_{(\vec x, \vec z)}
  [\1{S}(\vec x)(1-\1{S}(\vec z))]=
  2 \E_{(\vec x, \vec z)}
  [\1{S}(\vec x)\1{S^c}(\vec z)],
\end{equation}
which is equal to the probability of the correlated points $\vec x, \vec z$
landing at "opposite" sides of $S$.

Ledoux~\cite{Led94} and Pisier~\cite{Pis86} showed that the noise sensitivity of a set can be bounded by its Gaussian surface area.

\begin{definition}[Gaussian Surface Area]
  For a Borel set $A \subseteq \R^d$, its Gaussian surface area is
  $
  \Gamma(A) = \liminf_{\delta \to 0} \frac{\normal_0(A_\delta \setminus A)}{\delta},
  $  where $A_\delta = \{x : \dist(x, A) \leq \delta\}$.
\end{definition}

We will use the following lemma given in \cite{KOS08}.
\begin{lemma}[Corollary 14 of \cite{KOS08}]\label{lem:KOS08NoiseSurface}
  For a Borel set $S \subseteq \R^d$ and $\rho \geq 0$,
  $    \NS_{\rho}[\1{S}(\vec x)] \leq \sqrt{\pi} \sqrt{\rho}\ \Gamma(S).
  $
\end{lemma}
For more details on the Gaussian space and Hermite Analysis (especially from the theoretical
computer science perspective), we refer the reader to \cite{Don14}.
Most of the facts about Hermite polynomials that we shall use in this work are well known
properties and can be found, for example, in \cite{Sze67}.
 \section{Missing proofs of Section~\ref{sec:identifiabilityVC}} \label{app:vc}
We will use a standard tournament based approach for selecting a good hypotheses.
We will use a version of the tournament from \cite{DK14}.  See also
\cite{DL12}.
\begin{lemma}[Tournament \cite{DK14}]\label{lem:tournament}
  There is an algorithm, which is given sample access to some distribution $X$
  and a collection of distributions $\mcal{H}=\{H_1,\ldots,H_N\}$ over some
  set, access to a PDF comparator for every pair of distributions $H_i$, $H_j
  \in \mcal{H}$, an accuracy parameter $\eps >0$, and a confidence parameter
  $\delta >0$.  The algorithm makes $O(\log (1/\delta) \eps^2) \log N)$ draws
  from each of $X,H_1,\ldots, H_N $ and returns some $H\in \mcal{H}$ or
  declares ''failure'' If there is some $H\in \mcal{H}$ such that
  $\dtv{H}{X}\leq \eps$ then with probability at least $1-\delta$ the returned
  distribution $H$ satisfies $\dtv{H}{X}\leq 512\eps$.  The total number of
  operations of the algorithm is $O(\log(1/\delta)(1/\eps^2) (N\log N+
  \log1/\delta))$.
\end{lemma}

We first argue that if the class of sets $\mcal{S}$ has VC-dimension
$\VC(\mcal{S})$ then we can learn the truncated model in $\eps$ total variation
by drawing roughly $\VC(\mcal{S})/\eps$ samples.  We will use the following
standard fact whose proof may be found for example in page 398 of
\cite{ShalevB14}.  For convenience we restate the result using our notation.
\begin{lemma}[\cite{ShalevB14}]\label{lem:VCBound}
  Let $D$ be a distribution on $\R^d$.  Let $\mcal{S}$ be a family of subsets
  of $\R^d$.  Fix $\eps \in (0,1), \delta \in (0,1/4)$ and let \(
  N =
  O(\VC(\mcal{S}) \log(1/\eps)/\eps + \log(1/\delta) )
  \)
  Then, with probability at least $1-\delta$ over  a choice of a sample $X
  \sim D^N$ we have that if $D(S) \geq \eps$ then $|S \cap X| \neq \emptyset$.
\end{lemma}

\begin{prevproofbig}{Lemma}{lem:tvdVCBound}
  We define the class of sets $\mcal{A} = \{S^* \setminus S : S \in \mcal{S}
  \}$.  We first argue that for any $A \subset \R^d$ we have $\VC(\mcal{A})
  \leq \VC(\mcal{S})$.  Let $X \subset \R^d$ be a set of points.  The set of
  different labellings of $X$ using sets of $\mcal{S}$ resp. $\mcal{A}$ is
  $L_{\mcal{S}} = \{X \cap S : S \in \mcal{S} \}$ resp.  $L_{\mcal{A}} = \{X
  \cap S : S \in \mcal{A} \} = \{X \cap (A \setminus S) : S \in \mcal{S} \} $.
  We define the function $g: L_{\mcal{A}} \to L_{\mcal{S}}$ by $ g(X \cap (A
  \setminus S)) = X \cap S.  $ We that observe for $S_1, S_2 \in \mcal{S}$ we
  have that $X \cap S_1 = X \cap S_2$ implies that $X \cap (A \setminus S_1) =
  X \cap (A \setminus S_2)$.  Therefore, $g$ is one-to-one and we obtain that
  $|L_{\mcal{A}}| \leq |L_{\mcal{S}}|$.  We draw $N$ samples $X=\{x_i, i
  \in{N}\}$.  Applying Lemma~\ref{lem:VCBound} for the family $\mcal{A}$, we
  have that with $N$ samples, with probability at least $1-\delta$ it holds
  that if $\normal(\vec \mu, \matr \Sigma; S^* \setminus S) \geq \eps$ for some
  set $S \in \mcal{S}$ then $|(S^* \setminus S) \cap X| > 0$.   Therefore,
  every set that is consistent with the samples, i.e. every $S$ that
  that contains the samples, satisfies the property
  $\normal(\vec \mu, \matr \Sigma; S^* \setminus S) \leq \eps$.  Moreover,
  since $\dtv{\normal(\wb \mu, \wb \Sigma)}{\normal(\vec \mu, \matr \Sigma)}
  \leq \eps$ we obtain that $\normal(\wb \mu, \wb \Sigma, S^* \setminus S) \leq
  2 \eps $ for any set $S$ consistent with the data.

  Next, we use the fact that $\wt{S}$ is chosen so that
  $
  \normal(\wb \mu, \wb \Sigma, S^*) \geq
  \normal(\wb \mu, \wb \Sigma, \wt{S}).
  $
  This means that for all $x \in S^* \cap \wt S$ it holds
  $
  \normal(\wb \mu,\wb \Sigma, S^* ; \vec x) \leq
  \normal(\wb \mu, \wb \Sigma , \wt{S}; \vec x)
  $.
  To simplify notation we set
  $\wt \normal_{\wt S} = \normal(\wb \mu, \wb \Sigma, \wt S)$,
  $\wt \normal_{S^*} = \normal(\wb \mu, \wb \Sigma, S^*)$,
  and $\normal_{S^*} = \normal(\vec \mu, \matr \Sigma, S^*)$.
  We have
  \[
    2 \dtv{\wt \normal_{\wt S}}{\wt \normal_{S^*}} =
    \int_{
    \wt \normal_{S^*}(\vec x) \geq \wt \normal_{\wt{S}(\vec x)}}
    \lp(\wt \normal_{S^*}(\vec x) - \wt \normal_{\wt S}(\vec x) \rp)
    \d \vec x
    \leq
    \int_{S^* \setminus \wt S} \wt \normal_{S^*}(\vec x) \d \vec x
    \leq
    \frac{
      \normal(\wb \mu, \wb \Sigma; S^* \setminus \wt S)}{
    \alpha}
    \leq \frac \eps \alpha.
  \]
  Moreover,
  \[\dtv{\wt \normal_{S^*}}{\normal_{S^*}} \leq
    \frac{
      \dtv{\normal(\wb \mu, \wb \Sigma)}{\normal(\vec \mu, \matr \Sigma)}
    }{\alpha}
    \leq \frac\eps\alpha
  \]
  Using the triangle inequality we obtain that $
  \dtv{\normal(\wb \mu, \wb \Sigma, \wt S)}{\normal(\vec \mu, \matr \Sigma, S^*)}
  \leq 3\eps/(2 \alpha)$.
\end{prevproofbig}

\begin{prevproofbig}{Lemma}{lem:learningTVD}
  Using Lemma~\ref{lem:conditionalParameterDistance} we know that we can draw
  $\wt{O}(d^2 \log^2(1/\alpha \delta))$ samples and obtain estimates of the
  conditional mean and covariance $\wt{\vec \mu}_C$, $\wt {\matr \Sigma}_C$.
  Transforming the space so that $\wt{\vec \mu}_C = 0$ and $\wt {\matr
  \Sigma}_C = \matr I$.  For simplicity we will keep denoting the parameters of
  the unknown Gaussian $\vec \mu, \matr \Sigma$ after transforming the space.
  From  Lemma~\ref{lem:conditionalParameterDistance} we have that
  $\norm{\matr \Sigma^{-1/2} \vec \mu}_2 \leq O(\log(1/\alpha)^{1/2}/\alpha)$, $\Omega(\alpha^2) \leq
  \norm{\Sigma^{-1/2}}_2 \leq O(1/\alpha^2)$ and $\norm{I - \Sigma}_F \leq
  O(\log(1/\alpha)/\alpha^2)$.  Therefore, the cube of $\R^{d + d^2}$ where all
  the parameters $\mu_i, \Sigma_{ij}$ of the mean and the covariance lie has
  side length at most $O(1/\poly(a))$.  We can partition this cube into smaller
  cubes of side length $O(\eps \poly(a)/d)$ and obtain that there exists a
  point of the grid $(\vec u, \matr B)$ such that $\norm{\Sigma^{-1/2} (\vec u
    - \vec \mu)}_2 \leq \eps$, $\norm{\matr I - \matr \Sigma^{-1/2} \matr B
    \matr \Sigma^{-1/2} }_F \leq \eps $, which implies that $\dtv{\normal(\vec
  u,\matr B)}{\normal(\vec \mu, \matr \Sigma)} \leq \eps$.  Assume now that
  for each guess $(\vec u, \matr B)$ of our grid we solve the optimization
  problem as defined in Lemma~\ref{lem:tvdVCBound} and find a candidate set
  $S_{\vec u, \matr B}$.  Notice that the set of our hypotheses $\vec u,
  \matr B, S_{\vec u, \matr B}$ is $O((d^2/\eps)^{d^2+d})$.  Moreover, using
  Lemma~\ref{lem:tvdVCBound} and the fact that there exists a point $\vec u,
  \matr B)$ in the grid so that $\dtv{\normal(\vec u,\matr B)}{\normal(\vec
    \mu, \matr \Sigma)} \leq \eps$, we obtain that $\dtv{\normal(\vec u,
  \matr B, S_{\vec u, \matr B})}{\normal(\vec \mu, \matr \Sigma, S)} \leq
  \eps$.  Now we can use Lemma~\ref{lem:tournament} we can select a
  hypotheses $\normal(\vec u, \matr B, \wt S)$ within $O(\eps)$ total
  variation distance of $\normal(\vec \mu, \matr \Sigma, S)$, and the number
  samples required to run the tournament is as claimed.
\end{prevproofbig}
 \section{Missing Proofs of Section~\ref{sec:weighted_learning}}\label{app:weighted_learning}

To prove Theorem~\ref{thm:polynomialApproximation} we shall use the inequalities of
Lemma~\ref{lem:logInequalities}.
\begin{lemma}\label{lem:logInequalities}
  Let $k\in \nats$.
  Then for all  $0< x< \frac{2k+1}{2k}$ it holds,
  \begin{align*}
    -k \log x - \frac12 \log(1- 2k (x -1) &\leq 2 k^2 (x-1)^2 \lp(\frac1 x + \frac1{1-2k(x-1)}\rp)
  \end{align*}
  Moreover, for all  $x> \frac{2k -1}{2k}$
  \begin{align*}
    k \log x - \frac12 \log(1 - 2k (1-x)) &\leq k^2(1-x)^2\lp(1 + \frac{1}{1 - 2k(1-x)}\rp).
  \end{align*}
\end{lemma}
\begin{proof}
  We start with the first inequality.
  Let $f(x) = -k \log x - \frac12 \log(1- 2k (x -1)$.
  We first assume that $ 1 \leq x \frac{2k+1}{2k}$.
  We have
  \begin{align*}
    f(x) &=
    \int_1^x \lp( \frac k{1 -2k(t-1)} - \frac k t \rp)\d t \\
         &=
         k (1+2k) \int_1^x \frac{t-1}{t (1-2k(t-1))} \d t \\
         &\leq
         \frac{k (1+2k)}{1 - 2k(x-1)} \int_1^x (t-1) \d t \\
         &\leq
         2k^2\frac{(x-1)^2}{1-2k(x-1)}
  \end{align*}
  If $ 0 < x \leq 1$ we have
  \begin{align*}
    f(x) &\leq
    \frac{k (1+2k)}{x} \int_1^x (t-1) \d t
    \leq 2 k^2 \frac{(x-1)^2}{x}
  \end{align*}
  Adding these two bounds gives an upper bound for all $0<x<\frac{2k+1}{2k}$.
  Similarly, we now show the second inequality.  Let $g(x) = k \log x - \frac12
  \log(1 - 2k (1-x))$.  We first assume that $1 \leq x$ and write
  \begin{align*}
    g(x) & = \int_1^x\lp(\frac{k}{t} - \frac{k}{1-2k(1-t)}\rp) \d t \\
         & = k \int_1^x \frac{(t-1)(2k -1)}{t ( 1+2k(t-1))} \d t \\
         &\leq k (2k-1) \int_1^x \frac{t-1} \d t\\
         &\leq k^2 (x-1)^2.
  \end{align*}
  Similarly, if $\frac{2k -1}{2k} < x \leq 1$ we have
  \begin{align*}
    g(x) \leq k^2 \frac{(1-x)^2}{1 - 2k(1-x)}.
  \end{align*}
  We add the two bounds together to get the desired upper bound.

\end{proof}

\begin{prevproofbig}{Lemma}{lem:ratioBound}
  For simplicity we denote $\normal_i = \normal(\vec \mu_1, \mat \Sigma_i)$.
  We start by proving the upper bound.  Using Schwarz's inequality we write
  \[
    \E_{\vec x \sim \normal_0}
    \lp[
    \lp(\frac{\normal_1(\vec x)}{\normal_0(\vec x)}\rp)^k
    \lp(\frac{\normal_0(\vec x)}{\normal_2(\vec x)} \rp)^k
    \rp]
    \leq
    \lp(
    \E_{\vec x \sim \normal_0}
    \lp(\frac{\normal_1(\vec x)}{\normal_0(\vec x)}\rp)^{2k}
    \rp)^{1/2}
    \lp(
    \E_{\vec x \sim \normal_0}
    \lp(\frac{\normal_0(\vec x)}{\normal_2(\vec x)} \rp)^{2k}
    \rp)^{1/2}.
  \]
  We can now bound each term independently.  We start by the
  ratio of $\normal_1/\normal_0$.
  Without loss of generality we may assume that $\mat\Sigma^1$ is diagonal,
  $\mat\Sigma_1 = \diag(\lambda_1,\ldots, \lambda_d)$. We also let
  $\vec \mu_1 = (\mu_1,\ldots, \mu_d)$.
  We write
  \begin{align*}
    \E_{\vec x \sim \normal_0}
    \lp[\lp(\frac{\normal_1(\vec x)}{\normal_0(\vec x)}\rp)^{2k}\rp]
    &= \frac{1}{|\mat \Sigma_1|^{k}}
    \E_{\vec x \sim \normal_0} \lp[
    \exp\lp(-k
    (\vec x- \vec \mu_1)^T \mat\Sigma_1{^{-1}}(\vec x - \vec \mu_1)
    + k \vec x^T \vec x
    \rp)
    \rp]
 \\ &=
 \frac{
 \exp(-k \vec \mu_1{^T} \mat \Sigma_1{^{-1}} \vec \mu_1)
 } {|\mat \Sigma_1|^{k}}
 \E_{\vec x \sim \normal_0}
   \lp[
   \exp\lp(k \vec x^T (\mat I - \mat \Sigma_1^{-1}) \vec x
   + 2 k \vec \mu_1^T \mat \Sigma_1{^{-1}} \vec x
   \rp)
   \rp]
   \\ &\leq
    \frac{1}{|\mat \Sigma_1|^{k}}
    \E_{\vec x \sim \normal_0}\lp[
    \exp \lp( \sum_{i=1}^d
    \lp( k (1 - 1/\lambda_i) x_i^2 +
    2 k \frac{\mu_i}{\lambda_i} x_i \rp)
    \rp)
    \rp]
   \\ &=
   \underbrace{
   \prod_{i=1}^d
   \frac{1}{\lambda_i^{k}}
   \E_{x \sim \normal_0}
   \lp[
   \exp \lp(k (1 - 1/\lambda_i) x^2+ 2 k \frac{\mu_i}{\lambda_i} x \rp)
   \rp]
 }_{A}
 \end{align*}
 We now use the fact that for all $a < 1/2$.
 \[
   \E_{x\sim \normal_0}[\exp(a x^2 + b x)]
   = \frac{1}{\sqrt{1-2 a}} \exp\lp(\frac{b^2}{2 - 4 a} \rp)
 \]
 At this point notice that since for all $i$ it holds
 $\lambda_i < 2k/(2k- 1)$ we have that term $A$ is bounded.
 We get that
 \begin{align*}
   A &=
   \underbrace{
     \exp
     \lp(\sum_{i=1}^d
     \lp(
     k \log \frac{1}{\lambda_i} -
     \frac{1}{2} \log\lp(1 -2k\lp(1-\frac1 \lambda_i\rp) \rp)
     \rp)
     \rp)
   }_{A_1}
   \
   \underbrace{
     \exp \lp(\sum_{i=1}^d \frac{2 k^2 \mu_i^2}{\lambda_i^2(1 - 2k (1-1/\lambda_i))}
     \rp)
   }_{A_2}
 \end{align*}
 To bound the term $A_1$ we use the second inequality of Lemma~\ref{lem:logInequalities}
 to get
 \[
   A_1 \leq \exp\lp(\sum_{i=1}^d k^2 (1- 1/\lambda_i)^2
   \lp(1 + \frac{1}{1-2k (1-1/\lambda_i)} \rp)
   \rp)
   \leq \exp\lp(\frac{2 k^2 B}{\delta} \rp)
 \]
 Bounding $A_2$ is easier
 \[
   A_2 \leq \exp\lp(\frac{2k^2 \snorm{2}{\vec \mu_1}^2}{\lambda_{\min}^2 \delta}\rp)
 \]
 Combining the bounds for $A_1$ and $A_2$ we obtain
 \[
    \E_{\vec x \sim \normal_0}
    \lp[\lp(\frac{\normal_1(\vec x)}{\normal_0(\vec x)}\rp)^{2k}\rp]
        \leq \exp\lp(\frac{10 k^2}{\delta} B \rp)
 \]
 We now work similarly to bound the ratio $\normal_0/\normal_2$.
 We will again assume that $\mat \Sigma_2 = \diag(\lambda_1,\ldots,\lambda_d)$ and
 $\mu_2 = (\mu_1,\ldots,\mu_d)$.
 We have
 \begin{align*}
   \E_{\vec x \sim \normal_0}
   \lp[\lp(\frac{\normal_0(\vec x)}{\normal_2(\vec x)} \rp)^{2k} \rp]
   &= \exp(k \vec \mu_2^T \Sigma_2^{-1}  \vec \mu_2)
     \E_{\vec x \sim \normal_0}
     \lp[
     |\mat \Sigma_2|^k
     \exp\lp( k \vec x^T(\mat \Sigma_2^{-1} -
     \mat I) \vec x -2 k \vec \mu_2 \mat \Sigma_2^{-1} \vec x
     \rp)
     \rp]
     \\ & \leq
   \exp((k+1) B)
   \prod_{i=1}^d \E_{x \sim \normal_0}\lp[
   \exp\lp( k(1/\lambda_i-1) x^2 -k \log(1/\lambda_i) -2k (\mu_i/\lambda_i) x\rp)
   \rp]
   \\
        &=
        \exp\lp(\lp(\frac{8k^2}{\delta} + k + 1\rp)B \rp)
   \exp
   \lp(
   \sum_{i=1}^d
   \lp(- k \log(1/\lambda_i) - \frac12 \log\lp(1 - 2k (1/\lambda_i - 1) \rp) \rp)
   \rp)
   \\
        &\leq
        \exp\lp( \lp(\frac{10 k^2}{\delta} + 4k^2 + k + 1 \rp) B \rp),
 \end{align*}
 where to obtain the last inequality we used the first inequality of
 Lemma~\ref{lem:logInequalities} and the bounds for the maximum and minimum
 eigenvalues of $\mat \Sigma_2$.
 Finally, plugging in the bounds for the two ratios we get for $i=1,2$
 \[
   \E_{\vec x \sim \normal_0}\lp[
   \lp(\frac{\normal_{3-i}(\vec x)}{\normal_{i}(\vec x)}\rp)^k
   \rp]
   \leq \exp\lp(\frac{13 k^2}{\delta} B \rp).
 \]
 Having the upper bound it is now easy to prove the lower bound using
 the convexity of $x \mapsto x^{-1}$ and Jensen's inequality.
 \begin{align*}
   \E_{\vec x \sim \normal_0}\lp[
   \lp(\frac{\normal_1(\vec x)}{\normal_2(\vec x)}\rp)^k
   \rp]
   =
   \E_{\vec x \sim \normal_0}\lp[
   \lp(\frac{\normal_2(\vec x)}{\normal_1(\vec x)}\rp)^{-k}
   \rp]
   \geq
   \lp(
   \E_{\vec x \sim \normal_0}\lp[
   \lp(\frac{\normal_2(\vec x)}{\normal_1(\vec x)}\rp)^{k}
   \rp]
   \rp)^{-1}
   \geq \exp\lp(-\frac{13 k^2}{\delta} B\rp).
\end{align*}
\end{prevproofbig}

\begin{prevproofbig}{Lemma}{lem:realSensitivityConcentration}
  For any $\rho \in (0,1)$, using identity \ref{eq:noiseOperatorExpansion},
  we write
  \[
    \E_{\vec x \sim \normal_0}
    [f(\vec x) T_{1-\rho}(\vec x)] =
    \sum_{V \in \nats^d} (1-\rho)^{|V|} \widehat{f}(V)^2
  \]
  \begin{align*}
    \E_{\bm{x} \sim \mcal{N}(\bm{0}, \mat{I})}
    \lp[f(\bm{x})^2 - f(\bm{x}) T_{1-\rho} f(\bm{x}) \rp] &=
    \sum_{V \in \nats^d} \wh{f}(V)^2 - \sum_{V \in \nats^d} (1-\rho)^{|V|} \wh{f}(V)^2
                                                       \\ &=
                                                       \sum_{V \in \nats^d} (1 - (1-\rho)^{|V|})\ \wh{f}(V)^2
                                                       \\ &\geq
                                                       \sum_{|V|\geq 1/\rho} (1 - (1-\rho)^{|V|})\ \wh{f}(V)^2
                                                       \\ &\geq
                                                       \sum_{|V|\geq 1/\rho} (1 - (1-\rho)^{1/\rho})\ \wh{f}(V)^2
                                                       \\ &\geq (1-1/\me) \sum_{|V|\geq 1/\rho} \wh{f}(V)^2
  \end{align*}
\end{prevproofbig}

\begin{prevproofbig}{lemma}{lem:noiseDerivative}
  We first write
  \[
    \frac{1}{2}
    \E_{(x, z) \sim D_\rho}[(r(\vec x) - r(\vec z))^2]
    =
    \frac{1}{2}
    \E_{(x, z) \sim D_\rho}\lp[\frac{r(\vec x)^2}{2}
    + \frac{r(\vec z)^2}{2} - r(\vec x) r(\vec z) \rp]
    =
    \E_{(x, z) \sim D_\rho}[r(\vec x)^2 - r(\vec z) r(\vec x)].
  \]
  Let
  \[
    \sum_{V \in \nats^d} \wh{r}(V) H_V(\vec x)
  \]
  be the Hermite expansion of $r(\vec x)$.  From Parseval's identity and the
  Hermite expansion of Ornstein–Uhlenbeck operator, \eqref{eq:noiseOperatorExpansion} we have
  \begin{align*}
    \E_{(x, z) \sim D_\rho}[r(\vec x)^2 - r(\vec x) r(\vec z)]
    &=
    \sum_{V \in \nats^d} \wh{r}(V)^2 - \sum_{V \in \nats^d} (1-\rho)^{|V|} \wh{r}(V)^2
 \\ &\leq
 \rho \sum_{V \in \nats^d} |V| \wh{r}(V)^2,
  \end{align*}
  where the last inequality follows from Bernoulli's inequality
  $1 - \rho |V| \leq (1-\rho)^{|V|}$.
  We know that (see for example \cite{Sze67})
  \[
    \frac{\partial}{\partial x_i} H_V(\vec x)
    = \frac{\partial}{\partial x_i} \prod_{v_i \in V} H_{v_i} (x_i)
    = \prod_{v_j \in V\setminus v_i} H_{v_j}(x_j) \sqrt{v_i} H_{v_i-1}(x_i)
  \]
  Therefore,
  \[
    \frac{\partial r(\vec x)}{\partial{x_i}}
    = \sum_{V \in \nats^d} \wh{r}(V) \sqrt{v_i} H_{v_i-1}(x_i)
    \prod_{v_j \in V\setminus v_i} H_{v_j}(x_j)
  \]
  From Parseval's identity we have
  \[
    \E_{\vec x \sim \normal(\vec 0, \mat I)}
    \lp[
    \lp(
    \frac{\partial r(\vec x)}{\partial x_i}
    \rp)^2 \rp]
    =
    \sum_{V\in \nats^d}
    \wh{r}(V)^2 v_i.
  \]
  Therefore,
  \[
    \E_{\vec x \sim \normal(\vec 0, \mat I)}
    \lp[
    \snorm{2}{\nabla r(\vec x)}^2
    \rp]
    =
    \sum_{V\in \nats^d} |V| \wh{r}(V)^2 .
  \]
  The lemma follows.
\end{prevproofbig}

\subsection{Learning the Hermite Expansion}
In this section we present a way to bound the variance of
the empirical estimation of Hermite coefficients.
To bound the variance of estimating Hermite polynomials we shall need
a bound for the expected value of the fourth power of a Hermite polynomial.
\begin{lemma}\label{lem:fourthHermitePower}
  For any $V \in \nats^d$ it holds
  \(
    \E_{ \vec x \sim \normal_0} [H_V^4(\vec x)] \leq 9^{|V|}.
  \)
\end{lemma}
\begin{proof}
We compute
  \begin{align*}
    \E_{ \vec x \sim \normal_0} [H_V^4(\vec x)]
    &= \prod_{v_i \in V} \E_{x \sim \normal(0, 1)}[ H_{v_i}^2(x_i) H_{v_i}^2(x_i) ]
 \\ &=
 \prod_{v_i \in V} \E_{x \sim \normal(0, 1)}
 \lp[
 \lp(
 \sum_{r=0}^{v_i} \binom{v_i}{r} \frac{\sqrt{2 r!}}{ r!} H_{2r}(x_i) \rp)
 \lp(
 \sum_{r=0}^{v_i} \binom{v_i}{r} \frac{\sqrt{2 r!}}{ r!} H_{2r}(x_i)\rp)
 \rp]
 \\ &=
 \prod_{v_i \in V}  \sum_{r=0}^{v_i} \binom{v_i}{r}^2 \frac{(2 r)!}{(r!)^2}
 \E_{x \sim \normal(0, 1)}
 \lp[ H_{2r}(x_i)^2\rp]
 =
 \prod_{v_i \in V}  \sum_{r=0}^{v_i} \binom{v_i}{r}^2 \frac{(2 r)!}{(r!)^2}
 \\
    &\leq
    \prod_{v_i \in V}  \sum_{r=0}^{v_i} \binom{v_i}{r}^2 2^{2 r}
    \leq
    \prod_{v_i \in V}  \lp(\sum_{r=0}^{v_i} \binom{v_i}{r} 2^{r} \rp)^2
    \leq
    \prod_{v_i \in V}  9^{v_i}
    = 9^{|V|}.
  \end{align*}
  In the above computation we used the formula for the product of two (normalized)
  Hermite polynomials
  \[
    H_i(x) H_i(x) =
 \sum_{r=0}^{v_i} \binom{v_i}{r} \frac{\sqrt{2 r!}}{ r!} H_{2r}(x_i) ,
  \]
  see, for example, \cite{Sze67}.
\end{proof}

\begin{prevproofbig}{Lemma}{lem:hermiteCoefficientVariance}
  We have
  \begin{align*}
    \E_{ \vec x \sim \normal^*_S}
    [(H_V(\vec x) - c_V)^2]
    =
    \E_{ \vec x \sim \normal^*_S}
    [H_V^2(\vec x)] - c_V^2
    \leq
    \frac{1}{\alpha}
    \E_{\vec x \sim \normal^*} [H_V^2(\vec x)]
  \end{align*}
  We have
  \begin{align*}
    \lp| \E_{\vec x \sim \normal^*} [H_V^2(\vec x)]
    - 1
    \rp|
    &=
    \lp| \E_{\vec x \sim \normal^*} [H_V^2(\vec x)]
    - \E_{\vec x \sim \normal_0} [H_V^2(\vec x)]
    \rp|
\\  &\leq
\int
H_V^2(\vec x) |\normal^*(\vec x)
- \normal_0(\vec x) | \d \vec x
\\  &=
\int
H_V^2(\vec x) \sqrt{\normal_0(\vec x))}
\frac{|\normal^*(\vec x)
- \normal_0(\vec x) |}{\sqrt{\normal_0(\vec x)}} \d \vec x
\\  &\leq
\Bigg(
  \underbrace{
    \int
    H_V^4(\vec x) \normal_0(\vec x)
    \d \vec x
  }_{A}
\Bigg)^{1/2}
\Bigg(
  \underbrace{
    \int
    \frac{(\normal^*(\vec x)
    - \normal_0(\vec x))^2}{\normal_0(\vec x)}
    \d \vec x
  }_{B}
\Bigg)^{1/2}
  \end{align*}
  To bound term $A$ we use Lemma~\ref{lem:fourthHermitePower}.
  Using Lemma~\ref{lem:ratioBound} we obtain
  \begin{align*}
    B \leq \E_{x \sim \normal(\vec 0,\mat I)}
    \lp[
    \lp(
    \frac{\normal^*(\vec x)}{\normal_0(\vec x)}
    \rp)^2
    \rp]
    \leq \poly(1/\alpha).
  \end{align*}
  The bound for the variance follows from the independence of the samples.
\end{prevproofbig}
 \section{Missing Proofs of Section~\ref{sec:optimization}} \label{app:optimization}

\begin{prevproofbig}{Lemma}{lem:boundedVarianceStep}
  We have that $\lp|
    M_{\tind_k}(\vec u, \mat B) -
    M'_{\tind}(\vec u, \mat B)
    \rp| \le \lp|
    M_{\tind_k}(\vec u, \mat B) -
    M_{\tind}(\vec u, \mat B)
    \rp| + \lp|
    M_{\tind}(\vec u, \mat B) -
    M'_{\tind}(\vec u, \mat B)
    \rp|$.

  \noindent For the first term we have that
  \begin{align*}
    \lp|
    M_{\tind_k}(\vec u, \mat B) -
    M'_{\tind}(\vec u, \mat B)
    \rp|
    & \le
      C_{\vec u, \mat B}
          \E_{\vec x \sim \Nt_S}
          \lp[
          \frac{
            \N_0(\vec x)
          }{
            \N_{\vec u, \mat B}(\vec x)
          }
          \lp|
           \tind_k(\vec x) - \tind(\vec x) \rp|
          \rp] \\
  & \le
    C_{\vec u, \mat B}
        \sqrt{
        \E_{\vec x \sim \N_0}
        \lp[
        \lp(
        \frac{
          \Nt_S(\vec x)
        }{
          \N_{\vec u, \mat B}(\vec x)
        }
        \rp)^2
        \rp]
        \cdot
        \E_{\vec x \sim \N_0}
        \lp[
        (
         \tind_k(\vec x) - \tind(\vec x) )^2
        \rp]
      } \\
    & \le
      \frac{ C_{\vec u, \mat B} } {\alpha^*}
          \sqrt{
          \E_{\vec x \sim \N_0}
          \lp[
          \lp(
          \frac{
            \Nt(\vec x)
          }{
            \N_{\vec u, \mat B}(\vec x)
          }
          \rp)^2
          \rp]
          \cdot
          \E_{\vec x \sim \N_0}
          \lp[
          (
           \tind_k(\vec x) - \tind(\vec x) )^2
          \rp]
        }
        \intertext{now we can use Lemma \ref{lem:ratioBound},
        Lemma \ref{lem:cubBound} and Theorem
        \ref{thm:estimationApproximationError} to get}
        \lp|
        M_{\tind_k}(\vec u, \mat B) -
        M'_{\tind}(\vec u, \mat B)
        \rp|
        & \le \poly(1/\alpha^*) \sqrt{\eps}
  \end{align*}

  \noindent For the second term we have that
  \begin{align*}
    \lp|
    M_{\tind}(\vec u, \mat B) -
    M'_{\tind}(\vec u, \mat B)
    \rp|
    &\le
      \lp|1 - \frac{C'_{\vec u, \mat B}}{C_{\vec u, \mat B}} \rp| C_{\vec u, \mat B}
          \E_{\vec x \sim \Nt_S}
          \lp[
          \frac{
            \Nt(\vec x)
          }{
            \alpha^* \N_{\vec u, \mat B}(\vec x)
          }
          \rp]
  \end{align*}
  We need to bound

  \begin{align*}
  \lp|
  1 - \frac{C'_{\vec u, \mat B}}{C_{\vec u, \mat B}}
  \rp| &=
  \lp|
  1-
  e^{
  -\frac12 \lp(
  \tr( (\mat B - \mat I) (\S_S + \m_S \m_S^T - \wt{\S}_S) ) ) - \vec u^T \m_S
  \rp)}
  \rp| \le
  e^{
  \lp|\frac12 \lp(
  \tr( (\mat B - \mat I) (\S_S + \m_S \m_S^T - \wt{\S}_S) ) ) - \vec u^T \m_S
  \rp) \rp|} - 1\\
  &\le
    e^{
    \frac12 \lp(
    \|\mat B - \mat I\|_F \|\S_S + \m_S \m_S^T - \wt{\S}_S\|_F + \|\vec u\|_2
    \|\m_S\|_2 \rp)} - 1 \le \|\mat B - \mat I\|_F \|\S_S + \m_S \m_S^T -
    \wt{\S}_S\|_F + \|\vec u\|_2 \|\m_S\|_2
  \end{align*}
  where the last inequality holds when
  $\|\mat B - \mat I\|_F \|\S_S + \m_S \m_S^T - \wt{\S}_S\|_F + \|\vec u\|_2
  \|\m_S\|_2 \le 1$. But we know that $(\vec{u}, \mat{B}) \in \Domain$ and hence
  $\|\mat B - \mat I\|_F \le \poly(1/\alpha^*)$,
  $\|\vec u\|_2 \le \poly(1/\alpha^*)$. Also from Section
  \ref{sec:problemFormulation} we have that
  $\|\S_S + \m_S \m_S^T - \wt{\S}_S\|_F \le \eps$ and $\norm{\m_S}_2 \le \eps$
  and we can set $\eps$ to be any inverse polynomial in $1/\alpha^*$ times
  $\eps$. Hence we get
  \[ \lp|
  1 - \frac{C'_{\vec u, \mat B}}{C_{\vec u, \mat B}}
  \rp| \le \eps \]

  Now we can also use Lemma \ref{lem:cubBound} and Lemma \ref{lem:ratioBound}
  which imply that
  \[ C_{\vec u, \mat B}
      \E_{\vec x \sim \Nt_S}
      \lp[
      \frac{
        \Nt(\vec x)
      }{
        \alpha^* \N_{\vec u, \mat B}(\vec x)
      }
      \rp] \le \poly(1/\alpha^*) \]
  \noindent and therefore we have
  \[ \lp|
  M_{\tind}(\vec u, \mat B) -
  M'_{\tind}(\vec u, \mat B)
  \rp| \le \poly(1/\alpha^*) \eps. \]
  Hence we can once again divide $\eps$ by any polynomial of $1/\alpha^*$
  without increasing the complexity presented in Section
  \ref{sec:problemFormulation} and the lemma follows.
\end{prevproofbig}

\begin{prevproofbig}{Lemma}{lem:boundedVarianceStep}
  We apply successive Cauchy-Schwarz inequalities to separate the terms that appear in the expression for the squared norm of the gradient. We have that
  \begin{align*}
    \Exp_{\vec{x} \sim \Nt_S} \left[
      \norm{ \vec{v}(\vec{u}, \mat{B}) }_2^2
    \right] & =
    C^2_{\vec u, \matr B}
    \Exp_{\vec{x} \sim \Nt_S} \left[
      \left(
        \norm{\vec{x} \vec{x}^T - \tiS_S - \tim_S \tim_S^T}_F^2 +
         \norm{\tim_S - \vec{x}}_2^2
      \right)
      \frac { \N_0^2(\vec{x}) } {\N_{\vec u, \matr B}^2(x)} \tind_k^2(\vec{x})
    \right] \\
    & =
        C^2_{\vec u, \matr B}
        \Exp_{\vec{x} \sim \N_0} \left[
          \left(
            \norm{\vec{x} \vec{x}^T - \tiS_S - \tim_S \tim_S^T}_F^2 +
             \norm{\tim_S - \vec{x}}_2^2
          \right)
          \frac { \N_0(\vec{x})\Nt_S(\vec x) } {\N_{\vec u, \matr B}^2(x)} \tind_k^2(\vec{x})
        \right] \\
    & \le
        C^2_{\vec u, \matr B}
        \Exp_{\vec{x} \sim \N_0} \left[
          \left(
            \norm{\vec{x} \vec{x}^T - \tiS_S - \tim_S \tim_S^T}_F^2 +
             \norm{\tim_S - \vec{x}}_2^2
          \right)
          \frac { \N_0(\vec{x})\Nt_S(\vec x) } {\N_{\vec u, \matr B}^2(x)}
        \right]^{1/2}
        \Exp_{\vec{x} \sim \N_0} \left[
        \tind_k^4(\vec{x})
        \right]^{1/2} \\
    & \le
        C^2_{\vec u, \matr B}
        \Exp_{\vec{x} \sim \N_0} \left[
          \left(
            \norm{\vec{x} \vec{x}^T - \tiS_S - \tim_S \tim_S^T}_F^2 +
             \norm{\tim_S - \vec{x}}_2^2
          \right)^2
        \right]^{1/4} \\ & \quad \quad \quad
        \Exp_{\vec{x} \sim \N_0} \left[
          \frac { \N^2_0(\vec{x}){\Nt}^2_S(\vec x) } {\N_{\vec u, \matr B}^4(x)}
        \right]^{1/4}
        \Exp_{\vec{x} \sim \N_0} \left[
        \tind_k^4(\vec{x})
        \right]^{1/2} \\
      & \le
          C^2_{\vec u, \matr B}
          \Exp_{\vec{x} \sim \N_0} \left[
            \left(
              \norm{\vec{x} \vec{x}^T - \tiS_S - \tim_S \tim_S^T}_F +
               \norm{\tim_S - \vec{x}}_2
            \right)^4
          \right]^{1/4} \\ & \quad \quad \quad
          \Exp_{\vec{x} \sim \N_0} \left[
            \frac { \N^4_0(\vec{x}) } {\N_{\vec u, \matr B}^4(x)}
          \right]^{1/8}
          \Exp_{\vec{x} \sim \N_0} \left[
            \frac { {\Nt}^4_S(\vec x) } {\N_{\vec u, \matr B}^4(x)}
          \right]^{1/8}
          \Exp_{\vec{x} \sim \N_0} \left[
          \tind_k^4(\vec{x})
          \right]^{1/2}
  \end{align*}

  We now bound each term separately.
  \begin{itemize}
  \item By Lemma~\ref{lem:cubBound}, $C^2_{\vec u, \matr B} \le \poly(1/\alpha)$.
  \item Given that $(\tim_S, \tiS_S)$ are near-isotropic, $$\begin{aligned}\Exp_{\vec{x} \sim \N_0} &\left[
            \left(
              \norm{\vec{x} \vec{x}^T - \tiS_S - \tim_S \tim_S^T}_F +
               \norm{\tim_S - \vec{x}}_2
            \right)^4
          \right]^{1/4} \\ &\le \Exp_{\vec{x} \sim \N_0} \left[
            \left(
              \norm{\vec{x} \vec{x}^T}_F + \norm{\tiS_S}_F + \norm{\tim_S \tim_S^T}_F +
               \norm{\tim_S} + \norm{\vec{x}}_2
            \right)^4
          \right]^{1/4} \le d \poly(1/\alpha).\end{aligned}$$
  \item By Lemma~\ref{lem:ratioBound}, $$\Exp_{\vec{x} \sim \N_0} \left[
            \frac { \N^4_0(\vec{x}) } {\N_{\vec u, \matr B}^4(x)}
          \right]^{1/8}
          \Exp_{\vec{x} \sim \N_0} \left[
            \frac { {\Nt}^4_S(\vec x) } {\N_{\vec u, \matr B}^4(x)}
          \right]^{1/8} \le \poly(1/\alpha).$$
  \item For the last term, we have
          \begin{align*}
          \Exp_{\vec{x} \sim \N_0} \left[
            \tind_k^4(\vec{x})
          \right] & = \Exp_{\vec{x} \sim \N_0} \left[
            \left( \sum_{0 \le |V| \le k} \tilde{c}_V H_V(\vec{x}) \right)^4
          \right] \nonumber \\
          & \le 2^3 \sum_{0 \le |V| \le k} \tilde{c}^4_V
          \Exp_{\vec{x} \sim \N_0} \left[
            H_V^4(\vec{x})
          \right] \nonumber \\
          &
          \le 8
          \left(
            \sum_{0 \le |V| \le k} \tilde{c}^2_V
          \right)^2
          \cdot \left(
            \max_{0 \le |V| \le k} \left\{
              \Exp_{\vec{x} \sim \N_0} \left[
                H_V^4(\vec{x})
              \right]
            \right\}
          \right)
\end{align*}
          From Lemma \ref{lem:hermiteCoefficientVariance} and the
          conditioning on the event that the estimators of the Hermite coefficients
          are accurate we have that $(\tilde{c}_V - c_V)^2 \le 1$ and hence we get
          the following.
          $$
          \Exp_{\vec{x} \sim \N_0} \left[
            \tind_k^4(\vec{x})
          \right]
           \le 2^{10} d^{2 k} \left(
            \sum_{0 \le |V| \le \infty} c^2_V
          \right)^4
          \cdot
          \left(
            \max_{0 \le |V| \le k} \left\{
              \Exp_{\vec{x} \sim \N_0} \left[
                H_V^4(\vec{x})
              \right]
            \right\}
          \right)$$
          To bound
          $
              \Exp_{\vec{x} \sim \N_0} \left[
                H_V^4(\vec{x})
              \right]
              $
              we use Lemma~\ref{lem:fourthHermitePower}.
          Moreover, from Parseval's identity we obtain that
          $\sum_{0\leq |V| \leq \infty} c_V^2 = \E_{\vec x \sim \N_0} \psi^2(\vec x)$.
          From Lemma~\ref{lem:ratioBound} we get
          \[
            \E_{\vec x \sim \N_0} \psi^2(\vec x)
            \leq \ \frac{1}{\alpha} \E_{\vec x \sim N_0}
            \lp(\frac{\normal^*(\vec x)}{\normal_0(\vec x)} \rp)^2
            = \poly(1/\alpha).
          \]
          From Lemma~\ref{lem:hermiteCoefficientVariance} we obtain that
          \(
            \max_{0 \le |V| \le k} \left\{
              \Exp_{\vec{x} \sim \N_0} \left[
                H_V^4(\vec{x})
              \right]
            \right\}
            \leq 2^k.
            \)
            The result follows from the above estimates.
  \end{itemize}

\end{prevproofbig}

\begin{prevproofbig}{Lemma}{lem:strongConvexity}
  We will prove this lemma in two steps, first we will prove
  \begin{align} \label{eq:lem:strongConvexity:proof:1}
    \abs{\vec z^T
    \mathcal{H}_{M_{\tind_k}}(\vec u, \mat B)
    \vec z - \vec z^T
    \mathcal{H}_{M_{\tind}}(\vec u, \mat B)
    \vec z} \le \lambda
  \end{align}
  \noindent and then we will prove that
  \begin{align} \label{eq:lem:strongConvexity:proof:2}
    \vec z^T
    \mathcal{H}_{M_{\tind}}(\vec u, \mat B)
    \vec z \ge 2 \lambda
  \end{align}
  \noindent for some parameter $\lambda \ge \poly(\alpha^*)$. To prove
  \eqref{eq:lem:strongConvexity:proof:1} we define
  \[ p(\vec{z}; \vec{x}) = \lp(\vec z^T
    \begin{pmatrix}
      \frac12 \lp(\vec x \vec x^T - \tiS_S - {\tim}_S {\tim}_S^T \rp)^\flat
      \\
      {\tim}_S - \vec x
    \end{pmatrix}
  \rp)^2\]
  \noindent and we have that
  \begin{align*}
    & \abs{\vec z^T
    \mathcal{H}_{M_{\tind_k}}(\vec u, \mat B)
    \vec z - \vec z^T
    \mathcal{H}_{M_{\tind}}(\vec u, \mat B)
    \vec z} \\
    & ~~~~~~~~~~ = \E_{\vec x \sim \Nt_S} \lp[
      \me^{h(\vec u, \mat B; \vec x)}
      \normal(\vec 0, \mat I; \vec x) \cdot
      p(\vec{z}; \vec{x})
      \cdot \abs{\tind_k(\vec{x}) - \tind(\vec{x})}
    \rp] \\
    & ~~~~~~~~~~ = \E_{\vec x \sim \N_0} \lp[
      \me^{h(\vec u, \mat B; \vec x)}
      \cdot \1{S}(\vec{x}) \cdot \Nt(\vec{x}) \cdot
      p(\vec{z}; \vec{x})
      \cdot \abs{\tind_k(\vec{x}) - \tind(\vec{x})}
    \rp]
    \intertext{we then separate the terms using the Cauchy Schwarz inequality}
    & \abs{\vec z^T
    \mathcal{H}_{M_{\tind_k}}(\vec u, \mat B)
    \vec z - \vec z^T
    \mathcal{H}_{M_{\tind}}(\vec u, \mat B)
    \vec z} \\
    & ~~~~~~~~~~ \le \sqrt{ \E_{\vec x \sim \N_0} \lp[
        \me^{2 h(\vec u, \mat B; \vec x)}
        \cdot \1{S}(\vec{x}) \cdot \left( \Nt(\vec{x}) \right)^2 \cdot
        p^2(\vec{z}; \vec{x})
      \rp]
    }
    \cdot
    \sqrt{ \E_{\vec x \sim \N_0} \lp[
        \left( \tind_k(\vec{x}) - \tind(\vec{x}) \right)^2
      \rp]
    }
    \intertext{we apply now the Hermite concentration from Theorem
    \ref{thm:estimationApproximationError} and we get}
    & ~~~~~~~~~~ \le \sqrt{ \E_{\vec x \sim \N_0} \lp[
        \me^{2 h(\vec u, \mat B; \vec x)}
        \cdot \1{S}(\vec{x}) \cdot \left( \Nt(\vec{x}) \right)^2 \cdot
        p^2(\vec{z}; \vec{x})
      \rp]
    }
    \cdot
    \sqrt{ \eps
    } \\
    & ~~~~~~~~~~ \le \sqrt[4]{ \E_{\vec x \sim \Nt} \lp[
        \me^{4 h(\vec u, \mat B; \vec x)}
        \cdot \1{S}(\vec{x}) \cdot \left( \Nt(\vec{x}) \right)^2
        \left( \N_0(\vec{x}) \right)^2
      \rp]
    }
    \cdot
    \sqrt[4]{ \Exp_{\vec{x} \sim \Nt} \lp[
        p^4(\vec{z}; \vec{x})
      \rp]
    }
    \cdot
    \sqrt{ \eps
    } \\
    \intertext{we now use \eqref{eq:optimizationFunctionEquivalentExpression},
    Lemma \ref{lem:cubBound} and the fact that $\1{S}(\vec{x}) \le 1$ to get
    }
    & ~~~~~~~~~~ \le \sqrt[4]{ \E_{\vec x \sim \Nt} \lp[
       \me^{4 h(\vec u, \mat B; \vec x)}
       \left( \Nt(\vec{x}) \right)^2
       \left( \N_0(\vec{x}) \right)^2
     \rp]
   }
   \cdot
   \sqrt[4]{ \Exp_{\vec{x} \sim \Nt} \lp[
       p^4(\vec{z}; \vec{x})
     \rp]
   }
   \cdot
   \sqrt{ \eps
   }
  \end{align*}
  and finally we use Lemma \ref{lem:ratioBound} to prove the following
  \begin{align} \label{eq:lem:strongConvexity:proof:1:final}
    \abs{\vec z^T
    \mathcal{H}_{M_{\tind_k}}(\vec u, \mat B)
    \vec z - \vec z^T
    \mathcal{H}_{M_{\tind}}(\vec u, \mat B)
    \vec z}
    \le
    \sqrt[4]{ \Exp_{\vec{x} \sim \Nt} \lp[
        p^4(\vec{z}; \vec{x})
      \rp]
    }
    \cdot
    \poly(1/\alpha^*)
    \cdot \sqrt{ \eps }
  \end{align}

  \noindent Next we prove \eqref{eq:lem:strongConvexity:proof:2}. We have that
  \begin{align*}
    \vec z^T
    \mathcal{H}_{M_{\tind}}(\vec u, \mat B)
    \vec z & = \E_{\vec x \sim \Nt_S} \lp[
      \me^{h(\vec u, \mat B; \vec x)}
      \normal(\vec 0, \mat I; \vec x) \cdot
      p(\vec{z}; \vec{x})
      \cdot \tind(\vec{x})
    \rp] \\
    & = \frac{1}{\alpha^*} C_{\vec{u}, \mat{B}} \E_{\vec x \sim \Nt_S} \lp[
      \frac{\N^*(\vec{x})}{\N_{\vec{u}, \mat{B}}(\vec{x})}
      p(\vec{z}; \vec{x})
    \rp].
    \intertext{Now we define the set
    $\bar{Q}_{\vec{z}} = \left\{\vec{x} \in \R^{d} \mid
       \abs{p(\vec{z}; \vec{x})} \le \frac{1}{32 C} \left(\alpha^*\right)^4
       \sqrt[4]{\Exp_{\vec{x} \sim \Nt}\left[ p^4(\vec{z}; \vec{x})\right]}
    \right\}$, where $C$ is the universal constant guaranteed from
    Theorem \ref{thm:GaussianMeasurePolynomialThresholdFunctions}. Then
    using Theorem \ref{thm:GaussianMeasurePolynomialThresholdFunctions} and the
    fact that $p(\vec{z}; \vec{x})$ has degree $4$ we get that
    $\N(\mt, \St; \bar{Q}) \le \frac{\alpha^*}{2}$. Hence we define the set
    $S' = S \cap \bar{Q}$ and we have that $\N(\mt, \St; S') \ge \alpha^*/2$.
    }
    \vec z^T
    \mathcal{H}_{M_{\tind}}(\vec u, \mat B)
    \vec z & \ge \frac{1}{\alpha^*} C_{\vec{u}, \mat{B}}
    \E_{\vec x \sim \Nt_{S'}} \lp[
      \frac{\N^*(\vec{x})}{\N_{\vec{u}, \mat{B}}(\vec{x})}
      p(\vec{z}; \vec{x})
    \rp] \\
    & \ge \left( \min_{\vec{x} \in S'} p(\vec{z}; \vec{x})\right) \frac{1}{\alpha^*} C_{\vec{u}, \mat{B}}
    \E_{\vec x \sim \Nt_{S'}} \lp[
      \frac{\N^*(\vec{x})}{\N_{\vec{u}, \mat{B}}(\vec{x})}
    \rp]
    \intertext{and from the definition of $S'$ and Lemma \ref{lem:cubBound} we
    have that}
    \vec z^T
    \mathcal{H}_{M_{\tind}}(\vec u, \mat B)
    \vec z & \ge \poly(\alpha^*)
    \cdot
    \E_{\vec x \sim \Nt_{S'}} \lp[
      \frac{\N^*(\vec{x})}{\N_{\vec{u}, \mat{B}}(\vec{x})}
    \rp]
    \cdot
    \sqrt[4]{\Exp_{\vec{x} \sim \Nt}\left[ p^4(\vec{z}; \vec{x})\right]}
    \intertext{now we can apply Jensen's inequality on the convex function
    $x \mapsto 1/x$ and we get}
    \vec z^T
    \mathcal{H}_{M_{\tind}}(\vec u, \mat B)
    \vec z & \ge \poly(\alpha^*)
    \cdot \frac{1}
               {\E_{\vec x \sim \Nt_{S'}} \lp[
                  \frac{\N_{\vec{u}, \mat{B}}(\vec{x})}{\N^*(\vec{x})}
                \rp]
               }
    \cdot
    \sqrt[4]{\Exp_{\vec{x} \sim \Nt}\left[ p^4(\vec{z}; \vec{x})\right]} \\
    & \ge \poly(\alpha^*)
    \cdot \frac{1}
               {\sqrt{\E_{\vec x \sim \Nt} \lp[
                  \left( \frac{\N_{\vec{u}, \mat{B}}(\vec{x})}{\Nt(\vec{x})}
                  \right)^2
                \rp]}
               }
    \cdot
    \sqrt[4]{\Exp_{\vec{x} \sim \Nt}\left[ p^4(\vec{z}; \vec{x})\right]}
  \end{align*}
  \noindent finally using Lemma \ref{lem:ratioBound} we get
  \begin{align} \label{eq:lem:strongConvexity:proof:2:final}
    \vec z^T
    \mathcal{H}_{M_{\tind}}(\vec u, \mat B)
    \vec z & \ge \poly(\alpha^*)
      \sqrt[4]{\Exp_{\vec{x} \sim \Nt}\left[ p^4(\vec{z}; \vec{x})\right]}
  \end{align}

  Now using \eqref{eq:lem:strongConvexity:proof:1:final} and
  \eqref{eq:lem:strongConvexity:proof:2:final} we can see that it is possible to
  pick $\eps$ in the Hermite concentration to be the correct polynomial in
  $\alpha^*$ so that
  \[ \abs{\vec z^T
    \mathcal{H}_{M_{\tind_k}}(\vec u, \mat B)
    \vec z - \vec z^T
    \mathcal{H}_{M_{\tind}}(\vec u, \mat B)
    \vec z} \le \vec z^T
    \mathcal{H}_{M_{\tind}}(\vec u, \mat B)
    \vec z
  \]
  \noindent which implies from Jensen's inequality that
  \begin{align*}
    \vec z^T
  \mathcal{H}_{M_{\tind_k}}(\vec u, \mat B)
  \vec z & \ge \poly(\alpha^*)
    \sqrt[4]{\Exp_{\vec{x} \sim \Nt}\left[ p^4(\vec{z}; \vec{x})\right]} \\
         & \ge \poly(\alpha^*) \Exp_{\vec{x} \sim \Nt}\left[ p(\vec{z}; \vec{x}) \right]
  \end{align*}
  \noindent So the last step is to prove a lower bound for
  $\Exp_{\vec{x} \sim \Nt}\left[ p(\vec{z}; \vec{x}) \right]$. For this we can
  use the Lemma 3 of \cite{DGTZ18} from which we can directly get
  $\Exp_{\vec{x} \sim \Nt}\left[ p(\vec{z}; \vec{x}) \right] \ge
  \poly(\alpha^*)$ and the lemma follows.
\end{prevproofbig}
 \section{Details of Section~\ref{sec:set_recover}}
\label{app:setRecover}
We present here the of the proof of Theorem~\ref{thm:pacLearning}.  We
already proved that given only positive examples from a truncated
normal can obtain arbitrarily good estimations of the unconditional
(true) parameters of the normal using
Algorithm~\ref{alg:projectedSGD}.
Recall also that with positive samples we can obtain an approximation
of the function $\psi(\vec x)$ defined in
\ref{eq:definitionOfShiftedIndicator}.  From
Theorem~\ref{thm:estimationApproximationError} we know that with
$d^{\poly(1/\alpha) \Gamma(S)^2/\eps^4}$ samples we can obtain a
function $\tind_k(\vec x)$ such that
\[
  \E_{\vec x \sim \normal_0}[((\tind_k(\vec x) -\tind(\vec x))^2] \leq \eps.
\]
Now we can construct an almost indicator function using $\tind_k$ and the learned
parameters $\wt{\vec \mu}$, $\wt {\mat I}$.
We denote
$\wt{\normal}=\normal(\wt{\vec \mu},\wt{\mat \Sigma})$.
\begin{equation}\label{eq:apxMultIndicator}
  \wt{f}\vec(x) = \frac{\normal_0(\vec x)}{\wt{\normal}(\vec x)} \tind_k(\vec x).
\end{equation}
This function should be a good enough approximation to the function
\begin{equation}\label{eq:trueMultIndicator}
  f(\vec x) = \frac{\normal_0(\vec x)}{\normal^*(\vec x)} \tind(\vec x)
  = \frac{\1{S}(\vec x)}{\alpha^*}.
\end{equation}
Notice that even though we do not know the mass of the truncation set $\alpha^*$ we
can still construct a threshold function that achieves low error with respect to the
zero-one loss.
We first prove a standard lemma that upper bounds the zero-one loss with the
distance of $f$ and $\wt{f}$.
We prove it so that we have a version consistent with our notation.
\begin{lemma}\label{lem:zeroOneToRoot}
  Let $S$ be a subset of $\R^d$.  Let $D$ be a distribution on
  $\R^d$ and let $f: \R^d \to \{0, B\}$, where $B > 1$ such that
  \( f(\vec x) = B\ \1{S}(\vec x) \).
  For any $g:\R^d \mapsto [0,+\infty)$ it holds
  \(
\E_{\vec x \sim D} \lp[\vec{1}\{g(x) > 1/2)\} \neq \1{S}(\vec x)\} \rp]
  \leq
  \sqrt{2} \E_{\vec x \sim D}\lp[\sqrt{|g(x) - f(x)|}\rp].
  \)
\end{lemma}
\begin{proof}
  It suffices to show that for all $x \in \R^d$ it holds
  \begin{equation}\label{eq:zeroOneL1}
    \bm{1}\{\sgn(g(x) -1/2) \neq S(\vec x)\} \leq \sqrt{2} \sqrt{|g(x) - f(x)|}.
  \end{equation}
  We only need to consider the case where $\sgn(g(x) -1/2) \neq S(\vec x)$.
  Assume first that $g(x) > 1/2$ and $x\neq S$. Then the LHS of
  Equation~\eqref{eq:zeroOneL1} is $1$ and the RHS of \eqref{eq:zeroOneL1} is
  \(
  \sqrt{2} \sqrt{|g(x) - f(x)|} \geq \sqrt{2} \sqrt{|1/2 - 0|} = 1.
  \)
  Assume now that $g(x)<1/2$ and $S(\vec x)=1$. Then the RHS of \eqref{eq:zeroOneL1}
  equals
  \(
  \sqrt{2} \sqrt{|g(x) - f(x)|} \geq \sqrt{2} \sqrt{|B-1/2|} \geq 1.
  \)
\end{proof}
We now state the following lemma that upper bounds the distance of
$f$ and $\wt{f}$ in with the sum of the total variation distance of the
true and learned distributions as well as the approximation error of $\psi_k$.
\begin{lemma}\label{lem:rootError}
  Let $\alpha$ be the absolute constant of \eqref{eq:globalLowerBound}.
  Let $S \subseteq \R^d$ and let $\normal^*, \wt{\normal}$ be
  $(O(\log(1/\alpha)),1/16)$-isotropic.
  Let $\psi$ be as in \eqref{eq:definitionOfShiftedIndicator}.
  Moreover, let $\wt{f}, f$ be as in
  \eqref{eq:apxMultIndicator}, \eqref{eq:trueMultIndicator}.
  Then,
  \[
    \E_{\vec x \sim \normal^*}\lp[\sqrt{|\wt{f}(\vec x) -  f(\vec x)|}\rp] \leq
    \poly(1/\alpha)\
    \lp( \lp(\E_{\vec x \sim \normal_0} \lp[ (\tind_k(\vec x) - \tind(\vec x))^2 \rp] \rp)^{1/4}
    + \lp( \dtv{\normal^*}{\wt{\normal}} \rp)^{1/4}
    \rp)
  \]
\end{lemma}
\begin{proof}
  We compute
  \begin{align*}
    \E_{\vec x \sim \normal^*}\lp[\sqrt{|\wt{f}(\vec x) - f(\vec x)|} \rp]
    &\leq
    \E_{\vec x \sim \normal^*}
    \lp[ \lp(\lp|
    \tind_k(\vec x) \frac{\normal_0(\vec x)}{\wt{\normal}(\vec x)} -
    \tind(\vec x) \frac{\normal_0(\vec x)}{\normal^*(\vec x)}
    \rp|\rp)^{1/2} \rp]
 \\ &=
 \E_{\vec x \sim \normal^*}
 \lp[
 \lp(\lp|
 \tind_k(\vec x) \frac{\normal_0(\vec x)}{\wt{\normal}(\vec x)}
 - \tind(\vec x) \frac{\normal_0(\vec x)}{\wt{\normal}(\vec x)}
 + \tind(\vec x) \frac{\normal_0(\vec x)}{\wt{\normal}(\vec x)}
 - \tind(\vec x) \frac{\normal_0(\vec x)}{\normal^*(\vec x)}
 \rp| \rp)^{1/2}\rp]
 \\ &\leq
 \E_{\vec x \sim \normal^*}
 \lp[
 \lp(
 |\tind_k(\vec x) - \tind(\vec x)| \frac{\normal_0(\vec x)}{\wt{\normal}(\vec x)}
 \rp)^{1/2}
 \rp]
 +
 \E_{\vec x \sim \normal^*}
 \lp[
 \lp(
 \tind(\vec x)
 \lp|\frac{\normal_0(\vec x)}{\wt{\normal}(\vec x)}
 - \frac{\normal_0(\vec x)}{\normal^*(\vec x)}\rp|
 \rp)^{1/2}
 \rp]
 \\ &\leq
 \Bigg(
 \underbrace{
   \E_{\vec x \sim \normal^*}
 \lp[ |\tind_k(\vec x) - \tind(\vec x)| \rp]
 }_{A}
 \Bigg)^{1/2}
 \Bigg(
   \underbrace{
     \E_{\vec x \sim \normal^*}
     \lp[
     \frac{\normal_0(\vec x)}{\wt{\normal}(\vec x)}
     \rp]
   }_{B}
 \Bigg)^{1/2}
 \\
    &+
    \Bigg(
    \underbrace{
      \E_{\vec x \sim \normal^*}
      \lp[
      \tind(\vec x)
      \lp|
      \frac{\normal_0(\vec x)}{\wt{\normal}(\vec x)}
      - \frac{\normal_0(\vec x)}{\normal^*(\vec x)}\rp|
    \rp] }_{C}
  \Bigg)^{1/2}
  \end{align*}
  where for term $C$ we used Jensen's inequality.
  Using Lemma~\ref{lem:errorTransfer} and Lemma~\ref{lem:ratioBound} we have that
  \[
  A \leq
  \lp(
  \E_{\vec x \sim \normal_0} \lp[ (\tind_k(\vec x) - \tind(\vec x))^2 \rp]
  \rp)^{1/2}
  \lp(
   \E_{\vec x \sim \normal^*} \lp[
    \frac{\normal_0(\vec x)}{\normal^*(\vec x)}
   \rp]
  \rp)^{1/2}
  \leq
  \lp(
  \E_{\vec x \sim \normal_0} \lp[ (\tind_k(\vec x) - \tind(\vec x))^2 \rp]
  \rp)^{1/2}
  \poly(1/\alpha)
\]
  Since $\normal_0, \wt{\normal}$, and $\normal^*$ are $(O(\log(1/\alpha), 1/16)$-isotropic,
  using Lemma~\ref{lem:ratioBound} we obtain that
  \[
    B =
    \E_{\vec x \sim \normal_0}
    \lp[
    \frac{\normal_0(\vec x)}{\wt{\normal}(\vec x)}
    \frac{\normal^*(\vec x)}{\wt{\normal}(\vec x)}
    \rp]
    \leq
    \lp(
    \E_{\vec x \sim \normal_0}
    \lp[
    \frac{\normal_0(\vec x)}{\wt{\normal}(\vec x)}
    \rp]
    \rp)^{1/2}
    \lp(
    \E_{\vec x \sim \normal_0}
    \lp[
    \frac{\normal^*(\vec x)}{\wt{\normal}(\vec x)}
    \rp]
    \rp)^{1/2}
    \leq
    \poly(1/\alpha)
  \]
  We now bound term $C$. We write
  \begin{align}\label{eq:CBound}
    C&=
    \E_{\vec x \sim \normal^*}
    \lp[
    \tind(\vec x)
    \lp|
    \frac{\normal_0(\vec x)}{\wt{\normal}(\vec x)}
    - \frac{\normal_0(\vec x)}{\normal^*(\vec x)}\rp|
    \rp]
    =
  \frac{1}{\alpha^*}
  \E_{\vec x \sim \normal^*}
  \lp[
  \lp| \frac{\normal^*(\vec x)}{\wt{\normal}(\vec x)} - 1\rp|
  \rp]
  \end{align}
  To simplify notation, let
  \(
    \ell(\vec x) = \lp| \frac{\normal^*(\vec x)}{\wt{\normal}(\vec x)} - 1\rp|.
    \) Moreover, notice that
    \(
    \E_{\vec x \sim \wt{\normal}}[\ell(\vec x)] = \dtv{\normal^*}{\wt{\normal}}.
    \)
    Using the second bound of Lemma~\ref{lem:errorTransfer} and Lemma~\ref{lem:ratioBound}
    we obtain
    \[
      C \leq  \frac{1}{\alpha} \dtv{\normal^*}{\wt{\normal}} + \poly(1/\alpha) \sqrt{\dtv{\normal^*}{\wt{\normal}}}
      \leq \poly(1/\alpha) \sqrt{\dtv{\normal^*}{\wt{\normal}}}.
  \]
  Combining the bounds for $A,B$ and $C$ we obtain the result.
\end{proof}
Since we have the means two make both errors of Lemma~\ref{lem:rootError} small we
can now recover the unknown truncation set $S$.
\begin{prevproofbig}{Theorem}{thm:pacLearning}
  We first run Algorithm~\ref{alg:projectedSGD} to find estimates $\wt{\vec
  \mu}$, $\wt {\mat \Sigma}$.  From
  Theorem~\ref{thm:mainTheoremLocalConvergence} we know that $N =
  d^{\poly(1/\alpha) \Gamma^2(\mcal{S})/\eps^{32}}$ samples suffice to obtain
  parameters $\wt{\vec{\mu}}$, $\wt{\mat{\Sigma}}$ such that $
  \dtv{\normal(\vec \mu^*, \mat \Sigma^*)}{\normal(\wt{\vec \mu}, \wt{\mat
  \Sigma})} \leq \poly(\alpha) \eps^4.$ Notice, that from
  Theorem~\ref{thm:polynomialApproximation} we also know that $N$ samples from
  the conditional distribution $\normal^*_S$ suffice to learn a function
  $\psi_k$  such that $\E_{\vec x \sim \normal_0}[(\psi_k(\vec x) - \psi(\vec
  x))^2] \leq \poly(\alpha) \eps^4$.  Now we can construct the approximation $
  \wt{f}(\vec x) = \psi_k(\vec x)\normal_0(\vec x)/\wt{\normal}(\vec x)$.  Let
  our indicator $\wt{S} = \vec1\{\wt{f}(\vec(x) > 1/2\}$ and from
  Lemma~\ref{lem:zeroOneToRoot} and Lemma~\ref{lem:rootError} we obtain the
  result.
\end{prevproofbig}
\iffalse{
\begin{prevproofbig}{Lemma}{lem:zeroOneToRoot}
  It suffices to show that for all $x \in \R^d$ it holds
  \begin{equation}\label{eq:zeroOneL1}
    \bm{1}\{\sgn(g(x) -1/2) \neq S(\vec x)\} \leq \sqrt{2} \sqrt{|g(x) - f(x)|}.
  \end{equation}
  We only need to consider the case where $\sgn(g(x) -1/2) \neq S(\vec x)$.
  Assume first that $g(x) > 1/2$ and $x\neq S$. Then the LHS of
  Equation~\eqref{eq:zeroOneL1} is $1$ and the RHS of \eqref{eq:zeroOneL1} is
  \(
  \sqrt{2} \sqrt{|g(x) - f(x)|} \geq \sqrt{2} \sqrt{|1/2 - 0|} = 1.
  \)
  Assume now that $g(x)<1/2$ and $S(\vec x)=1$. Then the RHS of \eqref{eq:zeroOneL1}
  equals
  \(
  \sqrt{2} \sqrt{|g(x) - f(x)|} \geq \sqrt{2} \sqrt{|B-1/2|} \geq 1.
  \)
\end{prevproofbig}
}
\fi

\begin{lemma}\label{lem:errorTransfer}
  Let $P,Q$ be two distributions on $\R^d$ such that
  $P(\vec x), Q(\vec x) > 0$ for all $\vec x$
  and $\ell:\R^d \mapsto \R$ be a function.
  Then it holds
  \[
    \lp|
    \E_{\vec x \sim P}[\ell(\vec x)] - \E_{\vec x \sim Q}[\ell(\vec x)]
    \rp|
    \leq
    \Bigg(\E_{\vec x \sim P}[\ell^2(\vec x)] \E_{\vec x \sim P} \Bigg)^{1/2}
\Bigg(
    \lp[
    \lp(\frac{Q(\vec x)}{P(\vec x)} \rp)^2
    \rp]
  \Bigg)^{1/2}
  \]
  Moreover,
  \[
    \lp|
    \E_{\vec x \sim P}[\ell(\vec x)] - \E_{\vec x \sim Q}[\ell(\vec x)]
    \rp|
    \leq
    2
    \Bigg(
      \lp(\E_{\vec x \sim P}[\ell^2(\vec x)]
      +
      \E_{\vec x \sim Q}[\ell^2(\vec x)]
      \rp) \
    \Bigg)^{1/2}  \sqrt{\dtv{P}{Q}}
  \]
\end{lemma}
\begin{proof}
  Write
  \begin{align*}
    \lp|
    \E_{\vec x \sim P}[\ell(\vec x)] - \E_{\vec x \sim Q}[\ell(\vec x)]
    \rp|
    &\leq
    \int \ell(\vec x) \sqrt{P(\vec x)}
    \frac{|P(\vec x) - Q(\vec x)|}{\sqrt{P(\vec x)}}  \d x
    \\ &=
    \Bigg(\int \ell^2(\vec x) P(\vec x)
      \d x
      \int
      \frac{(P(\vec x) - Q(\vec x))^2}{P(\vec x)}  \d x
  \Bigg)^{1/2}
\end{align*}
For the second inequality we have
  \begin{align*}
    \lp|\E_{\vec x \sim P}[\ell(\vec x)]
    - \E_{\vec x \sim Q}[\ell(\vec x)]
    \rp|
     &\leq
    \int \ell(\vec x) |P(\vec x) - Q(\vec x) |
    \d x
    \\ &\leq
    \int \ell(\vec x) \sqrt{P(\vec x) + Q(\vec x)}
    \
    \frac{|P(\vec x) - Q(\vec x) |}
    {\sqrt{P(\vec x) + Q(\vec x)}}
    \
    \d x
    \\ &\leq
    \lp(
    \E_{x \sim P}[\ell^2(x)] +
    \E_{x \sim Q}[\ell^2(x)]
    \rp)^{1/2}
    \lp(
    \int \frac {(P(\vec x) - Q(\vec x))^2} {P(\vec x) + Q(\vec x)} \d x
    \rp)^{1/2}
  \end{align*}
  Now observe that
  \begin{align*}
    \lp(
    \int \frac{(P(\vec x) - Q(\vec x))^2} {P(\vec x) + Q(\vec x)} \d x
    \rp)^{1/2}
    &\leq
    \lp( 2 \int \lp(\sqrt{P(\vec x)} - \sqrt{Q(\vec x)}\rp)^2 \d x \rp)^{1/2}
     &=
    2 \dhel{P}{Q}
    \leq 2 \sqrt{\dtv{P}{Q}}
  \end{align*}

\end{proof}

 \section{Missing Proofs of Section~\ref{sec:IdentifiabilityGaussian}}
\label{app:IdentifiabilityGaussian}
In the following we use the polynomial norms.  Let $p(\vec x) = \sum_{V: |V|
\leq k} c_V x^V$ be a multivariate polynomial.  We define the $\norm{p}_\infty
= \max_{V: |V| \leq k} |c_V|$, $\norm{p}_1 = \sum_{V:|V|\leq k} |c_V|$.

\begin{prevproofbig}{Lemma}{lem:TVDPolynomial}
  Let $W = S_1 \cap S_2 \cap \{f_1 > f_2\} \cup S_1\setminus S_2$,
  that is the set of points where the first density is larger than the
  second.  We now write the $L_1$ distance between $f_1, f_2$ as
  \[
    \int |f_1(\vec x) - f_2(\vec x)| \d \vec x =
    \int \1{W}(\vec x)
    (f_1(\vec x) - f_2(\vec x)) \d \vec x
  \]
  Denote $p(\vec x)$ the polynomial that will do the approximation
  of the $L_1$ distance.
  From Lemma~\ref{lem:chiSquaredExistence} we know that there exists
  a Normal distribution within small chi-squared divergence of both
  $\normal(\vec \mu_1, \matr \Sigma_1)$ and $\normal(\vec \mu_2, \matr \Sigma_2)$.
  Call the density function of this distribution $g(\vec x)$.
  We have
  \begin{align}\label{eq:tvd_apx_int}
    \Big| \int |f_1(\vec x) &- f_2(\vec x)|\d \vec x -
    \int p(\vec x) (f_1(\vec x) - f_2(\vec x))
    \Big|
    \\
                            &= \lp| \int (\1{W}(\vec x) - p(\vec x))\ (f_1(\vec x) - f_2(\vec x))\d \vec x \nonumber
                            \rp|
                            \\
                            &\leq  \int |\1{W}(\vec x) - p(\vec x)|\ |f_1(\vec x) - f_2(\vec x)| \d \vec x \nonumber
                            \\
                            &\leq  \int |\1{W}(\vec x) - p(\vec x)| \sqrt{g(\vec x)}\ \frac{|f_1(\vec x) - f_2(\vec x)|}{\sqrt{g(\vec x)}} \d x \nonumber
                            \\
                            &\leq  \sqrt{\int (\1{W}(\vec x) - p(\vec x))^2 g(\vec x) \d \vec x}
                            \sqrt{\int \frac{(f_1(\vec x) - f_2(\vec x))^2}{g(\vec x)} \d \vec x},
  \end{align}
  where we use Schwarzs' inequality.  From Lemma~\ref{lem:chiSquaredExistence}
  we know that
  \[
    \int \frac{f_1(\vec x)^2}{g(\vec x)} \d \vec x
    \leq \int \frac{\normal(\vec \mu_1, \matr \Sigma_1; \vec x)^2}{g(\vec x)}
    \d \vec x
    = \exp(\poly(1/\alpha)).
  \]
  Similarly, $\int \frac{f_2(\vec x)^2}{g(\vec x)} \d x = \exp(\poly(1/\alpha))$.
  Therefore we have,
  \[
    \Big|
    \int |f_1(\vec x) - f_2(\vec x)|\d \vec x - \int p(\vec x) (f_1(\vec x) - f_2(\vec x))
    \Big|
    \leq \exp(\poly(1/\alpha)) \sqrt{\int (\1{W}(\vec x) - p(\vec x))^2 g(\vec x) \d \vec x}
  \]
  Recall that $g(\vec x)$ is the density function of a Gaussian distribution, and
  let $\vec \mu, \matr \Sigma$ be the parameters of this Gaussian.  Notice that
  it remains to show that there exists a good approximating polynomial $p(\vec
  x)$ to the indicator function $\1{W}$.  We can now transform the space so that
  $g(\vec x)$ becomes the standard normal.  Notice that this is an affine
  transformation that also transforms the set $W$; call the transformed set
  $W^t$.  We now argue that the Gaussian surface area of the transformed set
  $W^t$ at most a constant multiple of the Gaussian surface area of the original
  set $W$.    Let $\normal(\vec \mu_i, \matr \Sigma_i; S_i) = \alpha_i$ for $i =1,2$
  and let $h_1(\vec x) = \normal(\bm \mu_1, \matr \Sigma_1; \vec x)/\alpha_1$ resp.
  $h_2(\vec x) = \normal(\vec \mu_2, \matr \Sigma_2; \vec x)/\alpha_2$
  be the density of first resp. second Normal ignoring the truncation sets $S_1,
  S_2$.
  Notice that instead of $f_1, f_2$ we may use $h_1, h_2$ in the definition of $W$,
  that is
  \[
    W = (S_1 \cap S_2 \cap \{h_1 \geq h_2\}) \cup S_1 \setminus S_2.
  \]
  Now, since $\matr \Sigma^{-1/2} > 0$ we have that the affine
  map $T(x) = \matr \Sigma^{-1/2}(\vec x - \vec \mu)$ is a bijection.
  Therefore $T(A \cap B) = T(A) \cap T(B)$ and $T(A \cup B) = T(A) \cup T(B)$.
  Similarly to $W^t= T(W)$, let $S_1^t, S_2^t$, $\{h_1 \geq h_2\}^t$ be
  the transformed sets.  Therefore,
  \[
    W^t = (S_1^t \cap S_2^t \cap \{h_1 \geq h_2\}^t) \cup S_1^t \setminus S_2^t.
  \]
  We will use some elementary properties of Gaussian surface area (see
  for example Fact 17 of \cite{KOS08}).  We have that for any sets $S_1, S_2$
  $\Gamma(S_1 \cap S_2)$ and $\Gamma(S_1 \cup S_2)$ are upper bounded from
  $\Gamma(S_1) + \Gamma(S_2)$.  Moreover, $\Gamma(S_1 \setminus S_2) \leq
  \Gamma(S_1) +\Gamma(S_2^c) = \Gamma(S_1) + \Gamma(S_2)$.
  From our assumptions, we know that the Gaussian surface area of the
  sets $S_1^t, S_2^t$ is $O(\Gamma(\mcal{S})$.  Notice now that the
  set $\{h_1 \geq h_2\}^t$ is a degree $2$ polynomial threshold function.
  Therefore, using the result of \cite{Kan11} (see also Table~\ref{tab:surface-area})
  we obtain that $\Gamma(\{h_1 \geq h_2\}^t) = O(1)$.
  Combining the above we obtain that $\Gamma(W^t) = O(\Gamma(\mcal{S})$.
  To keep the notation simple we from now on we will by $W$ the
  transformed set $W^t$.  Now, assuming that a good approximating polynomial
  $p(\vec x)$ of degree $k$ exists with respect to $\normal(\vec 0, \matr I)$ then
  $p(\matr \Sigma^{-1/2} (\vec x - \bm \mu))$ is a polynomial of degree $k$ that
  approximates $\1{W}(\vec x)$ with respect to $g(\vec x)$.
  Since $\1{W} \in L^2(\R^d, \N_0)$ we can approximate it using Hermite
  polynomials.  For some $k \in \N$ we set $p(\vec x) = S_k \1{W}(x)$, that is
  \[
    p_k(\vec x) = \sum_{V: |V| \leq k} \wh{\1{W}} H_V(\vec x).
  \]
  Combining Lemma~\ref{lem:realSensitivityConcentration} and
  Lemma~\ref{lem:KOS08NoiseSurface} we obtain
  \[
    \E_{\vec x \sim \normal_0}[(\1{W}(\vec x) - p_k(x))^2] =
    O\lp( \frac{\Gamma(\mcal{S})} {k^{1/2}} \rp).
  \]
  Therefore,
  \[
    \Big|
    \int |f_1(\vec x) - f_2(\vec x)|\d \vec x - \int p_k(\vec x) (f_1(\vec x) - f_2(\vec x))
    \Big|
    = \exp(\poly(1/\alpha)) \frac{\Gamma(\mcal{S})^{1/2}}{k^{1/4}}
  \]
  Therefore, ignoring the dependence on the absolute constant $\alpha$,
  to achieve error $O(\eps)$ we need degree $k = O(\Gamma(\mcal{S})^2/\eps^4)$.

  To complete the proof, it remains to obtain a bound for the coefficients of
  the polynomial $q(\vec x) = p_k(\matr \Sigma^{-1/2} (\vec x-\vec \mu))$.
We use the standard notation of polynomial norms, e.g.
  $\norm{p}_\infty$ is the maximum (in absolute value) coefficient,
  $\norm{p}_1$ is the sum of the absolute values of all coefficients etc.
  From Parseval's identity we obtain that the sum of the squared weights is
  less than $1$ so these coefficients are clearly not large.  The large
  coefficients are those of the Hermite Polynomials.  We consider first the $1$
  dimensional Hermite polynomial and take an even degree Hermite polynomial
  $H_n$.  The explicit formula for the $k$-th degree coefficient is
  \[
    \frac{2^{k/2 - n/2} \sqrt{n!}}{\lp(n/2 - k/2 \rp)! k!} \leq 2^n,
  \]
  see, for example, \cite{Sze67}.
  Similarly, we show the same bound when the degree of the Hermite polynomial
  is odd.
  Therefore, we have that the maximum coefficient of
  $H_V(x) = \prod_{i=1}^d H_i(x_i)$ is at most
  $\prod_{i=1}^d 2^{v_i} = 2^{\sum_{i=1}^d v_i} = 2^{|V|}$.
  Using Lemma~\ref{lem:polynomialInfinityNormBound} we obtain that
  \begin{align*}
    \norm{H_V(\matr \Sigma^{-1/2} (\vec x - \vec \mu))}_1
    &\leq
    \binom{d+|V|}{|V|} 2^{|V|}
    \lp(\sqrt{d} \norm{\matr \Sigma^{-1/2}}_2
    + \norm{\matr \Sigma^{-1/2} \vec \mu}_2\rp)^{|V|}
    \\
    &\leq \binom{d+|V|}{|V|} (4d)^{|V|/2} (O(1/\alpha^2))^{|V|}
  \end{align*}

  Now we have
  \[
    \norm{q(\vec x)}_\infty \leq
    \sum_{V: |V| \leq k}
    |c_V|  \norm{H_V(\matr \Sigma^{-1/2}(\vec x - \vec \mu))}_\infty
    \leq \binom{d+k}{k}^2 (4d)^{k/2} (O(1/\alpha^2))^{k},
  \]
  where we used the fact that since $\sum_{V} |c_v|^2 \leq 1$ it holds that
  $|c_V| \leq 1$ for all $V$.
  To conclude the proof we notice that we can pick the degree $k$ so that
  \[
    \lp| \int q(\vec x) (f_1(\vec x)- f_2(\vec x)) \rp| =
    \lp| \sum_{V: |V| \leq k} \vec x^V (f_1(\vec x) - f_2(\vec x)) \rp|
    \geq \eps /2.
  \]
  Since the maximum coefficient of $q(\vec x)$ is bounded by $d^{O(k)}$
  we obtain the result.
\end{prevproofbig}

\begin{prevproofbig}{Theorem}{thm:momentMatching}
  We first draw $O(d^2/\eps^2)$ and compute estimates of the conditional mean
  $\wb \mu_C$ and covariance $\wb \Sigma_C$ that satisfy the guarantees of
  Lemma~\ref{lem:conditionalEstimation}.  We now transform the space so that
  $\wb \mu_C = \vec 0$ and $\matr{\Sigma}_C = \matr I$.  For simplicity we
  still denote $\bm{\mu}$ and $\matr \Sigma$ the parameters of the unknown
  Gaussian after the transformation.  From
  Lemma~\ref{lem:conditionalParameterDistance} we have that $\norm{\matr
  \Sigma^{-1/2} \vec \mu}_2 \leq O(\log(1/\alpha)^{1/2}/\alpha)$, and
  $\Omega(\alpha^2) \leq \norm{\matr \Sigma^{1/2}}_2 \leq O(1/\alpha^2)$.  Let
  $\wt m_V$ be the empirical moments of $\normal(\vec \mu, \matr \Sigma, S)$,
  \(
    \wt m_V = \frac{\sum_{i=1}^N \vec x^V}{N}.
  \)
  We first bound the variance of a moment $\vec x^V$.
  \[
    \Var_{\vec x \sim \normal(\bm \mu, \matr \Sigma, S)}[\vec x^V]
    \leq \E_{\vec x \sim \normal(\bm \mu, \matr \Sigma, S)} [\vec x^{2V}]
    \leq \frac{1}{\alpha} \E_{\vec x \sim \normal(\bm \mu, \matr \Sigma)} [\vec x^{2V}]
    = \frac{1}{\alpha} \E_{\vec x \sim \normal(\bm 0, \matr I)} [(\matr \Sigma^{1/2} \vec x +\vec \mu)^{2V}]
  \]
  Following the proof of Lemma~\ref{lem:polynomialInfinityNormBound} we get that
  $\norm{(\matr \Sigma^{1/2} \vec x + \vec \mu)^{2V}}_\infty
  \leq (\sqrt{d} \norm{\matr\Sigma^{1/2}}_2 + \norm{\vec \mu}_2)^{|V|}.
  $
  Using Lemma~\ref{lem:TVDPolynomial} we know that if we set
  $k = \Gamma(\mcal{S})/\eps^4$ then given any guess of the parameters
  $\wb \mu, \wb \Sigma, \wt S$ we can check whether the corresponding
  truncated Gaussian $\normal(\wb \mu, \wb \Sigma, \wt S)$ is in total
  variation distance $\eps$ from the true by checking that all moments
  $\E_{x \sim \normal(\wb \mu, \wb \Sigma, \wt S)}[ \vec x^V]$ of the
  guess are close to the (estimates) of the true moments.
  Using the above observations and ignoring the dependence on the constant
  $\alpha$ we get that $\norm{(\matr \Sigma^{1/2} \vec x + \vec
  \mu)^{2V}}_\infty \leq d^{O(k)}$.  Chebyshev's inequality implies that with
  $d^{O(k)}/\eps^2$ samples we can get an estimate such that with probability
  at least $3/4$ it holds \(|\wt m_V - m_V| \leq \eps/d^{O(k)} \).  Using the
  standard process of repeating and taking the median estimate we amplify the
  success probability to $1-\delta$.  Since we want all the estimates of all
  the moments $V$ with $|V| \leq k$ to be accurate we choose $\delta =
  1/d^{O(k)}$ and by the union bound we obtain that with constant probability
  $|\wt m_V - m_V| \leq \eps/d^{O(k)}$ for all $V$ with $|V| \leq k$.  Now, for
  any tuple of parameters $(\wb \mu, \wb \Sigma , \wt S)$ we check whether the
  first $d^{O(k)}$ moments of the corresponding truncated Gaussian $\normal(\wb
  \mu, \wb \Sigma, \wt S)$ are in distance $\eps/d^{O(k)}$ of the estimates
  $\wt m_V$. If this is true for all the moments, then
  Lemma~\ref{lem:TVDPolynomial} implies that $\dtv{\normal(\vec \mu, \matr
  \Sigma, S)}{\normal(\wb \mu, \wb \Sigma, \wt S)} \leq \eps$.
\end{prevproofbig}

\begin{prevproofbig}{Lemma}{lem:chiSquaredExistence}
  Without loss of generality we may assume that $N_1 = \normal(\vec 0, \matr
  I)$ and $N_2 = \normal(\vec \mu, \matr \Lambda)$, where $\matr \Lambda$ is a
  diagonal matrix with elements $\lambda_i>0$.
  We define the normal $N = \normal(\vec 0, \matr R)$ with
  $r_i = \max(1, \lambda_i)$.
  We have
  \begin{align*}
    \dchi{N_2}{N} + 1
    &= \int \frac{\normal(\vec \mu, \matr \Lambda ; \vec x)^2}{\normal(\vec 0, \matr R; \vec x)}
    \d \vec x
    \\
    &= \frac{\sqrt{|\matr R|}}{(2 \pi)^{d/2} |\matr \Lambda| }
    \exp(-\vec \mu^T \matr \Lambda^{-1} \vec \mu)
    \underbrace{
      \int
      \exp\lp( \vec x^T \lp( \frac 1 2 \matr R^{-1} - \matr \Lambda^{-1} \rp)
      + 2 \vec \mu^T \matr \Lambda^{-1} \vec x \rp)
      \d \vec x
    }_{I}
  \end{align*}
  We have
  \begin{align*}
    I = \prod_{i=1}^d \int
    \exp\lp( x_i^2 \lp(\frac{1}{2 r_i} - \frac{1}{\lambda_i}\rp) + 2 \frac{\mu_i}{\lambda_i} x_i
    \rp)
    \d x_i
    =
    (2 \pi)^{d/2}
    \prod_{i=1}^d
    \frac{\exp\lp(\frac{2 r_i \mu_i^2}{2 r_i \lambda_i - \lambda_i^2} \rp)}
    {\sqrt{2/\lambda_i - 1/r_i}}
  \end{align*}
  Therefore,
  \begin{align*}
    \dchi{N_2}{N} + 1
    &\leq
    \prod_{i=1}^d \sqrt{\frac{r_i}{ 2 \lambda_i - \lambda_i^2/r_i}}
    \exp\lp(\frac{2 r_i \mu_i^2}{2 r_i \lambda_i - \lambda_i^2} \rp)
    \\
    &=
    \exp
    \lp(
    \sum_{i=1}^d \frac{1}{2} \log\lp(\frac{r_i}{2 \lambda_i - \lambda_i^2/r_i}\rp)
    +
    \frac{2 r_i \mu_i^2}{2 r_i \lambda_i - \lambda_i^2}
    \rp)
  \end{align*}
  Using the fact that $r_i = \max(1, \lambda_i)$ we have
  \begin{align*}
    \sum_{i=1}^d \log\lp(\frac{r_i}{2 \lambda_i - \lambda_i^2/r_i}\rp)
    =
    \sum_{i:\lambda_i < 1}  \log\lp(\frac{1}{2 \lambda_i - \lambda_i^2}\rp)
    \leq
    \sum_{i:\lambda_i < 1} \lp(\frac1 \lambda_i - 1\rp)^2
    \leq \norm{\matr \Lambda^{-1} - \matr I}_F^2,
  \end{align*}
  where we used the inequality $\log(1/(2 x - x^2)) \leq (1/x - 1)^2$ which holds for
  all $x \in (0,1)$.
  Moreover,
  \[
    \sum_{i=1}^d
    \frac{2 r_i \mu_i^2}{2 r_i \lambda_i - \lambda_i^2}
    =
    \sum_{i: \lambda\leq 1}
    \frac{2 \mu_i^2}{2 \lambda_i - \lambda_i^2}
    +
    \sum_{i: \lambda > 1}
    \frac{2 \mu_i^2}{\lambda_i}
    \leq
    \sum_{i=1}^d
    \frac{2 \mu_i^2}{\lambda_i}
    = 2 \norm{\matr \Lambda^{-1/2} \vec \mu}_2^2,
  \]
  where we used the inequality $1/(2 x - x^2) \leq 1/x$ which holds for all $x
  \in (0,1)$.  Combining the above we obtain
  \[
    \dchi{N_2}{N} \leq \exp\lp(\frac 12 \norm{\matr \Lambda^{-1/2} \vec \mu}_2 +
    2 \norm{\matr \Lambda^{-1} - \matr I}_F^2 \rp)
  \]
  Similarly, we compute
  \begin{align*}
    \dchi{N_1}{N} +1 &=
    \prod_{i=1}^d \sqrt{\frac{r_i}{2 - 1/r_i}}
    =
    \exp\lp(
    \frac 12 \sum_{i: \lambda_i > 1} \log\lp( \frac{\lambda_i}{2 - 1/\lambda_i} \rp)
    \rp)
    \\
                     &\leq
                     \exp\lp(
                     \frac 12 \sum_{i: \lambda_i > 1}
                     \lambda_i \lp(1 - \frac 1 \lambda_i \rp)^2
                     \rp)
                     \leq \exp\lp(
                     \frac12
                     \max(\norm{\matr \Lambda}_2,1) \norm{\matr \Lambda^{-1} - \matr I}_F^2
                     \rp)
  \end{align*}

\end{prevproofbig}

The following lemma gives a very rough bound on the maximum coefficient of
multivariate polynomials of affine transformations.
\begin{lemma}\label{lem:polynomialInfinityNormBound}
  Let $p(\vec x) = \sum_{V: |V|\leq k} c_V x^V$ be a multivariate polynomial of
  degree $k$.   Let $\matr A \in \R^{d \times d}$, $\vec b \in \R^d$.  Let
  $q(\vec x) = p(\matr A \vec x + \vec b)$.  Then
  \(
    \norm{q}_\infty
    \leq
    \norm{p}_\infty
    \binom{d+k}{k}
    \lp(\sqrt{d} \norm{\matr{A}}_2 + \norm{\vec{b}}_2 \rp)^k
    .
    \)
\end{lemma}
\begin{proof}
  We have that
  \[
    q(\vec x) =
    \sum_{V: |V| \leq k} c_V \prod_{i=1}^d \lp(\sum_{j=1}^d A_{ij} x_j + b_i\rp)^{v_i}
  \]
  Therefore,
  \begin{align*}
    \norm{q}_1
    &\leq
    \sum_{V: |V| \leq k} c_V \prod_{i=1}^d \lp(\sum_{j=1}^d |A_{ij}| + |b_i|\rp)^{v_i}
    \leq
    \sum_{V: |V| \leq k} c_V \prod_{i=1}^d \lp(\norm{\matr A}_\infty + \norm{\vec b}_\infty \rp)^{v_i}
    \\
    &=
    \sum_{V: |V| \leq k} c_V \lp(\norm{\matr A}_\infty + \norm{\vec b}_\infty \rp)^{|V|}
    \leq
    \norm{p}_\infty
    \binom{d+k}{k}
    \lp(\norm{\matr A}_\infty + \norm{\vec b}_\infty \rp)^k
    \\
    &\leq
    \norm{p}_\infty
    \binom{d+k}{k}
    \lp(\sqrt{d} \norm{\matr{A}}_2 + \norm{\vec{b}}_2 \rp)^k
  \end{align*}
\end{proof}
 
\end{document}